\documentclass[11pt, leqno]{amsart}
\usepackage{amsthm,amsmath,amssymb,amscd,verbatim,graphicx}
\DeclareMathOperator{\Coker}{Coker}

\DeclareMathOperator{\cox}{Cox}
\DeclareMathOperator{\Gal}{Gal}

\DeclareMathOperator{\GL}{GL}

\DeclareMathOperator{\Gr}{Gr}

\DeclareMathOperator{\h}{ht}
\DeclareMathOperator{\head}{head}
\DeclareMathOperator{\Hom}{Hom}
\DeclareMathOperator{\id}{id}
\DeclareMathOperator{\im}{Im}

\DeclareMathOperator{\Ind}{Ind}

\DeclareMathOperator{\Ker}{Ker}

\DeclareMathOperator{\supp}{supp}

\DeclareMathOperator{\tail}{tail}

\newcommand{\mA}{{\mathbb A}}

\newcommand{\mF}{{\mathbb F}}
\newcommand{\mG}{{\mathbb G}}

\newcommand{\mN}{{\mathbb N}}

\newcommand{\mP}{{\mathbb P}}
\newcommand{\mQ}{{\mathbb Q}}

\newcommand{\mZ}{{\mathbb Z}}

\newcommand{\mcO}{{\mathcal O}}

\newcommand{\D}{{\mathcal D}}

\renewcommand{\max}{{\rm max}}
\renewcommand{\min}{{\rm min}}

\renewcommand{\P}{{\mathcal P}}

\renewcommand{\to}{\longrightarrow}

\theoremstyle{plain}
\newtheorem{thm}{Theorem}[section]
\newtheorem*{thm*}{Theorem}
\newtheorem{lemma}[thm]{Lemma}
\newtheorem{coro}[thm]{Corollary}
\newtheorem*{coro*}{Corollary}
\newtheorem{prop}[thm]{Proposition}

\newtheorem{conj}[thm]{Conjecture}
\newtheorem*{conj*}{Conjecture}

\theoremstyle{definition}

\newtheorem{defn}[thm]{Definition}
\newtheorem{eg}[thm]{Example}
\newtheorem{egs}[thm]{Examples}
\newtheorem{rk}[thm]{Remark}
\newtheorem{rks}[thm]{Remarks}

\numberwithin{equation}{section}

\begin{document}
\title[The cohomology of Deligne-Lusztig varieties]{The cohomology of Deligne-Lusztig varieties for the general linear group \\ 
\vspace{0.5cm}{\rm \tiny Dedicated to Michael Rapoport on the occasion of his 65th birthday}  \\ }
\author[Sascha Orlik]{Sascha Orlik}
\address{Bergische Universit\"at Wuppertal\\
Fachbereich C - Mathematik und Naturwissenschaften \\
Gaussstr. 20\\
42119 Wuppertal\\ Germany.}
\email{orlik@math.uni-wuppertal.de}

\date{\today}


\begin{abstract}
We propose two inductive approaches for determining the cohomology of Deligne-Lusztig
varieties in the case of $G=\GL_n$. The first one uses Demazure compactifications and analyses the corresponding Mayer-Vietoris
spectral sequence. This allows us to give an inductive formula  for the Tate twist $-1$ contribution of the cohomology of a DL-variety.
The second approach relies on considering more generally DL-varieties  attached to  hypersquares in the Weyl group. Here we give  
explicit formulas for  the cohomology of height one elements.
\end{abstract}

\maketitle

\section{Introduction}

In 1976 Deligne and Lusztig \cite{DL} introduced certain locally closed subvarieties in flag varieties over finite fields
which are  of particular importance in the representation theory of finite groups of Lie type. They proved that their Euler-Poincar\'e characteristic
considered as a  virtual representation of the corresponding finite group detect all irreducible
representations. However, a description of the individual cohomology groups of Deligne-Lusztig varieties has been determined
since then only in a few special cases cf. \cite{L2, DMR,DM,Du}, where in contrast the intersection cohomology groups of their Zariski closures
are treated in \cite{L3}.
In this paper we propose two inductive approaches for determining all of them
in the case of $G=\GL_n$ (resp. for reductive groups of Dynkin type $A_{n-1}$). Although the key ideas work for other (split) reductive groups as well, 
we have decided to treat
here only the case of the general linear group since things are more concrete in this special situation.

For a split reductive group ${\bf G}$ defined  over $k=\mF_q$, let $X$ be the set of all Borel subgroups of ${\bf G}$.
Let $F: X \rightarrow X$ be the Frobenius map over $\mF_q$.
The Deligne-Lusztig variety associated to  an element $w\in W$ of the Weyl group is the locally closed subset of $X$ given by
\begin{equation*}
X(w)=\{x\in X\mid {\rm inv}(x,F(x))=w\}.
\end{equation*}
Here ${\rm inv}:X \times X \to W$ is the relative position map induced by the Bruhat lemma.
Then $X(w)$ is a smooth quasi-projective variety defined over $\mF_q$. It is naturally equipped with an action of $G={\bf G}(k)$
and has dimension equal to the length of $w$.
The $\ell$-adic cohomology with compact support $H^\ast_c(X(w)):=H_c^\ast(X(w),\overline{\mQ}_\ell)$ has therefore the structure of a
$G \times \Gal(\overline{k}/k)$-module.

Let ${\bf G}=\GL_n.$ In this paper, we make heavily use of
certain maps $\gamma : X_1 \to X(w')$ resp. $\delta :X_2 \to X(sw')$ introduced in \cite[Theorem 1.8]{DL} and implicitly further studied
in \cite{DMR}.
Here $w,w'\in W$  and $s$ is a simple reflection with $w=sw's$ and $\ell(w)=\ell(w')+2.$ Further  $X_1$ is a closed
subset of $X(w)$ and $X_2$ denotes its open complement. It is proved in loc.cit that $\gamma$ is a $\mA^1$-bundle whereas
$\delta$ is a $\mG_m$-bundle. Here  we consider instead of the map $\delta$ its look-alike $X_2 \to X(w's).$  The above
maps extend to $\mP^1$-bundles $X_2 \cup X(sw') \cup X(w's) \to X(w's)$ and $X_1 \cup X(w') \to X(w')$ which glue in turn to a $\mP^1$-bundle
$$\gamma:  X(Q)  \to X(w's) \cup X(w')$$
where $X(Q)=X(w) \cup X(sw') \cup X(w's) \cup X(w').$
Here $\gamma_{|X(w's)\cup X(w')}=\id$ whereas the restriction of $\gamma$ to $Z:=X(w) \cup X(sw')$ is a
$\mA^1$-bundle over the base $Z':=X(w's)\cup X(w').$ In particular, we deduce that
$$H^i_c(X(Q))=H^{i}_c(Z')\oplus H^{i-2}_c(Z')(-1)$$ for all integers  $i\geq 2$, which has been already known since \cite{DMR}. 

The quadruple $Q=\{w',sw',w's,w\}\subset W$ is a square in the sense of \cite{BGG}. The notion of a square appears in the theory of
BGG-resolutions of finite-dimensional Lie algebra representations. It seems to be also useful
in the study of the cohomology of Deligne-Lusztig varieties. We consider more generally hypersquares in $W$ and even in the
monoid $F^+$ which is freely generated by the subset $S$ of simple reflections in $W$. In fact we work more generally with DL-varieties and their
Demazure compactifications attached to
elements  in  ${F^+}$ in the spirit of \cite{DMR}.
More precisely, let $w=s_{i_1}\cdots s_{i_r}$ be a fixed reduced decomposition of $w\in W$ and let $\overline{X}(w)$ be the associated
Demazure compactification of $X(w)$. This variety is equipped with a compatible action of $G.$
We consider the closed complement of $X(w)$ in $\overline{X}(w)$  which is - as already observed in \cite{DL} - a union of smooth equivariant divisors.
We analyse the resulting spectral sequence converging to the cohomology of $X(w).$
The crucial point is that the intersection of these divisors is again a compactification of a DL-variety attached to some
subexpression of $s_{i_1}\cdots s_{i_r} \in F^+$.
Concretely the  spectral sequence has the shape
$$E_1^{p,q}=\bigoplus_{v \preceq w, \ell(v)=\ell(w)-p} H^q(\overline{X}(v))
\Longrightarrow H^{p+q}_c(X(w)).$$
Another feature  is that if $w=sw's\in F^+$, then  $\overline{X}(w)$ is a
$\mP^1$-bundle over  $\overline{X}(w's)$. This comes about from the fact that
$\overline{X}(w)$ is paved by DL-varieties attached to  squares of the special type as above. So by induction on the length of $w's$ we know the cohomology of the
compactification $\overline{X}(w).$ Of course not every element $w$ in $F^+$ has the shape $w=sw's$, but by using a result of \cite{GKP}, every element can be transformed into such an element by
applying the usual Weyl group relations and a cyclic shift operator. We study henceforth the effect on the cohomology by these operations. The start of induction is given by
elements of minimal length in their conjugacy classes, i.e. by Coxeter elements in  Levi subgroups of $G.$
This is one reason why we deal only with reductive groups of Dynkin type $A_{n-1}.$  In this case the Demazure compactification of the standard Coxeter element
can be considered as one of the Drinfeld halfspace $\Omega^n=\mP^{n-1} \setminus \bigcup\nolimits_{H / \mF_q} H$
(complement of all $\mF_q$-rational hyperplanes in the projective space of lines in $V=\mF^n$), cf. \cite{Dr}, and may be realised as a sequence of blow ups as it comes up in the arithmetic
theory of $\Omega^{n}$ over a local field \cite{Ge,GK,I}.

The following results are already known to the experts.

\begin{thm*} Let $G=\GL_n.$

\noindent i) For all $w\in F^+$, we have  $H^i(\overline{X}(w))=0$ for odd $i.$

\noindent ii) Let $w=sw's$ with $s\in S$ and $w'\in F^+$. Then there are decompositions
$H^i(\overline{X}(w))=H^i(\overline{X}(w's)) \oplus H^{i-2}(\overline{X}(w's))(-1)$ $=H^i(\overline{X}(sw')) \oplus H^{i-2}(\overline{X}(sw'))(-1).$

\noindent iii) The action of the Frobenius on $H^i(\overline{X}(w))$ and $H^i_c(X(w))$ is semi-simple for all $w\in F^+$ and for all $i\geq 0.$

\noindent iv) For all $i \geq 0$, the cycle map $A^i(\overline{X}(w))_{\overline{\mQ}_\ell}\to H^{2i}(\overline{X}(w))$ is an isomorphism
$\bigl($where $A^i(\overline{X}(w))$ is the Chow group of $\overline{X}(w) \mbox{ in degree $i$} \bigr)$.
\end{thm*}

Whereas  part ii) of this theorem is already contained more generally in \cite[Prop. 3.2.3]{DMR}, it was pointed out to me that
part i) and iii) and iv) can be deduced  from \cite{L4}. Our proofs differ from loc.cit. 

It turns out that the cohomology of the varieties $\overline{X}(w)$ is similar to the classical situation of Schubert varieties.
Indeed, let $\prec$ be the Bruhat order on $F^+$. For a parabolic subgroup  $P\subset G$, let $i^G_P=\Ind^G_P(\overline{\mathbb Q_\ell})$
be the induced representation of the trivial one.  By part ii) of the above theorem and by  studying the effect of the usual Weyl group
relations on the Demazure  compactifications of DL-varieties we are able to deduce the next statement.

\begin{thm*}
 Let $w\in F^+.$ Then the cohomology of $\overline{X}(w)$ can be written as
$$H^\ast(\overline{X}(w))=\bigoplus_{z\preceq w} i^G_{P^w_z}(-\ell(z))[-2\ell(z)]$$
for certain std parabolic subgroups $P^w_z\subset G.$
\end{thm*}

The gradings are not canonical as there are in general plenty of choices. However, they behave functorial for
appropriate choices. 

\begin{thm*}
 Let $w,v\in F^+$ with $v \prec w.$  Then there are gradings 
 $H^{2i}(\overline{X}(w))=\bigoplus_{z\preceq w \atop \ell(z)=i} i^G_{P^w_z}$
 and $H^{2i}(\overline{X}(v))=\bigoplus_{z\preceq v \atop \ell(z)=i} i^G_{P^v_z}$
 such that the natural homomorphism $H^{2i}(\overline{X}(w)) \to 
 H^{2i}(\overline{X}(v))$ is graded. 
 Moreover, the maps $i^G_{P^w_z} \to i^G_{P^v_z}$ (which are induced by the double cosets of
$1$ in $W_{P^w_z}\setminus  W/W_{P^{v}_{z}}$  via Frobenius reciprocity) are injective or surjective  for all $z \preceq v.$
\end{thm*}

In a next step we analyse the above spectral sequence attached to $\overline{X}(w)$ and its divisors. By weight reasons the
spectral sequence degenerates in $E_1$
and we believe  that it can be evaluated via the following approach.

\begin{conj*}
Let $w\in F^+$ and fix an integer $i\geq 0.$ For $v\preceq w$, there are gradings
$H^{2i}(\overline{X}(v))=\bigoplus\limits_{z\preceq v \atop \ell(z)=i}i^G_{P^v_z}(-\ell(z))$ such that
the complex
\begin{equation*}
E_1^{\bullet,2i}: H^{2i}(\overline{X}(w)) \to  \bigoplus_{\genfrac{}{}{0pt}{2}{v\prec w}{\ell(v)=\ell(w)-1}} H^{2i}(\overline{X}(v)) \to \bigoplus_{\genfrac{}{}{0pt}{2}{v\prec w}{\ell(v)=\ell(w)-2}} H^{2i}(\overline{X}(v)) \to
\cdots \to  H^{2i}(\overline{X}(e))
\end{equation*}
is quasi-isomorphic to a direct sum $\bigoplus_{z\preceq w \atop \ell(z)=i} H(\,\cdot\,)_z$ of complexes of the shape
\begin{equation*}
H(\,\cdot\,)_z: i^G_{P^w_{z}} \rightarrow \bigoplus_{\genfrac{}{}{0pt}{2}{z \preceq v\preceq w}{\ell(v)=\ell(w)-1}}i^G_{P^v_{z}} \rightarrow
\bigoplus_{\genfrac{}{}{0pt}{2}{z\preceq v\preceq w}{\ell(v)=\ell(w)-2}}i^G_{P^v_{z}}
\rightarrow \dots \rightarrow i^G_{P^e_z}
\end{equation*}
(Here the maps $i^G_{P^v_z} \to i^G_{P^{v'}_{z}}$ in the complex are induced - up to sign - by the double cosets of
$1$ in $W_{P^v_z}\setminus  W/W_{P^{v'}_{z}}$  via Frobenius reciprocity. Further $i^G_{P^v_{z}}=(0)$ if $z\npreceq v.$).
\end{conj*}

A similar statement holds true for period domains over finite fields \cite{DOR}. 
We are able to prove the conjecture is some cases.
Additionally, we prove how to reduce the issue to the case where $w$ is again of the form $w=sw's.$

\begin{thm*}
 i) The conjecture is true for  Coxeter elements $w$.

ii) If $w\in F^+$ is arbitrary, then the conjecture is true for $i\in \{0,1,\ell(w)-1,\ell(w)\}.$
\end{thm*}

Of course the cases $i=0$ and $i=\ell(w)$ are trivial. As a consequence we derive an inductive formula for the Tate twist $-1$-contribution $H^\ast_c(X(w))\langle -1 \rangle$  of the cohomology of
a DL-variety $X(w)$.  For a parabolic
subgroup $P$ of $G$, let $v^G_P$ be the corresponding generalized Steinberg representation.

\begin{coro*}
 Let $w=sw's\in F^+$ with $\h(sw')\geq 1.$ Then

$$H^\ast_c(X(w))\langle -1 \rangle = \left\{\begin{array}{cc}
                                                     H^\ast_c(X(sw'))\langle -1 \rangle[-1] & \mbox{ if } s\in \supp(w') \\ \\
  H^\ast_c(X(sw'))\langle -1 \rangle[-1] \oplus v^G_{P(s)}(-1)[-\ell(w)] & \mbox{ if } s \notin \supp(w')
                                                    \end{array} \right. .$$
\end{coro*}

\medskip

\noindent Here $\supp(w')$ denotes the set of simple reflections appearing in $w'$ whereas  $P(s)$ is the parabolic subgroup
of $G$ generated by $B$ and $s.$ Moreover, $\h$ is the height function on $F^+.$
It has the property that
$\h(w)=\h(w')+1$ if $w=sw's$ as above. Further $\h(w)=0$ if $w$ is minimal length in its conjugacy class.
The start of the inductive formula is hence given by height one elements. Here we are even able to determine all cohomology groups.
For a partition $\lambda$  of $n$, let $j_\lambda$ be the corresponding irreducible $G$-representation.

\begin{thm*}  Let $w=sw's\in W$ with  $\h(w)=1$ and $\supp(w)=S.$ 

i) If $\h(sw')=0$, then we have for $i\in \mathbb N$, with $\ell(w)<i<2\ell(w)-1,$
 $$H^{i}_c(X(w'))= \big(H^{i-2}_c(X(w')) - j_{(i+1-n,1,\ldots,1)}(n-i)\big)(-1) -
j_{(i+2-n,1\ldots,1)}(n-i-1).$$

Furthermore,
$$H^i_c(X(w)) =\left\{\begin{array}{lc}
          v^G_B \oplus \big(v^G_{P(s)} - j_{(2,1\ldots,1)}\big)(-1)  &;\;\;\; i=\ell(w) \\ \\
          0 &; \;\;\;  i=2\ell(w)-1 \\ \\
          i^G_G(-\ell(w)) &; \;\;\; i=2\ell(w)
                 \end{array}\right. .$$

\smallskip
ii) If $\h(sw')=1$, then we have
$$H^i_c(X(w))=H^{i-2}_c(X(w' s)\cup X(w'))(-1) \oplus H^{i-1}_c(X(s w'))$$ for all $i\neq 2\ell(w)-1, 2\ell(w)-2$
and  $H^{2\ell(w)-1}_c(X(w))=H^{2\ell(w)-2}_c(X(w))=0.$
\end{thm*}

Moreover, we give an inductive recipe for the cohomology in degree $2\ell(w)-2$ of a DL-variety $X(w).$
Here the case of $\h(w')=0$ is treated by the above results.

\begin{coro*}
  Let $w=sw's\in F^+$ with $\h(w')\geq 1$ and $\supp(w)=S.$
 Then $$H^{2\ell(w)-2}_c(X(w))=H^{2\ell(w's)-2}_c(X(w's))(-1) \oplus (i^G_{P(w')}-i^G_G)(-\ell(w)+1).$$
\end{coro*}
\noindent Here $P(w')$ is the parabolic subgroup of $G$ which is generated by $B$ and $\supp(w')\subset S.$


\medskip
The proof of the two last formulas bases on the second approach for determining the cohomology of DL-varieties. This alternative proposal is
pursued for arbitrary elements of the Weyl group in the appendix.
Whereas the previous version
uses Demazure resolutions, i.e., DL-varieties attached  to maximal hypersquares, this time the procedure goes  the other way round in the sense
that the considered hypersquare grows. In fact, for determining the cohomology of $X(w)$, we study first the map
$$H^i_c(X(w)\cup X(sw')) \to H^i_c(X(sw'))$$ induced by the closed embedding $X(sw') \hookrightarrow X(w) \cup X(sw').$
By induction on the length  the cohomology of the RHS is known. Further we give a conjecture on the structure of this map.
Hence it suffices to know the cohomology of the edge $X(w)\cup X(sw')$ which is - as explained above - induced  by the cohomology of the
edge $X(w's)\cup X(w')$.
Thus we have transferred the question of determining the cohomology of the vertex $X(w)$
to the knowledge of the cohomology of the edge $X(w's)\cup X(w'),$ but which has has smaller length, i.e. $\ell(w's)<\ell(w).$ In the next step one reduces similar the case of an edge to the case of a square etc.
This second approach is a little bit vague as it depends among other things on some more conjectures.
Nevertheless, I have decided to include it into this paper for natural reasons.

In Section 2 we review some facts on unipotent representations of $\GL_n(\mF_q)$. In Section 3 we consider DL-varieties and study explicitly
the case of a Coxeter element. In Section 4 we deal with squares and their associated DL-varieties. Here we treat in particular
the case of the special square $Q=\{sw's,w's,sw',w'\}$ and prove that the map $X(Q)\to Z'$ is a $\mP^1$-bundle. In section 5 we determine the cohomology of DL-varieties for
$G=\GL_4$ and in general for height 1 elements in $W.$
In Section 6 we generalize the ideas of the foregoing section to hypersquares. Section 7 deals with the cohomology of Demazure Varieties.
In Section 8 we reconsider the spectral sequence and discuss the conjecture mentioned above. Finally in Section 9 we illustrate the Conjecture
resp. the Theorem in the case of $G=\GL_4.$

\tableofcontents


{\it Notation:}

- Let $k=\mF_q$ be a finite field of cardinality $q$  with fixed algebraic closure $\overline{k}$ and  absolute Galois group  $\Gamma:={\rm Gal}(\overline{k}/k)$.

- We denote for any $\Gamma$-module $V$ and any integer $i$, by $V\langle i \rangle$ the eigenspace
of the arithmetic  Frobenius  with eigenvalues of absolute value $q^i.$

- Let ${\bf G_0}=\GL_n$ be the
general linear group over $k$.  Denote by ${\bf G} = {\bf G_0}\times_{\mF_q} \mF$ the base change
to the algebraic closure. Let ${\bf T} \subset {\bf B} \subset {\bf G}$  be the  diagonal torus  resp. the  Borel subgroup of upper triangular matrices.
Let $W\cong S_n$ be the Weyl group of ${\bf G}$ and $S$ be the subset of simple reflections. For a subset $I\subset S$, we denote as usual by
$W_I$ be the subgroup of $W$ generated by $I.$

- We also use the cyclic notation for elements in the symmetric group. Hence the expression $w=(i_1,i_2,\ldots,i_r)$ denotes the permutation with $w(i_j)=i_{j+1}$
for $j=1,\ldots,r-1$, $w(i_r)=i_1$ and $w(i)=i$ for all $i\not\in\{i_1,\ldots,i_r\}.$

- For a vector space $V$ of dimension $n$ over $k$ and any integer $1\leq i \leq n$, we let ${\rm Gr}_i(V)$ be the Grassmannian parametrizing
subspaces of dimension $i.$

- We denote by $1$ or $e$ the identity in any group or monoid.

\smallskip

{\it Acknowledgements:} I want to thank  Olivier Dudas for his numerous remarks  on this paper.
I am grateful for the invitation to Paris and all the discussions with Fran\c{c}ois Digne   and Jean Michel. 
Finally I thank Roland Huber,  Michael Rapoport and Markus Reineke for their support.

\vspace{0.5cm}

\section{Unipotent Representations of $\GL_n$}

We start with a discussion on representations of $G={\bf GL_n}(k)$ which are called unipotent.
These kind of objects appear in the cohomology of Deligne-Lusztig varieties.
In this section, all representations will be in vector spaces over a fixed algebraically closed field $C$ of characteristic zero.

Recall that a standard parabolic subgroup (std psgp)  ${\bf P} \subset {\bf G}$ is a parabolic subgroup of ${\bf G}$
with ${\bf B} \subset {\bf P}.$
The set of all std psgp is in bijection with the set
$$\D=\D(n)=\bigl\{(n_1,\ldots,n_r) \in \mN^r \mid n_1 + \cdots + n_r=n,\, r\in \mN \bigr\}$$ of decompositions of $n.$
For a decomposition  $d=(n_1,\ldots,n_r)\in \D(n)$,
we let ${\bf P}_{d}$ be the corresponding std psgp with Levi subgroup ${\bf M_{P_d}}=\prod_i {\bf GL}_{n_i}.$
There is a partial order $\leq $ on $\D(n)$ defined by
$$d'\leq d \;\mbox{ if and only if }\; {\bf P}_{d'} \subset {\bf P}_{d}.$$ If ${\bf P}={\bf P}_d$ is a std psgp to a decomposition $d\in \D$,
then any $d' \in \D$ induces a std psgp ${\bf Q}_{d'}={\bf M_{P}} \cap {\bf P}_{d'}$ of ${\bf M_{P}}$ and this assignment  gives a bijection between
the sets
$$\{d' \in \D \mid d' \leq d\} \stackrel{\sim}{\longrightarrow} \{\mbox { std psgps of } {\bf M_{P}}\}.$$
In the sequel we call the finite group $P={\bf P}(k)$ attached to a std psgp ${\bf P}$ of $\bf G$ standard parabolic of $G$, as well.

Recall that a parabolic subgroup of $W$ is by definition  the Weyl group of the Levi component ${\bf M_P}$
of some std psgp ${\bf P}$. This defines   a one-to-one correspondence between the std psgp ${\bf P}$ of ${\bf G}$
and the parabolic subgroups $W_P$ of $W$. If ${\bf P}={\bf P}_d$ with $d=(n_1,\ldots,n_r)\in \D$, then $W_P\cong S_{n_1} \times \cdots \times S_{n_r}.$

We denote by $\P=\P(n)$ the set of partitions of $n.$ There is a map $\lambda: \D \to \P$  by
ordering a decomposition $(n_1,\ldots,n_r)$ in decreasing size.
Let $\mu=(\mu_1,\mu_2,\ldots, \mu_n)$ and $\mu'=(\mu_1',\mu_2',\ldots, \mu_n')$ be two partitions of $n$.
Consider the order $\leq$  on $\P$ defined by $\mu' \leq \mu$ if for all $r=1,\ldots,n,$ we have
$$ \sum _{i=1}^r \mu'_i \leq \sum_{i=1}^r \mu_i.$$ Then the map $\lambda$ is compatible with both orders in the sense that if
$d' \leq d$ then $\lambda(d') \leq \lambda(d)$.
If $P=P_d$ is a std psgp, then we also write $\lambda(P)$ for $\lambda(d).$
On the other hand,  two parabolic subgroups are called associate, if their Levi components are conjugate under $G$.
If $P=P_{d_1}$ and $P'=P_{d_2}$ with associated decompositions $d_1=(n_1,\ldots,n_r)$ and $d_2=(n_1',\ldots,n_{r'}')$ of $n,$
then $P$ and $P'$ are associate if and only if $\lambda(d_1)=\lambda(d_2).$

Let $P$ be a std psgp of $G$. We denote by $$i^G_P=\Ind^G_P {\bf 1}$$
the induced representation of the trivial representation ${\bf 1}$ of $P.$ It coincides with the set $C[G/P]$ of $C$-valued functions
on $G/P$ equipped with the natural action.
Let $\hat G(i^G_B)$ be the set of isomorphism classes of irreducible subobjects of  $i^G_B$.
We remind the reader at the following properties of the representations $i^G_P$, cf.  \cite[Thm. 3.2.1]{DOR}.

\begin{thm}\label{thmirredDarstG}
(i)  $i^G_P$ is equivalent to $i^G_{P'}$ if and only if $P$ is associate to $P'.$

\noindent (ii) $i^G_P$ contains a unique irreducible subrepresentation $j^G_P$ which occurs with multiplicity one and such that
$$ \Hom_G(j^G_P,i^G_{P'}) \neq (0) \Leftrightarrow \lambda(P') \leq \lambda(P).$$

\noindent (iii) We set for every $\mu\in \P,$
$$j_\mu=j^G_{P_\mu}$$
where $P_\mu$ is any std psgp with $\lambda(P_\mu)=\mu.$ Then
$\{j_\mu \mid \mu\in \P \}$ is a set of representatives for $\hat G(i^G_B)$.
\end{thm}

\begin{rks}\label{connection_symmetric group}
i) The proof of the above theorem makes use of  the representation theory of the symmetric group and bases
on the following result of Howe \cite{Ho}.
Let $\hat{W}$ be the set of isomorphism classes of irreducible representations of $W$.
Analogously to the definition of $i^G_P$, we set
$i^W_{W'}:=\Ind^W_{W'} {\bf 1}$ for any subgroup $W'$ of $W.$
There exists a unique bijection
$$ \alpha: \hat{G}(i^G_B) \to \hat{W}$$
characterized  by the following property. An irreducible representation $\sigma \in \hat{G}(i^G_B)$ occurs in $i^G_{P}$ if and only
if $\alpha(\sigma)$ occurs in $i^W_{W_P}$. Furthermore
$$ \dim \Hom_G(\sigma, i^G_{P}) = \dim \Hom_W(\alpha(\sigma), i^W_{W_P}).$$
In particular, we get by Frobenius reciprocity
$$\Hom_G(i^G_{Q}, i^G_{P}) \cong  \Hom_W(i^W_{W_Q}, i^W_{W_P})=\Hom_{W_Q}({\bf 1}, i^W_{W_P})=C[W_Q\backslash W / W_P].$$

ii) Let $P\subset Q$ be two standard parabolic subgroups. Then there is a natural inclusion of  $G$-representations
$i^G_Q \subset i^G_P$ which corresponds just to the double coset of $e\in W$ in $W_Q\backslash W / W_P$.
On the other hand, the map $i^G_P \to i^G_Q$ induced by $e$ is given by $\delta_{gP} \mapsto \sum_{gP\subset gQ} \delta_{gQ}$
where $\delta_{gP}\in C[G/P]$ is the Kronecker function with respect to $gP\in G/P.$ In general, for two arbitrary
standard parabolic subgroups $P,Q$ of $G$, the map $i^G_P \to i^G_Q$ induced by $e\in W$ 
is injective (surjective) if and only if $\lambda(P) \geq \lambda(Q)$ $(\lambda(P) \leq \lambda(Q))$. 
In fact, this property which is a refinement of Theorem \ref{thmirredDarstG} ii) was observed by Liebler and Vitale \cite{LV}
and reproved by Hazewinkel and Vorst \cite{HV} for the symmetric group.
In particular, one can speak
of the maximum $\max\{i^G_P,i^G_Q\}$ of these representations in this sense.
\end{rks}

\begin{coro}\label{grading_unique}
Let $V$ and $W$ be two finite-dimensional isomorphic $G$-representations. Let $V=\bigoplus_{i=1}^s i^G_{P_i}$, 
$W=\bigoplus_{i=1}^t i^G_{Q_i}$ be decompositions into induced representations. 
Then $s=t$  and after a possible permutation the parabolic subgroups $P_i$ and $Q_i$ are associate for all $i=1,\ldots,s.$
\end{coro}

\begin{proof}
The trivial $G$-representation $i^G_G={\bf 1}$ appears with multiplicity one in each induced  representation $i^G_P$. Hence $s=t.$

 The remaining  proof is by induction on $s.$ Let $s=1$. Then the claim follows by the theorem above.
 Let $s>1$. Then again we use this theorem  to deduce that there are associate parabolic subgroups $P_i,Q_j$ in these decompositions. 
 By dividing out the summand $i^G_{P_i}\cong i^G_{Q_j}$ on both sides, respectively, we get two decompositions
 of isomorphic representations of the type above. Now the claim follows by induction on $s$.
 \end{proof}

\begin{defn}
The {\it generalized Steinberg representation}\index{generalized!Steinberg representation} associated to a std psgp $P=P_d$ is the quotient
$$v^G_P = i^G_P / \sum_{Q\supsetneq P} i^G_Q=i^G_{P_d} / \sum_{d'>d} i^G_{P_{d'}}.$$
\end{defn}

For $P=G,$ we have $v^G_G ={\bf 1}$
whereas if $B=P$ then we get the ordinary Steinberg representation $v^G_B$ which is irreducible.
In general, the generalized Steinberg representations $v^G_P$ are not  irreducible. More precisely, we have the following
criterion.

\begin{prop}
Let $d\in \D$. Then $v^G_{P_{d}}$ is irreducible 
if and only if $d=(k,1,\ldots,1)$ for some $k$ with $1\leq k \leq n.$ 
\end{prop}

\begin{proof}
By Remark \ref{connection_symmetric group} it suffices to show that the corresponding claim is correct
for the attached Weyl group representation $v^W_{W_d}:=i^W_{W_d}/\sum_{d'>d} i^W_{W_{d'}}$.

First let $d=(k,1,\ldots,1)$ for some $k$ with $1\leq k \leq n.$ By induction hypothesis  the $n-1$-tuple $\tilde{d}=(k,1,\ldots,1)\in \mZ^{n-1}$ which is 
induced by
deleting the last entry gives rise to an irreducible $S_{n-1}$-representation $V.$ By Pieri's formula \cite{FH} we have a decomposition
$\Ind^{S_n}_{S_{n-1}}(V)=\sum_{\lambda} V_\lambda$ where $\lambda\in \P$ ranges over all partitions obtained by adding the 
integer 1 to the $n-1$-tuple $\tilde{d}.$ Hence for $k\geq 2$, there are exactly 3 irreducible representations appearing in this sum. Otherwise there are 2 of them. In any case
the irreducible representation $j_d$ is one of them. On the other hand, we have 
$\Ind^{S_n}_{S_{n-1}}(V)=i^W_{W_d}/\sum_{d'>d \atop d'\neq (k,1,\ldots,1,2)} i^W_{W_{d'}}.$ The missing contribution
$i^W_{W_{(k,1,\ldots,1,2)}}$ covers the remaining irreducible representations. Thus $v^W_{W_d}$ and therefore 
$v^G_{P_d}=j_d$ are irreducible.

Suppose now that $d$ is not of the shape above. We shall see that the representation $v^W_{W_d}$  is reducible. In a first step
we may assume that $d$ is a partition  since $v^W_{W_{\lambda(d)}} \subset v^W_{W_d}.$ Then we consider
the partition $(d_1+1,d_2-1,d_3,\ldots, d_n).$ By Remark \ref{connection_symmetric group} the irreducible representation
attached to $(d_1+1,d_2-1,d_3,\ldots, d_n)$ appears in $v^W_{W_d}.$ Hence the latter one is not irreducible.
\end{proof}

\medskip
Let us recall some further properties of the representations $i^G_P.$ Let $P=P_{d} \subset G$ for some $d\in \D$. Let $d' \leq d$ and consider
the std psgp $Q_{d'}\subset M_P$ of $M_P$.
We consider the induced representation $i^{M_P}_{Q_{d'}}$ as a $P$-module via the trivial action of the unipotent radical of $P.$
Then
$$\Ind^G_P(i^{M_P}_{Q_{d'}})=i^G_{P_{d'}}.$$ Since $\Ind^G_P$ is an exact functor we get for any
$d''\in \D$ with $d' \leq d'' \leq d$, the identity
$\Ind^G_P(i^{M_P}_{Q_{d'}}/i^{M_P}_{Q_{d''}})=i^G_{P_{d'}}/i^G_{P_{d''}}.$
In particular, we conclude that
$$\Ind^G_P(v^{M_P}_{Q_{d'}})=i^G_{P_{d'}}/ \sum_{\{d'' \in \D \mid d' < d'' \leq d \}}  i^G_{P_{d''}}. $$

Consider the special situation where $d=(n_1,n_2)\in \D$. Then $P=P_{(n_1,n_2)}$ and  $M_P=M_1 \times M_2$ with $M_1=\GL_{n_1}$, $M_2=\GL_{n_2}.$
Let for $i=1,2$, $d_i \in \D(n_i)$ be a decomposition of $n_i$ and consider the corresponding std psgps $P_{d_i}$ of $M_i$.
Denote by $(d_1,d_2)\in \D(n)$  the glued decomposition of $n$.

\begin{lemma} We have  the identity
 \begin{equation}\label{identity_induced_Steinberg}
\Ind^G_{P_{(n_1,n_2)}}(v^{M_1}_{P_{d_1}} \boxtimes v^{M_2}_{P_{d_2}})= i^G_{P_{(d_1,d_2)}}/ \sum_{d_1' > d_1} i^G_{P_{(d_1',d_2)}} + \sum_{d_2' > d_2} i^G_{P_{(d_1,d_2')}}.
\end{equation}
\end{lemma}

\begin{proof}
Since $$i^{M_1}_{P_{d_1}}\boxtimes i^{M_2}_{P_{d_2}}=
i^{M_1\times M_2}_{P_{d_1} \times P_{d_2}}$$ we get
$$\Ind^G_{P_{(n_1,n_2)}}(i^{M_1}_{P_{d_1}} \boxtimes i^{M_2}_{P_{d_2}})= i^G_{P_{(d_1,d_2)}}.$$
Then the identity above follows by applying the exact functor $\Ind^G_{P_{(n_1,n_2)}}$ to the exact sequence
$$0 \to  \sum_{d_1' > d_1} i^{M_1}_{P_{d_1'}} \boxtimes i^{M_2}_{P_{d_2}} + \sum_{d_2' > d_2}i^{M_1}_{P_{d_1}} \boxtimes
i^{M_2}_{P_{d_2'}} \to i^{M_1}_{P_{d_1}} \boxtimes i^{M_2}_{P_{d_2}} \to v^{M_1}_{P_{d_1}} \boxtimes v^{M_2}_{P_{d_2}} \to 0.$$
\end{proof}

For the next property of generalized Steinberg representations, we refer to \cite{Le} (resp. to \cite{DOR} for a detailed discussion 
on this complex). Here we set for any decomposition $d=(n_1,\ldots,n_r) \in \D$,
$r(d)=r.$
\begin{prop}\label{IIIaltsum}
Let $P=P_d$, where $d\in \D.$ Then there is an acyclic resolution of $v^G_P$ by $G$-modules,
\begin{equation}\label{IIIcomplex}
0 \rightarrow  i^G_G \rightarrow \bigoplus_{\genfrac{}{}{0pt}{2}{d'\geq d}{r(d')=2}}i^G_{P_{d'}} \rightarrow
\bigoplus_{\genfrac{}{}{0pt}{2}{d'\geq d}{r(d')=3}}i^G_{P_{d'}}
\rightarrow \dots \rightarrow \bigoplus_{\genfrac{}{}{0pt}{2}{d'\geq d}{r(d')=r(d)-1}}i^G_{P_{d'}}\rightarrow i^G_{P_d} \rightarrow
v^G_{P}\rightarrow 0.
\end{equation}
\end{prop}

\begin{rk}
The prove of Proposition \ref{IIIaltsum}  relies on some simplicial arguments as follows, cf. loc.cit.
Let $d,d' \in \D$. Then we have
\begin{equation}\label{IIIP-inters}
i^G_{P_d}\; \cap\; i^G_{P_{d'}} = i^G_{P_{d \vee d'}}
\end{equation}
for some $d\vee d' \in \D.$
Further for all $d_1,\ldots,d_r \in \D$, we have
\begin{equation}\label{IIIsum_intersection}
i^G_{P_d}\; \cap\; (i^G_{P_{d_1}} + i^G_{P_{d_2}} + \cdots + i^G_{P_{d_r}})  = (i^G_{P_d}  \cap
i^G_{P_{d_1}}) \; +\; \cdots +  (i^G_{P_d}  \cap i^G_{P_{d_r}}).
\end{equation}
\end{rk}

\medskip
We reinterpret the complex (\ref{IIIcomplex}) as follows. Let ${F^+}$ be the monoid which is freely generated by the subset $S\subset W$ of simple 
reflections.
Denote by $$\gamma: F^+\to W$$
the natural map. For $w=s_{i_1}\cdots s_{i_r} \in {F^+},$ let $\ell(w):=r$ be the length and
$$\supp(w):=\{s_{i_1},\ldots,s_{i_r}\}$$ its support.
Any subword $v$  which is induced by erasing factors in $w$ gives by definition
rise to an element which is shorter with respect to the Bruhat ordering $\preceq$  on ${F^+}$, cf. also \cite{Hu}.
Note that this ordering is not compatible with the usual one $\leq$ on $W$ via $\gamma$.

For $w\in F^+,$ let $I(w)\subset S$ be  a minimal subset such
that $w$ is contained in the submonoid generated by $I(w).$ Let
\begin{equation}\label{psgp}
P(w)=P_{I(w)}\subset G
\end{equation}
be the std parabolic subgroup
generated by $B$ and $I(w).$ Alternatively, let $d(w)\in \D$ be the decomposition which corresponds to the subset $\supp(w)\subset S$
under the  natural bijection $\D\stackrel{\sim}{\to} S$. Then $P(w)=P_{d(w)}.$

We may define for $w\in F^+$ the following complex where the differentials are defined similar as above:
\begin{equation}\label{complex_0_F}
C^\bullet_w: 0 \rightarrow  i^G_{P(w)} \rightarrow \bigoplus_{\genfrac{}{}{0pt}{2}{v\preceq w}{\ell(v)=\ell(w)-1}}i^G_{P(v)} \rightarrow
\bigoplus_{\genfrac{}{}{0pt}{2}{v\preceq w}{\ell(v)=\ell(w)-2}}i^G_{P(v)}
\rightarrow \dots \rightarrow \bigoplus_{\genfrac{}{}{0pt}{2}{v \preceq w}{\ell(v)=1}}i^G_{P(v)}\rightarrow i^G_{P(e)} \rightarrow 0.
\end{equation}

\medskip

\begin{eg}
Let $w=\cox_n=s_1s_2\ldots s_{n-1}\in F^+$ be the standard Coxeter element. Then the complex $C^\bullet_w$  coincides
- up to augmentation in the Steinberg representation $v^G_B$ -  with the complex (\ref{IIIcomplex}) where $d=d(e).$
\end{eg}

\begin{defn}
 Let $w\in F^+.$ Then we say that $w$ has full support if $\supp(w)=S.$
\end{defn}

\begin{prop}\label{qi_cox}
Let $w \in F^+$ be of  full support. Then the complex $C^\bullet_w$ is quasi-isomorphic to $C^\bullet_{\cox}.$
\end{prop}

\begin{proof}
 We may suppose that $w$ is not a Coxeter element. Hence there exists a simple reflection $s\in S$ which appears at least twice in $w$.
Write $w=w_1sw'sw_2$ with $s\in S$ and $w_1,w_2,w'\in F^+.$
Of course the subword $v=w_1w'sw_2$ has full support,  as well. Hence by induction on the length  the complex $C^\bullet_v$ is quasi-isomorphic to $C^\bullet_{\cox}.$
On the other hand, for any subword $v_1sv'sv_2$ of $w$, we have $P(v_1sv'v_2)=P(v_1sv'sv_2)$. Hence the difference between the complexes
$C^\bullet_w$ and $C^\bullet_v$ is a contractible
complex. The result follows.
\end{proof}

Let $w\in F^+$ and $s\in \supp(w).$ Hence we may write $w=w_1sw_2$ for some subwords $w_1,w_2\in F^+$ of $w.$ Then we also use the notation
$$w/s:=w_1w_2\in F^+$$
for convenience.  We may define analogously to (\ref{complex_0_F}) the complex
\begin{equation}\label{complex_0_F_s}
C^\bullet_{w,s}: 0 \rightarrow  i^G_{P(w)} \rightarrow \bigoplus_{\genfrac{}{}{0pt}{2}{s\preceq v\preceq w}{\ell(v)=\ell(w)-1}}i^G_{P(v)} \rightarrow
\bigoplus_{\genfrac{}{}{0pt}{2}{s\preceq v\preceq w}{\ell(v)=\ell(w)-2}}i^G_{P(v)}
\rightarrow \dots \rightarrow i^G_{P(s)} \rightarrow 0.
\end{equation}

\begin{eg}
Let $w=\cox_n\in F^+$ be the standard Coxeter element and $s\in S.$ Then the complex $C^\bullet_{w,s}$  coincides
- up to augmentation with respect to the generalized Steinberg representation -  with the complex (\ref{IIIcomplex}) where $d=d(s),$
i.e.,  $P=P(s).$
\end{eg}

With the same arguments as in Proposition \ref{qi_cox} one proves the next statement.

\begin{prop}\label{qi_cox_s}
Let $w \in F^+$ have full support and let $s\in \supp(w).$ Then the complex $C^\bullet_{w,s}$ is quasi-isomorphic to $C^\bullet_{\cox,s}$
if $s\not\in \supp(w/s)$. Otherwise, it is acyclic.\qed
\end{prop}

More generally, let $w\in F^+$ and fix a subword  $u \prec w$. For any $v\in F^+$ with  $u \preceq v\preceq w$, let $P_v \subset G$ be a std psgp chosen inductively in the following way.
Start with an arbitrary std psgp $P_u\subset G.$ Let $u\prec v \preceq w$ and suppose that for all $u\preceq z \prec v$, the std psgp $P_z$ are already defined.
Set $\{z_1,\ldots,z_r\}=\{u \preceq z \prec v \mid \ell(z)=\ell(v)-1\}.$
Then let $P_v\subset G$ be a std psgp such that $i^G_{P_v} = A\oplus B$ where $A\subset \bigcap_i i^G_{P_{z_i}}$
and $B$ maps to zero under all the maps $i^G_{P_v} \to i^G_{P_{z_i}}$ induced by the double cosets of $e\in W$ in $W_{P_v}\setminus  W/W_{P_{z_i}}$
via Frobenius reciprocity. Hence we get a sequence of $G$-representations

\begin{equation}\label{complex_F}
0 \rightarrow  i^G_{P_{w}} \rightarrow \bigoplus_{\genfrac{}{}{0pt}{2}{u\preceq v\preceq w}{\ell(v)=\ell(w)-1}}i^G_{P_{v}} \rightarrow
\bigoplus_{\genfrac{}{}{0pt}{2}{u\preceq v\preceq w}{\ell(v)=\ell(w)-2}}i^G_{P_{v}}
\rightarrow \dots \rightarrow \bigoplus_{\genfrac{}{}{0pt}{2}{u\preceq v \preceq w}{\ell(v)=\ell(u)+1}}i^G_{P_{v}}\rightarrow i^G_{P_u} \rightarrow 0
\end{equation}
which we equip analogously with the same signs as above. By the very construction of the sequence we derive the following fact.

\begin{lemma}\label{lemma_complex}
The  sequence $(\ref{complex_F})$ is a complex.\qed
\end{lemma}

\begin{eg}
Let ${\bf G}=\GL_4$ and $w,u\in F^+$ with $\ell(w)=\ell(u)+2.$ The sequence
 $$i^G_{P_{(3,1)}} \to i^G_{P_{(3,1)}} \oplus i^G_{P_{(2,2)}} \to i^G_{P_{(3,1)}}  $$
 (with differentials as explained above)  is a complex in the sense above, whereas
  $$i^G_{P_{(2,2)}} \to i^G_{P_{(2,2)}} \oplus i^G_{P_{(3,1)}} \to i^G_{P_{(2,2)}}  $$
  is not.
\end{eg}

For later use, we introduce the next definition.

\begin{defn}
 Let $V$ be a finite-dimensional $G$-representation. We denote by $\supp(V)$ the set of isomorphism classes of
irreducible subrepresentations which appear in $V.$
\end{defn}

Let $f: V \to W$ be a homomorphism of $G$-representations. We get for each irreducible $G$-representation $Z$ an induced map
$$f^Z:V^Z \to W^Z$$ of the $Z$-isotypic parts.

\begin{defn}
Let  $f:V \to W$  be a homomorphism of $G$-representations.

i) We call  $f$ si-surjective resp. si-injective resp.  si-bijective if
the map $f^Z$ is surjective resp. injective resp. bijective for all $Z \in \supp(V) \cap \supp(W).$

ii) We say that $f$ has si-full rang if the map $f^Z$ has full rang for all $Z \in \supp(V) \cap \supp(W).$
\end{defn}

\begin{rks}
i) The above definition makes of course sense for arbitrary groups. We will apply it in the upcoming sections in the case of $H:=G\times \Gamma.$
Here we shall  see later on  that the action of $H$ on the considered geometric representations is semi-simple.

ii) Obviously, the homomorphism $f$ has si-full rang if and only if the map $f^Z$ is injective or surjective
for all $Z \in \supp(V) \cap \supp(W).$
\end{rks}

We close this section with the following observation.

\begin{lemma}\label{exact_sequence}
Let $V^1 \stackrel{f_1}{\to} V^2 \stackrel{f_2}{\to} V^3 \stackrel{f_3}{\to} V^4$ be an exact sequence of $H$-modules with
$\supp(V^1) \cap \supp(V^4) =\emptyset.$

a) If $f_1$ is si-surjective then $f_2$ is si-injective.

b) If $f_3$ si-injective then $f_2$ si-surjective,

c) $f_2$ has si-full rang.
\end{lemma}

\begin{proof}
a) Let $Z\in \supp(V^2) \cap \supp(V^3).$ If the map  $f_2^Z$ has a kernel it follows that $V\in \supp(V^1).$ Since $f_1$ is si-surjective
we deduce that $f_2^Z$ is the zero map. Hence $V_3^Z$ maps injectively into $V^4$ which implies $Z \in \supp(V^4)$ which is a contradiction to
the assumption.

b)  Let $Z\in \supp(V^2) \cap \supp(V^3).$ If the map  $f_2^Z$ is not surjective it follows that $V\in \supp(V^4).$ Since $f_3$ is si-injective
we deduce that $f_2^Z$ is the zero map which implies $Z \in \supp(V^1)$ which is again a contradiction to
the assumption.

c) This is obvious.
\end{proof}

\vspace{0.5cm}

\section{Deligne-Lusztig varieties}

Let $X=X_{\bf G}$ be the set of all Borel subgroups of ${\bf G}$. This is a smooth projective algebraic variety homogeneous
under ${\bf G}$.  By the Bruhat lemma the set of orbits of ${\bf G}$ on $X\times X$ can be identified with $W.$ We denote by $\mcO(w)$ the orbit of
$(B,wBw^{-1})\subset X\times X$ and by $\overline{\mcO(w)}\subset X\times X$ its Zariski closure.

Let $F: X \rightarrow X$ be the Frobenius map over $\mF_q$.
The {\it Deligne-Lusztig variety associated to} $w\in W$ is the locally closed subset of $X$ given by
\begin{equation*}
X(w)=X_{{\bf G}}(w)=\big\{x\in X\mid {\rm inv}(x,F(x))=w\big\}
\end{equation*}
where by definition ${\rm inv}(x,F(x))=w \Leftrightarrow (x,F(x))\in \mcO(w).$
Denote by $\leq$ the Bruhat order and by $\ell$ the length function on $W.$
Then $X(w)$ is a smooth quasi-projective variety of
dimension  $\ell(w)$ defined  over $\mF_q $ and  which is equipped with an action of $G$, cf. \cite[1.4]{DL}.
We denote by $\overline{X(w)}\subset X$ its Zariski closure.

Before we proceed, let us recall some properties of the varieties $\mcO(w)$ we need in the sequel, cf. loc.cit.
If $w=w_1w_2$ with $\ell(w)=\ell(w_1)+\ell(w_2)$ then
\begin{enumerate}\label{properties_O}
 \item a) $(B,B')\in \mcO(w_1)$ and $(B',B'')\in \mcO(w_2)$ implies $(B,B'')\in \mcO(w)$

\smallskip
\noindent b) If $(B,B'')\in \mcO(w)$, then there is a unique $B'\in X$ with  $(B,B')\in \mcO(w_1)$ and $(B',B'')\in \mcO(w_2).$

\smallskip
\noindent In other words, there is an isomorphism of schemes $\mcO(w_1) \times_{X} \mcO(w_2) \cong \mcO(w).$

\smallskip
\item Let $w,w'\in W.$ Then $\mcO(w') \subset \overline{\mcO(w)}$ $\Leftrightarrow w' \leq w$  for the Bruhat order $\leq $ on $W.$
\end{enumerate}

As in the case of usual Schubert cells, there is by item (2) the following relation concerning the closures of DL-varieties. Let $w',w\in W$.
Then 
$$X(w') \subset \overline{X(w)} \Leftrightarrow w'\leq w.$$
It follows that we have a Schubert type stratification
\begin{equation}\label{stratification}
\overline{X(w)}=\bigcup_{v\leq w} X(v).
\end{equation}
In particular if $w'\leq w$ with $\ell(w)=\ell(w')+1,$ then $X(w) \cup X(w')$ is a locally closed subvariety of $X$ which is moreover smooth
since the dimensions of $X(w)$ and $X(w')$ differs by one.

\begin{eg}\label{Example_1}
Let $G=\GL_3$ and identify $X$ with the full flag variety of $V=\mF^3$. Then
\begin{eqnarray*}
X(1) & = & X(\mF_q) \\
X((1,2)) & = & \big\{(0) \subset V^1 \subset V^2 \subset V \mid V^2 \mbox{ is $k$-rational},F(V^1) \neq V^1 \big\} \\
X((2,3)) & = & \big\{(0) \subset V^1 \subset V^2 \subset V \mid V^1 \mbox{ is $k$-rational}, F(V^2) \neq V^2 \big\} \\
X((1,2,3)) & = & \big\{(0) \subset V^1 \subset V^2 \subset V \mid F(V^1) \subset V^2 , F(V^i) \neq V^i,i=1,2\big\} \\
X((1,3,2)) & = & \big\{(0) \subset V^1 \subset V^2 \subset V \mid V^1 \subset F(V^2) , F(V^i) \neq V^i,i=1,2\big\} \\
X((1,3)) & = & \big\{(0) \subset V^1 \subset V^2 \subset V \mid F(V^1) \nsubseteq V^2, V^1 \nsubseteq  F(V^2)\big\} .
\end{eqnarray*}
\end{eg}

\bigskip
Let $S=\{s_1,\ldots,s_{n-1}\} \subset W$ be the set of simple reflections.
Recall that a Coxeter element $w\in W$ is given by any product of all  $s\in S$ (with multiplicity one).
In the sequel we denote
by $$\cox_n:=s_1\cdot s_2 \cdots s_{n-1}=(1,2,\ldots,n)\in W$$ the standard Coxeter element.

\begin{eg}
Let $w=\cox_n$. Then $X(w)$ can be identified via the projection map $X \to \mP^{n-1}$  with the Drinfeld space
$$\Omega(V)=\Omega^n=\mP^{n-1} \setminus \bigcup\nolimits_{H / \mF_q} H$$
(complement of all $\mF_q$-rational hyperplanes in the projective space of lines in $V=\mF^n$), cf. \cite{DL}, \S 2.
Its inverse is given by the map
\begin{eqnarray*}
\Omega(V) & \to & X(w) \\
x & \mapsto & x \subset x+F(x) \subset x+F(x)+F^2(x) \subset \cdots \subset V
\end{eqnarray*}
For any Coxeter element $w$ for $\GL_n$, the corresponding DL-variety
$X(w)$ is universally homeomorphic to $\Omega^n$, cf. \cite{L2}, Prop. 1.10.
\end{eg}

In the sequel we denote for any variety $X$ defined over $k$, by $H^i_c(X)=H^i_c(X,\overline{\mQ}_\ell)$
(resp. $H^i(X)=H^i(X,\overline{\mQ}_\ell)$) the $\ell$-adic
cohomology with compact support (resp. the $\ell$-adic
cohomology) in degree $i.$ For a Deligne-Lusztig variety $X(w)$, there is by functoriality an action of $H=G\times \Gamma$
on  these cohomology groups.

\begin{prop}\label{cohomology_Coxeter}
Let $w=\cox_n$ be the standard Coxeter element. Then
$$H_c^\ast(X(w))= \bigoplus_{k=1}^n j_{(k,1,\ldots,1)}(-(k-1))[-(n-1)-(k-1)]. $$
\end{prop}
(Here for an integer $m\in \mZ$, we denote as usual by $(m)$ the Tate twist of degree $m.$)

\begin{proof}
 See \cite{L2,O} resp. \cite{SS} in the case of a local field.
\end{proof}

The cohomology of the Zariski closure of $X(\cox_n)$ has the following description.

\begin{prop}\label{cohomology_Coxeter_compactification}
Let $w=\cox_n\in W$ be the standard Coxeter element. Then
$$H^\ast(\overline{X(w)})= \bigoplus_{v\leq w} J(v)(-\ell(v))[-2\ell(v)], $$
where $J(v)=J^G(v)$ is defined inductively as follows. Write $v$ in the shape $$v=v'\cdot  s_{j+1}\cdot s_{j+2}\cdots s_{j+l}$$ where
$1\leq j \leq n-2$, $l\geq 0$ and where $v'\leq \cox_j=s_1\cdots s_{j-1}$. Then $$J(v)=\Ind^G_{P_{(j,n-j)}}(J^{\GL_j}(v')\boxtimes {\bf 1}).$$
\end{prop}

\begin{proof}
In terms of flags the DL-variety $X(w)$ has the description
\begin{eqnarray*}
X(w) & = &\big\{V^\bullet \mid F(V^j) \subset V^{j+1}, F(V^j) \neq V^j, \; \forall \,1\leq j \leq n-1 \big\}.
 \end{eqnarray*}
The Zariski closure $\overline{X(w)}$ of $X(w)$ in $X$ is then given by the subset
\begin{eqnarray*}
\overline{X(w)} & = &\big\{V^\bullet \mid F(V^j) \subset V^{j+1}, \; \forall\, 1\leq j \leq n-1 \big\}
\end{eqnarray*}
which can be identified with a sequence of blow-ups, cf. \cite{Ge,GK,I}. Start with $Y_0=\mP^{n-1}=\mP(V)$ where $V=\mF^n$ and consider the blow up $B_1$ in the set of rational
points $Z_0=\mP^{n-1}(k)\subset Y_0.$ Then we may identify $B_1$ with  the variety $\{V^1\subset V^2 \subset V \mid F(V^1) \subset V^2 \}.$ We set
$Y_1= \bigcup_{W \in {\rm Gr}_2(V)(k)} \mP(W) \subset B_1$ which is the strict transform of the finite set of $k$-rational planes and
blow up $B_1$ in $Y_1.$ The resulting variety $B_2$ can be identified with $\{V^1 \subset V^2  \subset V^3 \subset V \mid F(V^i) \subset V^{i+1},i=1,2 \}.$
Now we repeat this construction successively until we get $B_{n-2}=\overline{X(w)}.$
Hence the cohomology of $\overline{X(w)}$ can be deduced from the usual formula for blow ups \cite{SGA5}. More precisely, each time we blow up, we have to add
the cohomology of the variety 
$$\coprod_{W\in {\rm Gr}_j(V)(k)} \overline{X_{\GL(W)}(\cox_j)} \times \mP({V/W})\setminus \overline{X_{\GL(W)}(\cox_j)} \times \mP^0$$
which is 
$$\bigoplus_{v'\leq \cox_j}\bigoplus_{l=1}^{n-j-1} \Ind^G_{P_{(j,n-j)}}(J^{\GL_j}(v')\boxtimes {\bf 1}) (-\ell(v'\cdot s_{j+1}\cdot s_{j+2}\cdots s_{j+l}))[-2(\ell(v')+l)].$$
The start of this procedure  is given by the cohomology of the projective space $\mP(V)$, which we initialize by
$H^{2i}(\mP(V))=J(s_1s_2\cdots s_i)(-i)=i^G_G(-i), i=1,\ldots,n-1.$
\end{proof}

\begin{rks}\label{rk_cox_chow}
i) Since all DL-varieties for Coxeter elements are  homeomorphic \cite{L2}, they  have all the same cohomology.
By considering stratifications (\ref{stratification}) for different Coxeter elements, the same is true for their Zariski closures.
Alternatively, one might argue that the morphism $\sigma,\tau$ of the upcoming Proposition \ref{Prop_homeomorphic}
for   different Coxeter elements extend to their Zariski closures, thus inducing an isomorphism on their cohomology.

ii) It follows by the description of $\overline{X(\cox_n)}$ in terms of blow ups together with the remark before that thy cycle map
$$A^i(\overline{X}(w))_{\overline{\mQ}_\ell}\to H^{2i}(\overline{X(w)})$$
is an isomorphism for every Coxeter element $w.$ In fact, in the Chow group  the same formulas concerning blow ups are valid \cite{SGA5}.
The first Chow group of $\overline{X(\cox_n)}$ is also considered in \cite{RTW}.

iii) Let $w=s_{n-1}s_{n-2}\cdots s_2s_1$. Then by symmetry, there is is by considering the dual projective space $(\mP^{n-1})^\vee$ a  procedure
for realising
$$\overline{X(w)}=\big\{V^\bullet \mid V^i \subset F(V^{i+1}) ,\; i=1,\ldots, n-1\big\}$$ as a
sequence of blow ups.

iv) Since $\overline{X(\cox_n)}$ is smooth and projective the above formula can be also deduced by considering its stratification
into DL-varieties together with Proposition \ref{cohomology_Coxeter} and the statement below.
\end{rks}

On the other hand, for  elements $w\in W$ having not full support,  the cohomology of $X(w)$ can be deduced by an induction process. Here recall that
we say that $w$  has full support if it not contained in any proper parabolic subgroup $W_P$ of $W$.

\begin{prop}\label{cohomology_not_full}
Let $w\in W$ have not full support. Let ${\bf P=P_{(i_1,\ldots,i_r})}$ be a minimal parabolic subgroup such that $w \in W_P$.  Then
$$H_c^\ast(X(w))= \Ind^G_P\big(H_c^\ast(X_{\bf M_P}(w))\big)$$
and
$$H^\ast(\overline{X(w)})= \Ind^G_P\big(H^\ast(\overline{X_{\bf M_P}(w)})\big)$$
where $X_{\bf M_P}(w)=\prod_{j=1}^r X_{\GL_{i_j}}(w_j)$ and $w=w_1 \cdots w_r$ with $w_j\in S_{i_j}, j=1,\ldots,r.$
\end{prop}

\begin{proof}
See \cite[Prop. 8.2]{DL}.
\end{proof}

\begin{rk}
The proof in loc.cit.  shows that there is an identification of varieties $X(w)=\Ind^G_P\big (X_{\bf M_P}(w)\big).$
Hence we have the same formulas for Chow groups. In particular, by using Remark \ref{rk_cox_chow} we see that the cycle map is an 
isomorphism for all Coxeter elements in Levi-subgroups, as well.
\end{rk}

In the case when we erase some simple reflection in a reduced expression of a Coxeter element, we may deduce from the above results
the following consequence. Here we abbreviate the partition or decomposition  $(k,1,\ldots,1)\in \P$ by $(k,1^{(n-k)})$.
For any triple of integers $i,k,l\in \mZ$ with $1\leq i\leq n$ and $1\leq k \leq i$, $1\leq l \leq n-i$, we set
$$A_{k,l}= i^G_{P_{(k,1^{(i-k)},l,1^{(n-i-l)})}}/ \sum_{d_1 > (k,1^{(i-k)})} i^G_{P_{(d_1,l,1^{(n-i-l)})}}  + \sum_{d_2 > (l,1^{(n-i-l)})}i^G_{P_{(k,1^{(k-i)},d_2)}}.$$

\begin{coro}\label{coh_kleiner_Cox}
Let $w=\cox_n=s_1\cdots s_{n-1}$ be the standard Coxeter element and
let\footnote{Here the symbol $\hat{s_i}$ means as usual that $s_i$ is deleted from the above expression.} 
$w'=s_1\cdots \hat{s_i} \cdots s_{n-1}\in W$. Then
$$H_c^\ast(X(w'))=\bigoplus_{m=2}^{n}\bigoplus_{k+l=m} A_{k,l} (-(m-2))[-(n-2)-(m-2)].   $$
\end{coro}

\begin{proof}
The Weyl group element $w'$ is contained in the parabolic subgroup $W_{P}$ with $P=P_{(i,n-i)}.$ Now the expressions  
$s_1\cdots s_{i-1}$
and $s_{i+1}\cdots s_{n-1}$ are both Coxeter elements in the
Weyl groups attached to $M_1=\GL_i$ resp. $M_2=\GL_{n-i}.$ The cohomology of $H_c^\ast(X_{M_1}(w_1))$ resp. $H_c^\ast(X_{M_2}(w_2))$ are
given by Proposition \ref{cohomology_Coxeter}:
$$H_c^\ast(X_{M_1}(w_1))= \bigoplus_{k=1,\ldots,i} j^{M_1}_{(k,1,\ldots,1)}(-(k-1))[-(i-1)-(k-1)]$$
resp.
$$H_c^\ast(X_{M_2}(w_2))= \bigoplus_{l=1,\ldots,n-i} j^{M_2}_{(l,1,\ldots,1)}(-(l-1))[-((n-i)-1)-(l-1)].$$
Thus we get by Proposition \ref{cohomology_not_full} and identity (\ref{identity_induced_Steinberg})
\begin{eqnarray*}
 H_c^\ast(X(w')) & = & \Ind^G_{P_{(i,n-i)}}\Big(H_c^\ast(X_{M_1}(w_1)) \boxtimes H_c^\ast(X_{M_2}(w_2))\Big) \\
 & = & \bigoplus_{m=2}^{n}\bigoplus_{k+l=m} A_{k,l} (-(k+l-2))[-(n-2)-(k+l-2)].
\end{eqnarray*}
\end{proof}

Let $w=\cox_n\in W$. As we see from Proposition \ref{cohomology_Coxeter_compactification}
the cohomology of $\overline{X}(w)$ vanishes in odd degree.
For any integer $i\geq 0$, let
\begin{equation}\label{complex1}
 H^{2i}(\overline{X(w)})\to \bigoplus_{v < w  \atop \ell(v)=\ell(w)-1} H^{2i}(\overline{X(v)})\to \cdots \to
\bigoplus_{v < w  \atop \ell(v)=1} H^{2i}(\overline{X(v)})\to  H^{2i}(\overline{X(e)})
\end{equation}
be the natural complex induced by the closed complement $\bigcup_{v < w} \overline{X(v)}$ of $X(w)$ in $\overline{X(w)}.$
This complex determines the  contribution with Tate twist $-i$ to the  cohomology of $X(w).$

On the other hand, we may consider the grading
$$H^{2i}(\overline{X(w)})=\bigoplus_{z \leq w \atop \ell(z)=i} H(w)_z$$
described in Proposition \ref{cohomology_Coxeter_compactification}, i.e. $H(w)_z=J(z)(-\ell(z))$ for $z<w.$
By Proposition \ref{cohomology_not_full} we also have such a grading for all subexpressions $v < w$, i.e.
$H^{2i}(\overline{X(v)})=\bigoplus_{z \leq  v \atop \ell(z)=i} H(v)_z$.

\begin{lemma}
 let $v_1 < v_2\leq  w$ with $\ell(v_2)=\ell(v_1)+1.$ Then $H(v_2)_z \subset H(v_1)_z$ for all $z\leq v_1.$
\end{lemma}

\begin{proof}
By Proposition \ref{cohomology_not_full} we may suppose that $w=v_2$ and $v=v_1=s_1\cdots \hat{s_i}\cdots s_{n-1}.$ 
Let $z=z'\cdot  s_{j+1}\cdot s_{j+2}\cdots s_{j+l}$ with $z'\leq \cox_j$ and $l \geq 1$
Then $H(w)_z=\Ind^G_{P_{(j,n-j)}}(J^{\GL_j}(z')\boxtimes {\bf 1}).$

If $i< j$  then $z'=z_1'z_2'$ with $z_1'< \cox_i$ and $z_2'| s_{i+1}\cdots s_{j-1}.$ Thus we see that
$H(v)_z=\Ind^G_{P_{(i,n-i)}}(J^{\GL_{i}\times \GL_{n-i}}(z))$
with
\begin{eqnarray*}
 J^{\GL_{i}\times \GL_{n-i}}(z) & = & J^{\GL_{i}}(z_1') \boxtimes J^{\GL_{n-i}}(z_2'\cdot  s_{j+1}\cdot s_{j+2}\cdots s_{j+l}) \\
& = & J^{\GL_{i}}(z_1') \boxtimes \Ind^{\GL_{n-i}}_{P_{(j-i,n-j)}}(J^{\GL_{j-i}}(z_2') \boxtimes {\bf 1}).
\end{eqnarray*}
Hence $H(v)_z=\Ind^G_{P_{(i,j-i,n-j)}}\big(J^{\GL_{i}}(z_1') \boxtimes J^{\GL_{j-i}}(z_2')\boxtimes {\bf 1}\big).$ 
Since we have by induction on $j$ the inclusion  $J^{\GL_j}(z') \subset \Ind^{\GL_j}_{P_{(i,j-i)}}(J^{\GL_{i}}(z_1') \boxtimes J^{\GL_{j-i}}(z_2'))$
the claim follows in this case.

If $i=j$ then one deduces moreover by arguing in the same way that the inclusion is an identity, i.e., $H(w)_z=H(v)_z$
as $J^{\GL_{n-i}}(s_{i+1}s_{i+2}\cdots s_{i+l})$ is the trivial representation.

If $i>j$ then we  necessarily  have $i>j+l$. Then $z=z_1z_2$ with $z_1\leq \cox_i$ and $z_2=e$. Thus 
\begin{eqnarray*} H(v)_z & = & \Ind^G_{P_{(i,n-i)}}(J^{\GL_{i}\times \GL_{n-i}}(z))\\ 
& = & \Ind^G_{P_{(i,n-i)}}(J^{\GL_{i}}(z_1)\boxtimes J^{\GL_{n-i}}(z_2))\\
& = & \Ind^G_{P_{(i,n-i)}}(\Ind^{\GL_{i}}_{P_{(j,i-j)}}(J^{\GL_j}(z')\boxtimes {\bf 1} )\boxtimes i^{\GL_{n-i}}_{B\cap \GL_{n-i}})\\
& = &\Ind^G_{P_{(j,i-j,n-i)}}(J^{\GL_j}(z')\boxtimes {\bf 1} \boxtimes i^{\GL_{n-i}}_{B\cap \GL_{n-i}}).
\end{eqnarray*}
Hence it contains $H(w)_z.$
\end{proof}

By Lemma \ref{lemma_complex} get thus for each subword $z\leq w$, a complex 
 $$H(\,\cdot \,)_z: H(w)_z \to \bigoplus_{z\leq v \leq w \atop \ell(v)=\ell(w)-1}  H(v)_z\to \cdots
\to \bigoplus_{z\leq v \leq w  \atop \ell(v)=1} H(v)_z \to \bigoplus_{z \leq e \atop \ell(z)=i} H(e)_z.$$

\begin{lemma}
If  $z\not\in \{s_1,s_1s_2,\ldots ,\cox_n\}$ then this complex is acyclic. Otherwise, it coincides
with the complex (\ref{IIIcomplex}) with respect to the  generalized Steinberg representation $v^G_{P(z)}$.
\end{lemma}

\begin{proof}
 Let $z=s_1s_2\cdots s_i$. Then the complex $H(\,\cdot\,)_z$ is easily verified by the above description just the complex 
 (\ref{IIIcomplex})  with respect to $v^G_{P(z)}$. So let $z\not\in \{s_1,s_1s_2,\ldots, \cox_n\}$ and write
$z=z'\cdot  s_{j+1}\cdot s_{j+2}\cdots s_{j+l}$ with $z'\leq \cox_j$ and $l \geq 1$ as above. In particular, we have $j\geq 1.$
By the proof of the foregoing lemma, we know that for $z\leq v \leq w$ with $s_j|v$ we have $H(v)_z=H(v/s_j)_z.$   
By standard simplicial arguments it follows that the complex is contractible. 
\end{proof}

Hence by summing up we get a graded complex
\begin{eqnarray}\label{complex2}
 & & \bigoplus_{z  \leq w \atop \ell(z)=i} H(w)_z\to \!\!\!\!\!\bigoplus_{v \leq w \atop \ell(v)=\ell(w)-1} \bigoplus_{z \leq v \atop \ell(z)=i} H(v)_z\to \cdots
\to \bigoplus_{v \leq w  \atop \ell(v)=1} \bigoplus_{z \leq v \atop \ell(z)=i} H(v)_z\to \bigoplus_{z \leq e \atop \ell(z)=i} H(e)_z \\ \nonumber
& = & \bigoplus_{z \leq w \atop \ell(z)=i} \Big(H(w)_z \to \bigoplus_{z\leq v \leq w \atop \ell(v)=\ell(w)-1}  H(v)_z\to \cdots
\to \bigoplus_{z\leq v \leq w  \atop \ell(v)=1} H(v)_z \to \bigoplus_{z \leq e \atop \ell(z)=i} H(e)_z \Big).
\end{eqnarray}

\medskip
 
\begin{prop}\label{Coxeter_complex}
The complexes (\ref{complex1}) and (\ref{complex2}) are  quasi-isomorphic.
\end{prop}

\begin{proof}
We have a natural  morphism of complexes $(\ref{complex2}) \to (\ref{complex1})$. The claim follows now from the 
foregoing lemma and Proposition \ref{cohomology_Coxeter}.
\end{proof}

 More generally, we believe that the natural generalization holds true.

\begin{conj}\label{conjecture_coxeter_u}
Let $u < w$ and $i\in \mN$. Then for all $v\in W$ with $u \leq v \leq w$,  there are  gradings
$$H^{2i}(\overline{X(v)})=\bigoplus_{z \leq v \atop \ell(z)=i} H(v)_z$$ 
into induced representations $H(v)_z=i^G_{P^v_z}$ for certain parabolic subgroups $P^v_z$
(here the $H(v)_z$ do not necessarily coincide with the expressions in (\ref{complex2})) such that the complex
\begin{equation}
 H^{2i}(\overline{X(w)})\to \!\!\!\!\bigoplus_{u \leq v \leq w  \atop \ell(v)=\ell(w)-1} H^{2i}(\overline{X(v)})\to \cdots \to
\bigoplus_{u \leq  v \leq  w  \atop \ell(v)=\ell(u)+1} H^{2i}(\overline{X(v)})\to  H^{2i}(\overline{X(u)})
\end{equation}
is quasi-isomorphic to a graded complex
\begin{eqnarray}
 & & \bigoplus_{z \leq w \atop \ell(z)=i} H(w)_z\to \!\!\!\!\bigoplus_{u \leq v \leq w  \atop \ell(v)=\ell(w)-1} \bigoplus_{z \leq v \atop \ell(z)=i} H(v)_z\to \cdots
\to \bigoplus_{u \leq v \leq w  \atop \ell(v)=1} \bigoplus_{z \leq v \atop \ell(z)=i} H(v)_z\to \bigoplus_{z \leq u \atop \ell(z)=i} H(u)_z\\
 \nonumber & = & \bigoplus_{z \leq w \atop \ell(z)=i}\Big(H(w)_z\to  \bigoplus_{u \leq v \leq w  \atop \ell(v)=\ell(w)-1} H(v)_z\to \cdots
\to \bigoplus_{u \leq v \leq w  \atop \ell(v)=1}  H(v)_z\to \bigoplus_{z \leq u \atop \ell(z)=i} H(u)_z \Big).
\end{eqnarray}
\end{conj}

\bigskip
We shall prove a more concrete version of  the conjecture in the case where $\ell(u)=\ell(w)-1$.

\begin{prop}\label{qi_restriction}
Let $w$ be a Coxeter element and let $w'\in W$ with $w'< w$ and $\ell(w')=\ell(w)-1.$ There  are gradings
$H^{2j}(\overline{X(w)})=\bigoplus_{z \leq w \atop \ell(z)=j} H(w)_z$ and $H^{2j}(\overline{X(w')})=\bigoplus_{z \leq w' \atop \ell(z)=j} H(w')_z$
into induced representations such that the  homomorphism
$H^{2j}(\overline{X(w)}) \to H^{2j}(\overline{X(w')})$ coincides with the graded one. Moreover, the maps
$H(w)_z \to H(w')_z$ are injective for all $z\leq w'.$
\end{prop}

As for the proof we note that by Remark \ref{rk_cox_chow}  the cycle map $A^j(\overline{X(w)})_{\overline{\mQ}_\ell}\to H^{2j}(\overline{X(w)})$ is an isomorphism. 
The same holds true for the subvariety $\overline{X(w')}.$ 
Moreover for the Chow groups, we have by the constructive proof an integral version of Proposition 
\ref{cohomology_Coxeter_compactification}, i.e. the induced representations $i^G_P$ appearing in loc.cit. are
replaced by their integral models $i^G_P(\mZ)=\{f: G \to \mZ \mid f(gp)=f(g)\; \forall p \in P\}$.
Thus by  Poincar\'e duality, the above result follows from the following statement.

\begin{prop}
Let $w$ be a Coxeter element and let $w'\in W$ with $w'< w$ and $\ell(w')=\ell(w)-1.$ 
There are gradings $A_{j}(\overline{X(w)})=\bigoplus_{z \leq w \atop \ell(z)=j} A(w)_z$ and 
 $A_{j}(\overline{X(w')})=\bigoplus_{z \leq w' \atop \ell(z)=j} A(w')_z$ into induced representations $i^G_P(\mZ)$
such that the homomorphism
$A_{j}(\overline{X(w')}) \to A_{j}(\overline{X(w)})$ coincides with the graded one. Moreover, the maps
$A(w')_z \to A(w)_z$ are surjective for all $z\leq w'.$
\end{prop}

\begin{proof}
We may assume that $w=\cox_n$ is the standard Coxeter element. If $w'=s_2\cdots s_{n-1}\in W$, the claim is a result 
of the proof of Proposition \ref{cohomology_Coxeter_compactification}
following inductively  the process of blow ups. Here we may consider on the cohomology groups $A_{j}(\overline{X(w)})$ and
$A_{j}(\overline{X(w')})={\rm Ind}^G_{P_{(1,n-1)}}\big(A_{j}(\overline{X_{\GL_1\times \GL_{n-1}}(w')})\big)$  the natural
gradings given by loc.cit. 

 If on the other extreme $w'=s_1 \cdots s_{n-2}$, then we identify $\overline{X(\cox_n)}$ with the variety $\overline{X(s_{n-1}s_{n-2}\cdots s_2s_1)}$
 (cf. Remark \ref{rk_cox_chow} i))  and argue in the same way as above using the variant of Proposition \ref{cohomology_Coxeter_compactification} (cf. Remark \ref{rk_cox_chow} iii))
  describing the latter space as a sequence of blow ups using hyperplanes. Here the gradings are induced by mirroring the Dynkin diagram and 
  the   induced representations.

In general, let $w'=s_1\cdots \hat{s_i} \cdots s_{n-1}.$ Here the reasoning is a kind of mixture of the previous cases. We consider
the exact sequences \cite[Prop. 6.7 e)]{Fu}
$$0 \to A_i(Y_{n-3}) \to A_i(\tilde{Y}_{n-3}) \oplus A_i(B_{n-3}) \to A_i(\overline{X}(w))\to 0 $$
and
$$0 \to A_i(Y_{n-3}') \to A_i(\tilde{Y}_{n-3}') \oplus A_i(B_{n-3}') \to A_i(\overline{X}(w'))\to 0 $$
where $B_{n-3}'=B_{n-3} \cap \overline{X}(w')$, $Y_{n-3}'=Y_{n-3} \cap \overline{X}(w')$, 
$\tilde{Y}_{n-3}=\overline{X}(s_1s_2 \ldots s_{n-3}s_{n-1} )$ is the preimage of $Y_{n-3}$ in $\overline{X}(w)$
and  $\tilde{Y}_{n-3}'=\tilde{Y}_{n-3}' \cap \overline{X}(w').$
By induction on $n$ and the steps in the blow up process it suffices to
show that the claim of the statement is true for the maps $i_\ast:A_j(Y_{m-1}) \to  A_j(B_{m-1})$ with $m\neq 1,n-1.$ 

Recall that  
$$B_{m-1}=\{(0) \subset V^1 \subset V^2 \subset \cdots \subset V^{m-1} \subset V^{m} \subset V \mid F(V^j) \subset V^{j+1} ,j=1,\ldots ,m-1\}$$ 
and 
$$Y_{m-1} =\{V^\bullet \in B_{m-1} \mid F(V^m)=V^m\}.$$
We can consider the equivalent situation given by 
$$B'_{m-1}=\{(0) \subset V^1 \subset V^2 \subset \cdots \subset V^{m-1} \subset V^{m}\subset V \mid V^j \subset F(V^{j+1}) ,j=1,\ldots ,m-1\}$$ 
and 
$$Y'_{m-1} =\{V^\bullet \in B'_{m-1} \mid F(V^m)=V^m\}.$$
Similarly as above this situation is induced by successive blow ups starting with 
$B=\{(0) \subset V^{m-1} \subset V^{m}\subset V\mid V^{m-1} \subset F(V^{m}) \}$ and $Y=\{V^\bullet \in B \mid F(V^m)=V^m\},$ respectively.
Hence the claim follows from the next lemma.
\end{proof}

\begin{lemma}
 Let $m\geq 1$ and let $B=B_{(m)}=\{(0) \subset V^{m-1} \subset V^{m}\subset V\mid F(V^{m-1}) \subset V^{m} \}.$ 
 
 i) The  cycle map induces for all $j\geq 0$,  an isomorphism $A^j(B)_{\bar{\mQ}_\ell} \cong H^{2j}(B)$. Set 
 $I=\{z\leq \cox_n \mid z=s_{k+1}s_{k+2}\ldots s_{k+l} \mbox { for $l\geq 1$ and $k\leq m-1$ }  \}$ Then for $m\leq \frac{n}{2},$ we further have
 $A_j(B)=\bigoplus_{z\in I\atop \ell(z)=j} A_B(z)$ with $A_B(z)\cong i^G_{P_{(k,n-k)}}(\mZ)$ for 
 $z=s_{k+1}s_{k+2}\ldots s_{k+l}.$

 ii) Let $Y=\{V^\bullet \in B \mid F(V^m)=V^m\}\subset B.$ Then there are gradings on $A_j(B)$ and $A_j(Y)$ such the the induced homomorphism 
 $A_j(Y) \to A_{j}(B)$ is in diagonal form.
\end{lemma}

\begin{proof}
 The proof is by induction on $m$. By symmetry we may assume that $m \leq \frac{n}{2}.$ If $m=1$, then the first claim follows from the proof
 of Proposition \ref{cohomology_Coxeter_compactification}.
 In general, we consider the diagram

 \medskip
\noindent $\begin{matrix}
  & &  \big\{(0)\subset V^{m-1} \subset V^m \subset V^{m+1} \subset V \mid F(V^k)\subset V^{k+1}, k=m-1,m \big\}  & & \\ & \swarrow & & \searrow & \\
  B_{(m)} & & &  & B_{(m+1)}  
 \end{matrix}$
 
 \medskip
 \noindent where the maps are the projections. 
 By induction me may suppose that the statement is true for $B_{(m)}.$
 All the appearing varieties are smooth and projective. 
 Thus we get the desired formula in i) by first blowing up and then blowing down and using \cite[Prop. 6.7]{Fu}.
 
 Concerning the second statement, we have $Y=\coprod_{W\in \Gr_m(V)(k)} \check{\mP}(W).$ Hence 
 $A_j(Y)=\Ind^G_{P_{(m,n-m)}}\big(A_j(\check{\mP}(W)\big)$ which is labeled by the (single) element $s_{m-j}\cdots s_{m-1}\leq \cox_n$ 
 (which is  $e$ for $j=0$) and which is identified with the $\ell(z)$-dimensional cycle 
$$\bigoplus_{H\in {\rm Gr}_m(V)(\mF_q)} \!\!\!\!\!\big\{(0)\subset V^{m-1} \subset H  \subset V\mid V^{m-1} \mbox{ contains  an $m-j-1$-dim. rational subspace}\big\}  .$$
Apriori we identify (by the very construction) for $z=s_m\cdots s_{m+l}$  the $\ell(z)$-dimensional cycle  $A_B(z)$ with
$$\bigoplus_{W\in {\rm Gr}_{m-1}(V)(\mF_q)}  \!\!\!\!\!\!\!\!\!\! \big\{(0)\subset W\subset V^{m} \subset V\mid V^{m} 
\mbox { is contained in an $m+l+1$-dim. rational subspace} \big\}.$$
We replace $A_B(z)$ by the cycle
\begin{multline*}
\bigoplus_{W\in {\rm Gr}_{m-1}(V)(\mF_q)} \sum_{H\in {\rm Gr}_{m}(V)(\mF_q) \atop W\subset H}  \big\{(0) \subset V^{m-1} \subset H \subset V \mid  V^{m-1} \subset H  \subset V\mid V^{m-1} \\ \mbox{ contains  an $m-j-1$-dim. rational subspace}\big\}$$ 
\end{multline*}
and relabel it with $s_{m-l-1}\dots s_{m-1}$. On the other hand, we replace the cycle $A_B(z)$
where $z=s_{m-l-1}\dots s_{m-1}$  with
\begin{multline*} \bigoplus_{L\in {\rm Gr}_{m-l-2}(V)(\mF_q)}\sum_{W\in {\rm Gr}_{m-1}(V)(\mF_q)\atop L\subset W}\!\!\! \big\{(0) \subset W 
\subset V^m \subset V\mid V^m \mbox{ is contained }  \\  \mbox{in a $m+l+1$-dim. rational subspace}\big\} 
\end{multline*}
and relabel it  by $s_m\cdots s_{m+l}.$ Hence we get a new grading $A_j(B)=\bigoplus_{z\in I \atop \ell(z)=j} A'_B(z).$
Then it is clear that the map $i_\ast$ is in diagonal form and  maps surjectively  $A_Y(z)$ onto $A'_B(z)$ (since $i\leq \frac{n}{2}$) for all $z\leq \cox_{m-1}.$
The claim follows.  
\end{proof}

\begin{egs}\label{Example_cycle}
i) Let $j=1.$ If $w'=s_2\cdots s_{n-1},$ then we have the natural
gradings given by loc.cit., i.e.
$$H(w)_{s_1}=i^G_G(-1) \mbox{ and } H(w)_{s_{i+1}}=i^G_{P_{(i,n-i)}}(-1), \; i\geq 1,$$
resp.
$$H(w')_{s_2}=i^G_{P_{(1,n-1)}}(-1) \mbox{ and } H(w')_{s_{i+1}}=i^G_{P_{(1,i-1,n-i)}}(-1),\;  i\geq 2 .$$
 If on the other extreme $w'=s_1 \cdots s_{n-2}$, then we identify as explained above $\overline{X(\cox_n)}$ with 
 $\overline{X(s_{n-1}s_{n-2}\cdots s_2s_1)}$  and argue in the same way as above using the variant of 
 Proposition \ref{cohomology_Coxeter_compactification} We set 
  $$H(w)_{s_{n-1}}=i^G_G(-1) \mbox{ and } H(w)_{s_{i}}=i^G_{P_{(i+1,n-i-1)}}(-1), \; i < n-1,$$
resp.
$$H(w')_{s_{n-2}}=i^G_{P_{(n-1,1)}}(-1) \mbox{ and } H(w')_{s_{i}}=i^G_{P_{(i+1,n-i-2,1)}}(-1),\;  i< n-2 .$$

Now let $w'=s_1\cdots \hat{s_i} \cdots s_{n-1}\in W$ for some $1<i<n-2$. 
On $H^{2}(\overline{X(w')})$ we consider the natural grading induced by Proposition \ref{cohomology_Coxeter_compactification}, i.e.,
we set 
$$H(w')_{s_{j}}=\left\{\begin{array}{lc}
                              i^G_{P_{(j-1,i-(j-1),n-i)}}(-1) & ; \; 1 < j < i  \\ \\
                              i^G_{P_{(i,n-i)}}(-1) & ; \;j=1, j=i+1\\ \\
                              i^G_{P_{(i, j-(i+1),n-j+1)}}(-1) & ;\; i+1  <  j < n
                            \end{array} \right. .$$
As for $w$ we set,
$$H(w)_{s_{j}}=\left\{\begin{array}{lc}
                              i^G_{P_{(j-1,n-j+1)}}(-1) & ;\; j \neq i, \neq 1 \\ \\
                              i^G_G(-1) &  ;\; j = i \\ \\
                              i^G_{P_{(i-1,n-i+1)}}(-1) & ; \; j = 1
                            \end{array} \right.$$
Here the contribution  $i^G_{P_{(k,n-k)}}(-1)$ is induced by the cycles $\{V^\bullet \in \overline{X}(w)\mid V^k \mbox{ is rational }\}.$

ii) Let $j=n-2$ and $w'=s_1\cdots \hat{s_i} \cdots s_{n-1}\in W$ for some $1\leq i \leq n-2$.  Using Poincar\'e duality we treat the 
equivalent situation  by considering the homomorphism $i_\ast:A_{n-2}(\overline{X}(w')) \to A_{n-2}(\overline{X}(w)).$ We supply the latter object with 
the basis given by
the cycles  $\{V^\bullet \in \overline{X}(w)\mid V^k \mbox{ is rational }\}$ labeled by $s_1s_2\cdots \hat{s_k} \cdots s_{n-1}$
for $1 \leq k \leq n-2.$ We realize the missing trivial representation by the cycle $f^\ast(\mP(H))$ labeled as usual by 
$s_1s_2\cdots s_{n-2}$. Here
$f:\overline{X}(w) \to \mP(V)$ is composite of all blow up maps, i.e., the projection map onto the first filtration step and $H\subset V$ is
a (rational) hyperplane. Hence we have a grading $A_{n-2}(\overline{X}(w))=\bigoplus_k A(w)_k.$

If $i<n-2$, then the map  $A_{n-2}(\overline{X}(w')) \to A_{n-2}(\overline{X}(w))$
is graded by the choice of bases. If $i=n-1$ then  $i_\ast= \oplus_{k} i_k$ where $i_k:A(w') \to A(w)_k$ is induced by Frobenius reciprocity
by the double coset of $e\in W$ multiplicated by $-1$.  The reason is that $f^\ast(\mP(H))$ coincides with  the cycle 
$$\overline{X(w)}_H + \sum_{i=1}^{n-2} \sum_{W\in \Gr_i(V)(k) \atop W \subset H} \overline{X(w)}_W ,$$
where $\overline{X(w)}_W:=\{V^\bullet \in \overline{X(w)} \mid V^i=W\}$ when $\dim W=i.$

If we want to have a graded morphism, then we simply use the bases on $A_{n-2}(\overline{X}(w))$ by using the approach
via  the dual projective space, cf. Remark \ref{rk_cox_chow} iii).

\end{egs}

To the end of this section we recall the definition of Deligne-Lusztig varieties attached to elements of the Braid monoid $B^+$ of $W$
and to the description of smooth compactifications of them, cf. \cite{DL,DMR,L4}.
The Braid monoid  $B^+$ is the quotient of ${F^+}$ given by the relations
$(st)^{m_{s,t}}=1$ where $s,t\in S$ with $s\neq t.$ Here $m_{s,t}\in \mZ$ is the order of the element $st\in W$. Thus we have surjections
$${F^+} \stackrel{\alpha}{\to} B^+ \stackrel{\beta}{\to} W$$ with $\gamma=\beta\circ \alpha$.
There is a section $W \hookrightarrow B^+$ of $\beta$ which identifies $W$ with the subset
$$B^+_{\rm red}=\{w\in B^+\mid \ell(w)=\ell(\beta(w))\}$$ of
reduced elements in $B^+$, cf. \cite{GKP}. In the sequel we will identify $W$ with $B^+_{\rm red}.$

For any element $w=s_{i_1}\cdots s_{i_r} \in {F^+},$ set
\begin{eqnarray*}X(w) & := & X(s_{i_1},\ldots,s_{i_r}) \\ & := & \big\{x=(x_0,\ldots, x_r) \in X^{r+1}
 \mid {\rm inv}(x_{j-1},x_{j})=s_{i_j}, j=1,\ldots, r,\; x_r=F(x_0)\big \}.
\end{eqnarray*}
This is a smooth variety over $k$ equipped with an action of $G$.
If $w\in W$  and $w=s_{i_1}\cdots s_{i_r}$ is a fixed reduced decomposition, then we also simply write $X^{{F^+}}(w)$
for $X(s_{i_1},\ldots,s_{i_r})$. For any $w\in W$, the  map
\begin{eqnarray}\label{isomorphism}
 X(s_{i_1},\ldots,s_{i_r}) & \to &  X(w) \\
\nonumber (x_0,\ldots,x_r) & \mapsto &  x_0
\end{eqnarray}
   defines  a $G$-equivariant  isomorphism of varieties over $k.$
Moreover by Brou\'e, Michel \cite{BM} and Deligne \cite{De} the variety $X^{{F^+}}(w)$ depends up to an unique isomorphism only on
the image of $s_{i_1}\cdots s_{i_r}$ in $B^+.$

Finally, we consider compactifications of the varieties $X(w)$. More generally, we associate to certain elements of the completed braid monoid 
$\underline{B}^+$ a DL-variety. Here we do not treat the general machinery as developed in \cite{DMR}, not to mention the definition
of $\underline{B}^+$, as we need later on only some of them.
Let $\hat{F}^+$  be the free monoid generated by the set $\hat{S}$ consisting of $S$ and of all reflections in $W$ of the shape $sts$
with $s,t\in S, st\neq ts \in W.$ In order
to distinguish the generator $sts \in \hat{S}$ where $s,t\in S$ and $st\neq ts \in W$ from the product $sts  \in F^+\subset \hat{F}^+ ,$  
we also write $\widehat{sts}$ for this element.

For any product $t=t_1\ldots t_r \in \hat{F}^+$ with $t_i\in \hat{S}$, we set
\begin{eqnarray*}
\overline{X}(t) & := & \overline{X}(t_1,\ldots, t_r)\\ &:= & \big\{x=(x_0,\ldots, x_r) \in X^{r+1} \mid {\rm inv}(x_{j-1},x_{j})
\leq t_j\},j=1,\ldots, r, x_r=F(x_0)\big\}.
\end{eqnarray*}
Again this is a $k$-variety with an action of $G.$

\begin{prop}\label{smooth_and_proj}
The variety $\overline{X}(t)$ is smooth and projective.
\end{prop}

\begin{proof}
See \cite[Prop. 2.3.5, 2.3.6]{DMR}.
\end{proof}

Let $w=s_{i_1},\ldots, s_{i_r}\in F^+.$ Then the variety $\overline{X}(w)$ includes $X(w)$ as an open subset so that we get a compactification of $X(w).$
More precisely, for all $v\preceq w$ we can identify $X(v)$ with a locally closed subvariety of $\overline{X}(w)$ and we get in this way
a stratification $\overline{X}(w)= \bigcup_{v\preceq w} X(v).$

Hence  if $w\in W$ and $w=s_{i_1} \cdots s_{i_r}$ is a reduced decomposition, then the
variety
$$\overline{X}^{F^+}(w):=\overline{X}(s_{i_1},\ldots, s_{i_r})$$ is a smooth compactification of $X(w).$ The map (\ref{isomorphism})
extends to a surjective proper birational  morphism
$$\pi:\overline{X}^{F^+}(w) \to \overline{X(w)}$$ which we call the Bott-Samelson or  Demazure resolution of $X(w)$ with respect to the reduced
decomposition.

\begin{rk}
 If $w$ is a Coxeter element, then the map $\pi$ is an isomorphism. In fact, this follows easily by considering the natural stratifications
on both sides.
\end{rk}

\begin{rk}\label{cohomology_not_full_F}
When $w\in F^+$ is not full, then the obvious variant of Proposition \ref{cohomology_not_full} does also hold true for the DL-varieties $X(w)$ and their 
compactifications
$\overline{X}(w)$, cf. \cite[Cor. 3.1.3]{DMR}.
\end{rk}

\vspace{0.5cm}

\section{Squares}

We consider the action of $W$ on itself by conjugation. The set of conjugacy classes $C$ in $W$ is in bijection with
the set of partitions $\P$ of $n.$ For a partition $\mu \in \P$, let $C^\mu$ be the corresponding conjugacy class.
Let $C_{\rm min}$ be the set of minimal elements in a given conjugacy class $C.$ 

\begin{coro}
Let $C$ be a conjugacy class and let $v,w\in C_{\rm min}.$ Then $H^\ast_c(X(v)) \cong H^\ast_c(X(w)).$
\end{coro}

\begin{proof}
By using the previous proposition, it is easily verified  that the statement is true for all Coxeter elements in $G$.
Hence by by Proposition \ref{cohomology_not_full} the same is true for all Coxeter elements in a fixed Levi subgroup.
On the other hand,  the Levi subgroups to the minimal elements $v$ and $w\in C_{\rm min}$  are conjugated by an element in $G$ which induces
the  isomorphism. 
\end{proof}


In order to deal with non-minimal elements we recall the following result of Geck, Kim and Pfeiffer.  Let $w,w' \in W$ and $s\in S$. Set
$w \stackrel{s}{\rightarrow} w'$ if $w'=sws$ and $\ell(w') \leq \ell(w)$. We write
$w \rightarrow w'$ if $w=w'$ or if there are elements $s_1,\ldots,s_r\in S$ and $w=w_1,\ldots, w_r=w'\in W$  with
$w_i \stackrel{s_i}{\rightarrow} w_{i+1}, \,i=1,\ldots,r-1.$

\begin{thm}\label{GKP}{\rm ((\cite{GKP}, Thm. 2.6)}
Let $C$ be a conjugacy class of $W$. For any $w\in C$, there exists some $w' \in C_{min}$ such that $w\rightarrow w'.$\qed
 \end{thm}

As a consequence, for any $w\in W$, which is not minimal in its conjugacy class there exists a finite set of cyclic shifts (i.e. elementary 
conjugations $w\to sws$, $s\in S)$ such that the resulting element  has the shape $sw's$ with $\ell(w)=\ell(w')+2.$

Let $s\in S$ and let $w,w'\in W$ with $w=sw's$. Suppose that $\ell(w)=\ell(w')+2$. We put
$$Z=X(w) \cup X(sw') \mbox{ and } Z'=X(w's) \cup X(w')$$
and
$$\tilde{Z}=X(w) \cup X(w's) \mbox{ and } \tilde{Z}'=X(sw') \cup X(w').$$

\begin{prop}(Deligne-Lusztig)\label{Prop_homeomorphic}
i) The varieties $X(w's)$ and $X(sw')$ are  universally homeomorphic by maps $\sigma:X(sw') \to X(w's)$ and 
$\tau: X(w's) \to X(sw')$
with $\tau \circ \sigma=F$ and $\sigma \circ \tau=F $. Hence $H^\ast_c(X(w's)) \cong H^\ast_c(X(sw'))$.

ii) The above maps extend to morphisms $\tau: Z' \to \tilde{Z}'$ and  $\sigma: \tilde{Z}'\to Z'$ with $\sigma|X(w')=id$ and $\tau|X(w')=F.$
\end{prop}

\begin{proof}
i) For later use we just recall the construction of the maps and refer for the proofs to \cite{DL}. Let $B\in X(sw').$ Then there is by
property (1) b) of the varieties $O(-)$ a unique
Borel subgroup $\sigma(B) \in X$ with $(B,\sigma (B)) \in \mcO(s)$ and $(\sigma B , F(B)) \in \mcO(w').$ An immediate computation  shows that
$(\sigma(B) ,F(\sigma B)) \in \mcO(w's)$, cf. \cite[Thm 1.6]{DL}.  The map $\tau$  is defined analogously.

ii) This follows easily from the definitions of the maps $\sigma,\tau.$
\end{proof}

The following statement bases on a observation made in Theorem 1.6 in \cite{DL} .

\begin{prop}
There is a $\mathbb A^1$-bundle  $\gamma : Z \to Z'$.
\end{prop}

\begin{proof}
As in the proof of loc.cit., we may write $X(w)$ as a (set-theoretical) disjoint union
$$X(w)=X_1 \cup X_2$$
where $X_1$ is closed in $X(w)$ and $X_2$ is its open complement
(Note that we have interchanged the role of $w$ and $w'$ compared to \cite{DL}).
Recall the definition of $X_i, i=1,2$. For $B\in X(w)$ there are unique Borel subgroups $\delta(B),\gamma(B) \in X$ with
$(B,\gamma(B)) \in \mcO(s)$, $(\gamma(B),\delta(B))\in \mcO(w')$ and $(\delta(B), F(B))\in \mcO(s).$  Then
$$X_1=\big\{B\in X(w) \mid \delta(B)=F(\gamma(B)) \big\} \mbox{ resp. } X_2=\big\{B \in X(w) \mid\delta(B)\neq F(\gamma(B)) \big\}.$$
In \cite{DL} it is shown that the map $\gamma:X_1 \to X(w')$ is a $\mathbb A^1$-bundle whereas $\delta: X_2 \to X(w's)$ is
a $\mathbb G_m$-bundle.

On the other hand, if  $B\in X_2$, then $(\delta(B), F(B))\in \mcO(s)$, $(F(B),F(\gamma B))\in \mcO(s), \delta(B) \neq F(\gamma B)$, hence
$(\delta(B), F(\gamma(B))) \in \mcO(s)$. This was already shown in \cite{DL}.
It follows that $(\gamma(B), F(\gamma(B)))\in \mcO(w's)$, hence $\gamma(B) \in X(w's).$ Thus we have a morphism
$\gamma:X(w) \to Z'$ of varieties compatible with the action of $G$ and $F.$

By Proposition \ref{Prop_homeomorphic} there is a homeomorphism
$\sigma: X(sw') \to X(w's)$. This map is compatible with $\gamma: X(w) \to Z'$
in the sense that both maps glue to a morphism
$\gamma: Z \to Z'.$ This map is clearly an $\mA^1$-bundle.
\end{proof}

\begin{coro}\label{cohomology_vb}
There is an isomorphism $H_c^i(Z)=H_c^{i-2}(Z')(-1)$ for all $i\geq 2.$ \qed
\end{coro}

\begin{rk}
 This corollary and the upcoming results Corollary \ref{cohomology_vb}, Corollary \ref{cohomology_pb}, Proposition \ref{hypersquares_surjective} and  Remark \ref{P1_closures} 
 were already proved in \cite[Prop. 3.2.10]{DMR}). Indeed, they consider the situation where  a cyclic shift is applied to $w=sw's.$
 \end{rk}

The map $\gamma$ even extends to a larger locally closed subvariety as follows.
We set $\hat{Z}:= Z \cup Z'.$

\begin{lemma}\label{subvariety}
The set $\hat{Z}$ is  locally closed in $X.$
\end{lemma}

\begin{proof}
For proving the assertion, it suffices to show (by considering topological closures of DL-varieties) that there is no
element $v\in W$ different from $sw'$ resp. $w's$ with $w' \leq v \leq w$
This proof is a consequence of  the following result.
\end{proof}

\begin{lemma}\label{square}
 Let $w_1,w_2 \in W$ with $w_1 \leq w_2$ and $\ell(w_2)=\ell(w_1)+2.$ Then there are uniquely determined elements $v_1,v_2 \in W$ with
$w_1 \leq v_i \leq w_2$ and $\ell(v_i)=\ell(w_1)+1.$
\end{lemma}

\begin{proof}
See \cite[Lemma 10.3]{BGG}.
\end{proof}

In the above situation, Bernstein, Gelfand, Gelfand call the quadruple $Q=\{w_1,v_1,v_2,w_2 \}$ a square in $W.$
Here we use sometimes the graphical illustration of \cite{Ku} (resp. for technical reasons sometimes without arrows) to indicate this kind of object:

$$
Q:\,\,\, \begin{array}{ccccc}
 & & w_2 & & \\ & \nearrow & & \nwarrow & \\ v_1 & & & & v_2  \\ & \nwarrow & & \nearrow & \\ & & w_1 & &
\end{array} \hspace{2cm} (\mbox{ resp. }
 \begin{array}{ccc}
   & w_2  & \\ v_1  & &  v_2  \\ & w_1  &
\end{array}).
$$

\medskip
Lemma \ref{subvariety} generalizes as follows.

\begin{lemma}
For any square  $Q=\{w_1,v_1,v_2,w_2 \}$ in $W$, the subset
$$X(Q):=X(w_2) \cup X(v_2) \cup X(v_1) \cup X(w_1)$$ is locally closed in $X$. \qed
\end{lemma}

For later use, we also mention the given-below property.

\begin{lemma}\label{vanishing_square}
For any square  $Q=\{w_1,v_1,v_2,w_2 \}  \subset W,$ let $$\delta^{i-1}_{w_1,v_1\cup v_2}: H^{i-1}_c(X(w_1)) \to H^i_c(X(v_1)\cup X(v_2))$$ and
$$\delta_{v_1\cup v_2, w_2}^i: H^i_c(X(v_1)\cup X(v_2)) \to H^{i+1}_c(X(w_2))$$ be the corresponding  boundary homomorphism.
Then $\delta_{v_1\cup v_2, w_2}^i\circ \delta_{w_1,v_1\cup v_2}^{i-1}=0.$
\end{lemma}

\begin{proof}
 This is clear as the map $\delta_{v_1\cup v_2, w_2}^i\circ \delta_{w_1,v_1\cup v_2}^{i-1}$ is just the composite of the corresponding differentials in the $E_1$-term associated to the stratification
 $X(Q)=X(w_2) \stackrel{\cdot}{\bigcup} \big(X(v_1) \cup X(v_2)\big) \stackrel{\cdot}{\bigcup} X(w_1).$
\end{proof}

Now we come back to the the locally closed subvariety $\hat{Z}\subset X.$

\begin{prop}
The map $\gamma$ extends to a $\mP^1$-bundle $\hat{Z} \to Z'$ with $\gamma|_{Z'} = id_{Z'}.$
\end{prop}

\begin{proof}
 This is a direct consequence of the definitions of $\gamma$ and the variety $Z'$ realising the latter space as the set
 $\{(B_0,B_1,B_2,B_3) \in X^4\mid (B_0,B_1)\in \mcO(e), (B_1,B_2)\in \mcO(w'), (B_2,B_3)\in \overline{\mcO(s)}, B_3=F(B_0)\}.$
\end{proof}

\begin{coro}\label{cohomology_pb}
There is an isomorphism of $H$-modules $$H_c^i(\hat{Z})=H_c^i(Z') \oplus H_c^{i-2}(Z')(-1)$$ for all $i\geq 0.$ \qed
\end{coro}

For the next statement, we consider the open subset $Y:=X(w) \cup X(sw') \cup X(w's)$ of $\hat{Z}.$

\begin{coro}\label{splitting}
There is a natural splitting $H_c^i(Y)=H_c^i(Z) \oplus H_c^{i}(X(w's))$ as $H$-modules  for all $i\geq 0.$
\end{coro}

\begin{proof}
The existence of a splitting is easily verified by considering the diagram of long exact cohomology sequences

\begin{equation*} \begin{array}{ccccccccc}
\cdots \to & H^i_c(Z) & \to &  H_c^{i}(\hat{Z}) & \to & H^i_c(Z') & \stackrel{\delta^i}{\to} & H^{i+1}_c(Z)  & \to \cdots \\
 & \lVert & & \uparrow & & \uparrow & & \lVert &  \\
\cdots \to &  H_c^{i}(Z) & \to &  H_c^{i}(Y) &  \to & H^{i}_c(X(w's)) &  \to & H^{i+1}_c(Z)  & \to \cdots \\
\end{array}
\end{equation*}
together with the fact that the differential map $\delta^i$ vanishes. That it is natural comes about from the fact
that the subset $U:=X_2(w)\cup X(sw')\cup X(w's)$ is open in $Y$ and we have a $\mP^1$-bundle $\gamma: U \to X(w's)$ with
$\gamma \circ i=\id$ where $i:X(w's) \hookrightarrow U$ is the inclusion.
 \end{proof}

In the sequel, we denote by $$r^i_{w,sw'}: H^i_c(Z) \to H^i_c(X(sw'))$$
the map which is induced by the closed immersion $X(sw') \hookrightarrow Z.$ We consider the corresponding long exact cohomology
sequence
$$\cdots \to H^{i-1}_c(X(sw')) \to H^i_c(X(w)) \to H^i_c(Z) \to H^i_c(X(sw')) \to \cdots $$
which by Corollary \ref{cohomology_vb} identifies with the sequence
\begin{eqnarray}\label{long_exact_coh_sequence}
 \cdots & \to &  H^{i-1}_c(X(sw')) \to H^i_c(X(w)) \to H^{i-2}_c(X(w's) \cup X(w'))(-1)  \\ \nonumber
 & \to & H^i_c(X(sw')) \to \cdots .
\end{eqnarray}

\medskip

\begin{rk}\label{contribute}
In \cite{DMR} it is proved that there is a long exact cohomology sequence
\begin{eqnarray*}
\cdots \to & H_c^{i}(X(w)) & \to H_c^{i-2}(X(w'))(-1)  \to   H_c^{i-1}(X(sw'))(-1) \oplus H^{i}_c(X(sw')) \\
\to & H_c^{i+1}(X(w)) & \to \cdots
\end{eqnarray*}
which relies on the fact that the natural maps $\delta^{i}: H^{i-2}_c(X(sw'))(-1) \to H^i_c(X(sw'))$ induced by the $\mG_m$-bundle $X_2$ over
$X(sw')$ are trivial (as already stated in \cite[Thm. 1.6]{DL}). In particular, it follows that the cokernel of the boundary map
$H^{i-3}_c(X(w')) \to H^{i-2}_c(X(w's))$ always contributes to $H^i_c(X(w)).$
\end{rk}

\begin{rk}\label{same_is_true_for_F}
 The same statements presented here (Prop. 4.4. - Cor. 4.12) are true if we work with elements in  ${F^+}$ instead  in $W$, cf. also \cite{DMR}.
 More precisely, if $w=s w' s$ for $w,w'\in {F^+}$ and $s\in S$, then we can define subsets $X_1, X_2 \subset X(w)$ such that
$X_1$ is a $\mA^1$-bundle over $X(w')$ and such that $X_2$ is an $\mG_m$-bundle over $X(w's).$ With the same reasoning, the subset
$X(w) \cup X(sw')$ is an $\mA^1$-bundle over $X(w's) \cup X(w'),$ etc.
\end{rk}

\begin{eg}\label{Example3}
We reconsider Example \ref{Example_1} (which is also discussed in \cite[ch. 4]{DMR}). So let  $w=sw's=(1,3)$ with $s=s_2=(2,3), w'=s_1=(1,2).$ We are going to determine the
cohomology of the DL-variety $X(w).$
The cohomology of $X(s_1s_2)$ resp. $X(s_2s_1)$
is given by Proposition \ref{cohomology_Coxeter} by
$$H^\ast_c(X(s_2s_1))=H^\ast_c(X(s_1s_2))= v^G_B[-2] \oplus v^G_{P_{(2,1)}}(-1)[-3] \oplus i^G_G(-2)[-4].$$
Furthermore we have $H^\ast_c(X(w'))= i^G_B/i^G_{P_{(2,1)}}[-1] \oplus i^G_{P_{(2,1)}}(-1)[-2].$
Now the variety $Z'=X(w's) \cup X(w')$ coincides with the set $\{V^\bullet \mid F(V^1) \subset V^2, V^1\neq F(V^1)\}$ which we may
identify with the open subset  $\mP(V)\setminus \mP(V)(k)$ of $\mP(V).$
Hence we obtain (which follows also by applying Proposition \ref{boundary_Coxeter})
$$H^\ast_c(Z') = v^G_{P_{(2,1)}}[-1] \oplus i^G_G(-1)[-2] \oplus i^G_G(-2)[-4]$$
and therefore  $$H^\ast_c(Z) = v^G_{P_{(2,1)}}(-1)[-3] \oplus i^G_G(-2)[-4] \oplus i^G_G(-3)[-6]$$
by Corollary \ref{cohomology_vb}. We claim that the maps $r^3_{w,sw'}, r^4_{w,sw'}$ are surjective. Indeed  for $i=4$
this is clear since $H^4_c(X(sw'))$ is the top cohomology group  of $X(sw').$
As for $i=3$ we consider the boundary  map $H^2_c(X(s)) \to H^3_c(X(sw'))$ which 
is surjective since $X(s) \cup X(sw')$ has the same cohomology as $Z'.$ Let $\tilde{Z}=X(s^2)\cup X(s)$ . Then the map $H^2_c(\tilde{Z}) \to H^2_c(X(s))$ is surjective,
as well, since the RHS is  the top cohomology degree and both varieties have the same number of connected components.
We consider the resulting commutative diagram
\begin{eqnarray*}
 H^3_c(Z) & \to & H^3_c(X(sw')) \\
 \uparrow & & \uparrow \\
 H^2_c(\tilde{Z}) &  \to & H^2_c(X(s)).
\end{eqnarray*}
It follows that the map $r^3_{w,sw'}:H^3_c(Z) \to H^3_c(X(sw'))$ is surjective.
Hence we get
$$H^\ast_c(X(w))=v^G_B[-3] \oplus i^G_G(-3)[-6].$$
\end{eg}

\vspace{0.5cm}

\section{Cohomology of DL-varieties of height one}

In this section we determine the cohomology of DL-varieties attached to Weyl group elements which are slightly larger than Coxeter elements,
i.e. to elements which are  of height one. For the  definition of the height function we recall that by Theorem \ref{GKP}  there is for any 
$w\in W$ some element $w'\in W$ with $\ell(w)=\ell(w')+2$ and $w \to w'.$

\begin{defn}
We define the height of $w$ inductively by $\h(w)=\h(w')+1.$ Here we set $\h(w)=0$ if
$w$ is minimal in its conjugacy class.
\end{defn}

In order to extend the definition of the height function to $F^+$, we use the following statement.

\begin{lemma}\label{doppelts}
Let $w\in B^+$ such that $\ell(\beta(w))<\ell(w)$, i.e. such that $w \in B^+\setminus W$. Then $w$  has the shape $w=w_1s s w_2$ for some $s\in S$ and
$w_1,w_2\in B^+.$
\end{lemma}

\begin{proof}
See \cite[Exercise 4.1]{GP}.
\end{proof}

\begin{defn}
i) Let $w\in B^+.$ We define the height inductively by

$$\h(w):=\left\{\begin{array}{cc}
\h(w) & \mbox{ if } w\in W \\ \\
 h(w_1w_2)+1 & \mbox{ if } w= w_1s s  w_2 \mbox{ is as above.}
 \end{array}\right.$$

\medskip
ii) For $w\in F^+$, we set $\h(w):=\h(\alpha(w)).$
\end{defn}

Thus we may write all elements in $B^+$ modulo cyclic shift  in the shape $w=sw's$ for some $s\in S$ and $w'\in B^+.$ 
The proof of the next statement is immediate.

\begin{lemma}
Let $w\in F^+$ and let $w_\min\in W$ be a minimal element lying in the conjugacy class of $\gamma(w).$ Then $\ell(w)=\ell(w_\min)+ 2\h(w).$ \qed
\end{lemma}

For any irreducible $H$-representation
$V=j_\mu(i)$, $\mu\in \P,i\in \mZ$, we set $t(V)=i$.

\begin{prop}\label{prop_misc}
Let $v,w\in F^+$, $i,j,m\in \mZ_{\geq 0}$ and suppose that $\h(v)=0.$ Let $V \subset H^i_c(X(w))$ be a subrepresentation such that
$V(m)\subset H^j_c(X(v)).$   Then $$\ell(w)-\ell(v) + m \geq i - j \geq \ell(w) -\ell(v) + m - \h(w).$$
\end{prop}

\begin{proof}
As $\h(v)=0$ we deduce by Proposition \ref{cohomology_Coxeter} and Proposition \ref{cohomology_not_full} that
$j=\ell(v)-t(V(m)).$  In a first step we may suppose that $V(m)$ sits in the top cohomology degree
of $X(v)$. Then $\ell(v)=-t(V(m)).$ As any unipotent representation is realized in $H^0(X(e))$ we may assume that $v=1$ and therefore $j=0.$

We start with the case where $\h(w)=0.$  In this case one has even - by looking again at Proposition \ref{cohomology_Coxeter} -
the stronger identity $$i= \ell(w) + m.$$

\noindent Now let $\h(w)\geq 1$ and suppose that $w=sw's.$ Consider the long exact cohomology sequence (\ref{long_exact_coh_sequence}). We distinguish the following cases:

\noindent Case a) Let $V\subset H^{i-1}_c(X(sw'))$. By induction on the length we deduce that $\ell(sw')+m \geq i-1 \geq   \ell(sw') + m -\h(sw').$
As $\ell(sw')=\ell(w)-1$ and $\h(w) \geq \h(sw')$, we see that
$\ell(w)+m \geq i \geq \ell(w) + m -h(w).$

\noindent Case b) Let $V\subset H^i_c(Z).$ Then we must have $m\geq 1.$

Subcase i)  Let $V(1) \subset H^{i-2}_c(X(sw')).$  By induction on the length we deduce that $\ell(sw') + m-1 \geq i-2 \geq
\ell(sw') + m-1 - h(sw').$ As $\ell(sw')=\ell(w)-1$ and $\h(w) \geq \h(sw')$, we see that
$\ell(w)+m \geq i \geq \ell(w) + m -\h(w).$

Subcase ii) Let $V(1) \subset H^{i-2}_c(X(w')).$  By induction on the length we deduce that $\ell(w')+m-1 \geq i-2  \geq   \ell(w') + m-1-\h(w').$
As $\ell(w')=\ell(w)-2$ and $h(w)=h(w')+1$, we see that even the stronger identity
$\ell(w)+m-1 \geq i\geq  \ell(w)  + m -h(w)$ holds true.
\end{proof}

\begin{coro}
Let $v,w\in W$ and $i,j,m \in\mZ_{\geq  0}.$ Let $V \subset H^i_c(X(w))$ be a subrepresentation such that
$V(m)\subset H^j_c(X(v)).$ Then $$\ell(w) -\ell(v) + m + \h(v)\geq i - j \geq \ell(w) -\ell(v) + m - \h(w).$$
\end{coro}

\begin{proof}
We apply the foregoing proposition where $w$ is replaced by $v$ and $v$ by $1.$ Then $\ell(v)+t(V(m)) \geq j \geq \ell(v) + t(V(m)) - \h(v).$
Multiplying this term with $-1$ and adding the result to the sequence of inequalities $\ell(w)+t(V) \geq i \geq \ell(w) +t(V) - \h(w)$
gives the statement.
\end{proof}

For $w'\in W$ with $w' \leq w$ and $\ell(w')=\ell(w)-1,$ the set $Z'=X(w) \cup X(w')$ is a subvariety of $X$ as already observed above.
Here $X(w')$ is closed and $X(w)$ is open in $Z'.$ We denote by
$$\delta^\ast_{w',w}: H_c^\ast(X(w')) \to H_c^{\ast+1}(X(w))$$
the associated boundary map.

\begin{prop}\label{boundary_Coxeter}
Let $w$ be a Coxeter element and let $w'\in W$ with $w'\leq w$ and $\ell(w')=\ell(w)-1.$
Then the boundary homomorphism $\delta_{w',w}^j: H_c^{j}(X(w')) \to H_c^{j+1}(X(w))$ is surjective for all $j \leq 2\ell(w')=2(n-2) .$
(In particular, $\delta^j_{w',w}$ is si-surjective for all $j=0,\ldots, 2\ell(w)=2(n-1).$)
\end{prop}

\begin{proof}
We may suppose that $w=\cox_n.$  Since all the  representations $H_c^i(X(w))\neq (0)$ are irreducible, it suffices to show
that the boundary maps $\delta_{w',w}^i$ for $i<2\ell(w)-1,$ are non-trivial.
Let $w'=s_1\cdots \hat{s_i} \cdots s_h$ be as above. In terms of flags the DL-varieties in question have the following description
\begin{eqnarray*}
X(w) & = &\big\{V^\bullet \mid F(V^j) \subset V^{j+1}, F(V^j) \neq V^j, \;  1\leq j \leq n-1 \big\}, \\
& & \\
X(w')& = & \big\{V^\bullet \mid F(V^j) \subset V^{j+1}, F(V^i)=V^i, F(V^j) \neq V^j, \;  1\leq j\neq i \leq n-1  \big\}.
\end{eqnarray*}
Their Zariski closures are given by
\begin{eqnarray*}
\overline{X(w)} & = & \big\{V^\bullet \mid F(V^j) \subset V^{j+1}, \;  1\leq j \leq n-1 \big\}, \\
& & \\
\overline{X(w')}& = & \big\{V^\bullet \mid F(V^j) \subset V^{j+1}, F(V^i)=V^i,  \;  1\leq j\neq i \leq n-1  \big\}.
\end{eqnarray*}

The complement of $X(w)$ in $\overline{X(w)}$ is a divisor $D=\bigcup_W D_W$ where the union is over all $k$-rational subspaces $W$  of $V$.
For any rational flag $W^\bullet= (0)\subsetneq W^{i_1} \subsetneq W^{i_2} \subsetneq \cdots \subsetneq W^{i_k} \subsetneq V$ of $V$, we set
$D_{W^\bullet}= D_{W^{i_1}}\cap D_{W^{i_2}} \cap \cdots \cap D_{W^{i_k}}$
and ${\rm lg}(W^\bullet)=k.$ This construction gives rise for any constant sheaf $A$ on $\overline{X(w)}$ to a resolution
$$A \to  \bigoplus_{W} A_{D_W} \to \bigoplus_{W^\bullet, {\rm lg}\,(W^\bullet)=2} A_{D_{W^\bullet}} \to \cdots \to
\bigoplus_{W^\bullet,\, {\rm lg}(W^\bullet)=n-1} A_{D_{W^\bullet}}.$$
of $A_{X(w)}.$ On the other hand, we have $\overline{X(w')}=\bigcup_{W\in {\rm Gr}_i(V)(k)} D_W.$ Similarly as above, we get
a resolution
$$A_{\overline{X(w')}}\to \bigoplus_{\genfrac{}{}{0pt}{2}{W^\bullet, {\rm lg}\,(W^\bullet)=2}{W^i \in W^\bullet}} A_{D_{W^\bullet}} \to
\bigoplus_{\genfrac{}{}{0pt}{2}{W^\bullet, {\rm lg}\,(W^\bullet)=3}{W^i \in W^\bullet}} A_{D_{W^\bullet}} \to \cdots \to
\bigoplus_{\genfrac{}{}{0pt}{2}{W^\bullet,\, {\rm lg}(W^\bullet)=n-1}{W^i \in W^\bullet}} A_{D_{W^\bullet}}$$
of $A_{X(w')}.$
The second complex is a subcomplex of the first one and this inclusion induces just the boundary map.
In other terms, applying $H^{2i}(-)$ to both resolutions (strictly speaking to injective resolutions of $A=\mZ/l^n\mZ$, $n\in \mN)$), we just get the complexes
\begin{equation*}
 H^{2i}(\overline{X(w)})\to \bigoplus_{v < w  \atop \ell(v)=\ell(w)-1} H^{2i}(\overline{X(v)})\to \cdots \to
\bigoplus_{v < w  \atop \ell(v)=1} H^{2i}(\overline{X(v)})\to  H^{2i}(\overline{X(e)})
\end{equation*}
and
\begin{equation*}
 H^{2i}(\overline{X(w')})\to \bigoplus_{v < w'  \atop \ell(v)=\ell(w')-1} H^{2i}(\overline{X(v)})\to \cdots \to
\bigoplus_{v < w'  \atop \ell(v)=1} H^{2i}(\overline{X(v)})\to  H^{2i}(\overline{X(e)}).
\end{equation*}

If $w'=s_1s_2\cdots s_{n-2}$, then $X(w')\cong \coprod_H\Omega(H)$ with $H$ running through all rational hyperplanes in $V=\mF^n$.
Further we may identify $X(w)$ with $\Omega(V)\subset \mP(V)$.
Here the result is well-known in the setting of period domains. In fact, by considering also the varieties $\Omega(E)$ with $E$ a rational
subspace of $V$, we get a stratification of the projective space $\mP(V)$. Then the result follows by weight reasons and the cohomology formula
in Proposition \ref{cohomology_Coxeter} with respect to
the varieties $\Omega(E).$ Alternatively, one might use the fundamental complex in \cite{O}. By symmetry the same reasoning applies to
$w'=s_2s_3\cdots s_{n-1}.$

In general we distinguish the cases whether $j=2\ell(w')$ or $j<2\ell(w')$.
Suppose first that $j=2\ell(w').$ Here the claim follows by Example \ref{Example_cycle} ii) since
the contribution $H^{j}(\overline{X}(w'))$ does not lie in the image of the map
$H^{j}(\overline{X(w)})\to \bigoplus_{v < w  \atop \ell(v)=\ell(w)-1} H^{j}(\overline{X(v)}).$

If $j<2\ell(w')$ then we argue as follows.  Let $v=s_1s_2\cdots s_{n-2}$ and $v'={\rm gcd}(w',v)=s_1\cdots \hat{s_i} \cdots s_{n-2}.$
By induction on $n$ the map $H^{j-1}_c(X(v')) \to H^{j}_c(X(v))$ is surjective. On the other hand, by what we have
observed above the map $H^{j}_c(X(v)) \to H^{j+1}_c(X(w))$ is surjective, as well.
Using Lemma \ref{vanishing_square} we deduce that the map $H^{j}_c(X(w')) \to H^{j+1}_c(X(w))$ is non-trivial.
\end{proof}

The next two statements give the cohomology of all Weyl group elements having full support and which are of height 1.
Arbitrary elements of height one are handled by Proposition \ref{cohomology_not_full}.

\begin{prop}\label{height_one}
Let   $w=sw' s$ where $w'\in W$ is a Coxeter element in some Levi subgroup  of a proper maximal parabolic subgroup in $G$ .
Then the maps $r^j_{w,sw'}: H^j_c(Z) \to H^j_c(X(sw'))$ are all surjective for $j> \ell(sw')=n-1$.
(In particular they are si-surjective for all $j\geq 0$.)
\end{prop}

\begin{proof}
We start with the observation that $sw'$ and $w's$ are  both Coxeter elements in $W$. We may suppose that
$w'=s_1\cdots \hat{s_i}\cdots s_{n-1}$ and $s=s_i.$ It is clear that
$v^G_B \not\in \supp H^{n-1}_c(Z).$
So let $j>n-1$ and suppose that $r^j_{w,sw'}$ is not surjective. Then the irreducible module $H^j_c(X(sw'))$ maps injectively into
$H^{j+1}_c(X(w))$ via the boundary homomorphism $\delta^j_{sw',w}.$ First let $i<n-1.$
Set $$w'':=s_i s_1s_2 \cdots \hat{s_i}\cdots  s_{n-2}=s_iw's_{n-1}.$$
This is a Coxeter element in the parabolic subgroup $W_{(n-1,1)}$ of $W$ with $w''\leq s_iw'$ and $\ell(w'')=\ell(w').$ Consider the square

\hspace{2cm} $
 \begin{array}{ccccc}
 & & w & & \\ & \nearrow & & \nwarrow & \\ s_iw' & & & & w''s_i . \\ & \nwarrow & & \nearrow & \\ & & w'' & &
\end{array}
$

\noindent The boundary homomorphism
$\delta_{w'',s_iw'}^{j-1}: H^{j-1}_c(X(w'')) \to H^j_c(X(s_iw'))$ is si-surjective by  Proposition \ref{boundary_Coxeter}.
On the other hand, the boundary map $\delta_{w'',w''s_i}^{j-1}: H^{j-1}_c(X(w'')) \to H^j_c(X(w''s_i))$ vanishes as the map
$r^j_{w''s_i,w''}: H^{j-1}_c(X(w''s_i)\cup X(w'')) \to H^{j-1}_c(X(w''))$ is (si-)surjective by induction on $n$.
Indeed, both elements $w'',w''s_i$ are of the shape above and in the Weyl group of $\GL_{n-1}.$ The start of induction is 
given by Example \ref{Example3}.
By Lemma \ref{vanishing_square} the composite $\delta^j_{s_iw',w}\circ \delta_{w'',s_iw'}^{j-1}$
vanishes, a contradiction. The result follows in this special case.

If $i=n-1$, then we set $w'':=s_1 w' s_{n-1}$ and consider $w' s_{n-1}$ instead of $s_{n-1}w'$ and $s_{n-1} w''$
instead of $w'' s_{n-1}.$ Then the same argument goes through.
\end{proof}

By Proposition \ref{boundary_Coxeter} we are able to give a formula for the cohomology of these height 1 elements.
Here we could give the description of the induced representation
$H^\ast_c(X(w'))=\Ind^G_{P_{(i,n-i)}}(H^\ast_c(X_{{\bf M}}(w'))$ by using Littlewood-Richardson
coefficients, cf. \cite[\S A]{FH}
(Note that the structure or combinatoric of unipotent $G$- and $W$-representations is the same, cf. Remark \ref{connection_symmetric group}).
Instead we prefer to use the notation which is common in the Grothendieck group of $G$-representations. Hence if we write $V-W$ for two
$G$-representations $V,W$, then we mean implicitly that $W$ is a subrepresentation of $V$.

\begin{coro}\label{coro_height1_a} In the situation of the foregoing proposition, we have for $j\in \mathbb N$, with $\ell(w)<j<2\ell(w)-1,$
$$H^j_c(X(w)) = \big(H^{j-2}_c(X(w')) - j_{(j+1-n,1,\ldots,1)}(n-j)\big)(-1) -
j_{(j+2-n,1\ldots,1)}(n-j-1).$$
Moreover, we have $H^{\ell(w)}_c(X(w)) = v^G_B \bigoplus \big(v^G_{P(s)} - j_{(2,1\ldots,1)}\big)(-1),\;\; H^{2\ell(w)-1}_c(X(w))=0$ and $H^{2\ell(w)}_c(X(w))=i^G_G(-\ell(w)).$
\end{coro}

\begin{proof}
By Proposition \ref{height_one}, we deduce that $H^j_c(X(w))=\ker\big(H^j_c(Z) \to H^j_c(X(sw'))$ for all $j>\ell(w)=n.$
By Proposition \ref{cohomology_Coxeter} we have $H^j_c(X(sw'))=j_{(j+2-n,1\ldots,1)}(n-j-1).$
Further $H^j_c(Z)=H^{j-2}_c(X(w's) \cup X(w'))(-1)$
and the boundary map $$H^{j-2}_c(X(w')) \to H^{j-1}_c(X(w's))=j_{(j+1-n,1,\ldots,1)}(n-j)$$ is surjective for $j-2\leq 2\ell(w')=2\ell(w)-4$ by
Proposition \ref{boundary_Coxeter}. Hence we get the first identity in the statement.

If $j=\ell(w)=n$ then one verifies easily that $$H^{j-2}_c(X(w')) - j_{(j+1-n,1,\ldots,1)}(n-j)=H^{\ell(w')}_c(X(w'))-v^G_B=v^G_{P(s)}.$$
In addition the Steinberg representation appears as a summand in $H^{n}_c(X(w))$. It is induced via the boundary map 
$H^{\ell(w's)}_c(X(w's)) \to H^{\ell(w)}_c(X(w))$
from $H^{\ell(w)-1}_c(X(w's))=H^{\ell(w's)}_c(X(w's))=v^G_B.$

The remaining identities for $j=2\ell(w)-1,2\ell(w)$ are easily verified in the same way.
\end{proof}

\begin{rk}\label{vanishing_coh}
Let  $w\in W$ have full support. Then we always have $H^{\ell(w)}_c(X(w))\supset v^G_B=j_{(1,\ldots,1)}$ and $H^{2\ell(w)}_c(X(w))=v^G_G(-\ell(w))=j_{(n)}(-\ell(w)),$
cf. \cite[Prop. 1.22]{L2}, \cite[Prop. 3.3.14, 3.3.15]{DMR}. More precisely, these are the only cohomology degrees where these extreme
unipotent representations appear. Further $H^i_c(X(w))=0$ for all $i<\ell(w).$
\end{rk}

\begin{eg}\label{Example_(1,3,4)}
Let $n=4$, $w=(1,2)(2,3)(3,4)(1,2)\in W.$ Then $$H_c^\ast(X(w))=v^G_B[-4] \oplus j_{(2,2)}(-2)[-5] \oplus i^G_G(-4)[-8].$$
\end{eg}

\smallskip
\begin{eg}\label{Example_(1,3)(2,4)}
Let $n=4$, $w=(2,3)(1,2)(3,4)(2,3)\in W.$ Then
\begin{eqnarray*}
H_c^\ast(X(w)) & = & v^G_B[-4] \oplus j_{(2,2)}(-1)[-4] \oplus j_{(2,1,1)}(-2)[-5] \\
& & \oplus j_{(3,1)}(-2)[-5] \oplus j_{(2,2)}(-3)[-6] \oplus i^G_G(-4)[-8].
\end{eqnarray*}
\end{eg}

The remaining elements with full support and which are of height 1 are treated by the next result.

\begin{coro}\label{coro_height1_b}
Let   $w=s w'  s\in W$ with $\ell(w)=\ell(w')+2$ for some Coxeter element $w'\in W$ and $s\in S.$ Then the map
$r^j_{w,sw'}: H^j_c(Z) \to H^j_c(X(sw'))$
vanishes for all $j\neq 2\ell(w)-2$ and is an isomorphism for $j=2\ell(w)-2$. Hence we have
$$H^j_c(X(w))=H^{j-2}_c(X(w' s)\cup X(w'))(-1) \oplus H^{j-1}_c(X(s w'))$$ for all $j\neq 2\ell(w)-1, 2\ell(w)-2$
and  $H^{2\ell(w)-1}_c(X(w))=H^{2\ell(w)-2}_c(X(w))=0.$
\end{coro}

\begin{proof}
By Prop. \ref{height_one} the boundary map $H^{j-2}_c(X(w')) \to H^{j-1}_c(X(w's))$ vanishes for all $j\in \mN$ with $j-2\neq \ell(w')=n-1.$
If $j-2=\ell(w')$, then it is an injection, since on the LHS we have  the Steinberg representation $v^G_B.$
Hence
\begin{eqnarray*}
H^j_c(X(s w' s) \cup X(s w')) & \cong & H^{j-2}_c(X(w' s)\cup X(w'))(-1) \\ & = &  H^{j-2}_c(X(w' s))(-1)\oplus H^{j-2}_c(X(w'))(-1)
\end{eqnarray*}
for all $j>\ell(w')+2=n+1.$
By Remark \ref{contribute} we know that  $r^j_{w,sw'}$ applied to a contribution of $H^{j-2}_c(X(w' s))(-1)$ vanishes.
On the other hand, by comparing weights we see that the Tate twist of a contribution in $H^j_c(X(s w' s) \cup X(s w'))$
induced by $H^{j-2}_c(X(w'))(-1)$ is different from the Tate twist of $H^j_c(X(sw')$,
except for $j=2\ell(w)-2.$ Here we have the trivial representation on both sides. The result follows.
\end{proof}

\begin{eg}
Let $n=4$, $w=(2,3)(1,2)(2,3)(3,4)(2,3)\in W.$ Then one verifies that
\begin{eqnarray*}
 H_c^\ast(X(w)) &=& v^G_B[-5] \oplus j_{(2,2)}(-2)[-6] \oplus
j_{(2,1,1)}(-2)[-6]\\ & & \oplus j_{(3,1)}(-3)[-7] \oplus j_{(2,2)}(-3)[-7] \oplus  i^G_G(-5)[-10].
\end{eqnarray*}
\end{eg}

\vspace{0.5cm}
Although the previous results do not apply directly to the element $w=(1,4)\in W$, we are able to compute the cohomology
of  $X((1,4)).$ Indeed, we write $w=sw's$ with $w'=(1,3)$ and $s=(3,4).$ 
The cohomology groups of $X((1,3))$ and $X((1,3)(3,4))$ behave disjointly, cf. Examples \ref{Example3}, 
\ref{Example_(1,3,4)}. We deduce that 
\begin{eqnarray*}
H_c^\ast(X(w')\cup X(w's)) & = & j_{(2,1,1)}[-3] \oplus j_{(2,2)}(-2)[-5] \oplus j_{(3,1)}(-3)[-6] \\ & \oplus &   i^G_G(-3)[-6] 
 \oplus  i^G_G(-4)[-8]
\end{eqnarray*}
and so
\begin{eqnarray*}
H_c^\ast(X(w) \cup X(sw')) & = &  j_{(2,1,1)}(-1)[-5] \oplus j_{(2,2)}(-3)[-7] \oplus j_{(3,1)}(-4)[-8] \\ & \oplus &  i^G_G(-4)[-8] \oplus  i^G_G(-5)[-10].
\end{eqnarray*}
But these groups behave again disjointly (apart from the top cohomology) from those of $X(sw').$ Hence we get
\begin{eqnarray*}
 H_c^\ast(X(w)) &=& v^G_B[-5] \oplus j_{(2,1,1)}(-1)[-5] \oplus j_{(2,2)}(-2)[-6] \oplus
\\ & & \oplus j_{(2,2)}(-3)[-7] \oplus j_{(3,1)}(-4)[-8] \oplus  i^G_G(-5)[-10].
\end{eqnarray*}

For  determining the cohomology of the DL-variety attached to the longest element in the Weyl group of $\GL_4$ we refer to the appendix. 

\vspace{0.5cm}

\section{Hypersquares}

Here we generalize some of the results of the previous section to hypersquares.

For elements $v,w\in W$ with $v\leq w$, we let $I(v,w)=\{z\in W \mid v \leq z \leq w \}\subset W$ be the interval between $v$ and $w$. Analogously
we define $I^{F^+}(v,w)=I(v,w)$ for $v,w\in F^+.$ Note that if we have fixed reduced decompositions of $v,w\in W$, the set $I^{F^+}(v,w)$ is in
general not compatible with $I(v,w)$ in the sense that $\gamma(I^{F^+}(v,w))=I(v,w)$.  Further we set for any interval $I=I(v,w)$,
$$\head(I)=w \mbox{ and } \tail(I)=v.$$
\begin{defn}
Let $v \leq w \in W$ with  $\ell(w)-\ell(v)=d.$ We say that $I(v,w)$ is a hypersquare of dimension $d$ in $W$
if  $$\#\big\{z \in I(v,w) \mid \ell(z)=\ell(w)-i\big\} ={d \choose i }$$ for all $1\leq i \leq d.$
\end{defn}
If $I(v,w)$ is a hypersquare, then $\#I(v,w)=2^d$ (the converse is also true). In this case we also write $Q(v,w)=I(v,w).$

The definition of a hypersquare in ${F^+}$ is similar but easier in the sense that for all $v,w\in F^+$ with $v \preceq w$
the cardinality of $I({v,w})$  is always $2^{\ell(w)-\ell(v)}.$

\begin{defn}
Let $v \preceq w\in F^+$ with  $\ell(w)-\ell(v)=d.$ The associated hypersquare of dimension $d$ in ${F^+}$ is given by the set
$Q(v,w)=I(v,w)$.
\end{defn}

If we consider $v,w\in W$ with $v\leq w$ and with fixed reduced decompositions $w=s_{i_1}\cdots s_{i_r}$ and $v=s_{j_1}\cdots s_{j_s}$, then we also
write $Q^{F^+}(v,w)$ for $Q(s_{j_1}\cdots s_{j_s},s_{i_1}\cdots s_{i_r}).$
For a hypersquare resp. interval $I=I(v,w)$ of $W$ (resp.  $F^+$), let
$$X(v,w):=X(I):= \bigcup_{w\in I} X(w)$$
be the induced locally closed subvariety of $X$ (resp. of $X^{\ell(w)+1}$ where $w=\head(I)$).
In particular, for $w\in F^+$ the compactification $\overline{X}(w)$ of $X(w)$ can be rewritten as
$$\overline{X}(w)=X(Q(1,w)).$$

\begin{lemma}
 Let $Q \subset W$ (resp. $Q\subset F^+)$ be a hypersquare. Then the variety $X(Q)$ is smooth.
\end{lemma}

\begin{proof}
If $Q\subset F^+$, then the claim follows from Proposition \ref{smooth_and_proj} since $X(Q)$ is an open subset of $\overline{X}(\head(Q)).$
If $Q\subset W$, then the claim follows from \cite[Theorem 6.2.10]{BL}. Indeed, let $d:=\dim Q$. By the rigidity of $Q$ it has to coincide with the 
Bruhat graph
$B(\tail(Q),\head(Q)).$ Then loc.cit. says that that the Schubert type analogue $X^{\rm Sch}(Q):=\bigcup_{w\in Q} X^{\rm Sch}(w)$
(where $X^{\rm Sch}(w)$ is the Schubert cell to $w$)
is (rationally) smooth if each vertex in the Bruhat graph has exactly $d$ edges. But each vertex in $Q$ has already by definition $d$ edges.
Now the argumentation is completely analogous as for ordinary DL-varieties. In fact, next
we deduce that $\mathcal{O}(Q):=\bigcup_{w\in Q} \mathcal{O}(w)\subset X\times X$ is smooth. Since $\mathcal{O}(Q)$ is transversal to 
the graph of the Frobenius, we see that $X(Q)$ is smooth.
\end{proof}

\begin{defn}\label{special_square}
A square $Q\subset W$ (resp. $Q\subset F^+$) is called special if it has the shape $Q=\{sw's,sw',w's,w'\}$ for some $w,w' \in W$, $s\in S$ (resp. $F^+$)
with $\ell(w)=\ell(w')+2.$  In this case we also write $Q_w=Q_{w,s}=Q.$
\end{defn}

The generalization of Proposition \ref{cohomology_vb} is given by the next result.

\begin{prop}\label{hypersquares_surjective}
Let $Q'=Q(v',w')\subset W$ be a hypersquare of dimension $d.$ Suppose that for $s\in S$, the sets $sQ', Q's$ and $Q:=sQ's$ are hypersquares,
 as well, and that $\ell(sw's)=\ell(w')+2$, $\ell(sv's)=\ell(v')+2.$ Then $X(Q) \cup X(sQ')$ is an $\mA^1$-bundle over $X(Q's) \cup X(Q').$ Consequently,
$$H_c^i(X(Q) \cup X(sQ')) \cong H^{i-2}_c(X(Q's) \cup X(Q'))(-1).$$
Moreover $X(Q)\cup X(sQ') \cup X(Q's) \cup X(Q')=X(Q(v',sw's))$ is a $\mP^1$-bundle over $X(Q's) \cup X(Q')$.
\end{prop}

\begin{proof}
 The claim follows easily from the fact the hypersquare $Q \cup sQ' \cup Q's \cup Q'$ is paved by special squares
 together with Corollary \ref{cohomology_pb}.
\end{proof}

\begin{rk}\label{P1_closures}
The same statement is true if we work in ${F^+}$ where the assumptions are automatically satisfied. In particular,
if $w=sw's\in F^+,$ then $$H^i(\overline{X}(w))=H^i(\overline{X}(w's)) \oplus H^{i-2}(\overline{X}(w's))(-1)$$
for all $i \geq 2$ and $H^0(\overline{X}(w))=H^0(\overline{X}(w's)).$ Analogously, we have
$$H^i(\overline{X}(w))=H^i(\overline{X}(sw')) \oplus H^{i-2}(\overline{X}(sw'))(-1)$$
for all $i \geq 2$ and $H^0(\overline{X}(w))=H^0(\overline{X}(sw')).$

This statement is already proved  more generally in \cite[Prop. 3.2.3]{DMR}. In fact, the notion of a hypersquare
can be expressed by using elements of the completed braid monoid $\underline{B}^+.$
\end{rk}

For later use we mention the following statement. Recall that  we denote for any $\Gamma$-module $V$ and any integer $i$,
by $V\langle i \rangle$ the eigenspace of the arithmetic  Frobenius  with eigenvalues of absolute value $q^i.$

\begin{lemma}\label{vanishing_odd_degree}
Let $w=sw's\in F^+$ with $\h(sw')=0.$ Then $H^{2i+1}_c(X(s^2,w))\langle - i \rangle=0.$
\end{lemma}

\begin{proof}
For proving the assertion  we may assume that $w$ is full.
If $\ell(w)>\ell(\gamma(w))$ then  $s$ commutes with every simple reflection in $w'$. Hence the variety $X(s^2,w)$ is 
homeomorphic to $X(s^2) \times \overline{X}(w')$. One computes easily that
$$H^\ast_c(X(s^2))=i^G_B/i^G_{P(s)}[-2] \bigoplus i^G_{P(s)}(-2)[-4].$$
Further  the cohomology of $\overline{X}(w')$ vanishes in odd degree by Proposition 
\ref{cohomology_Coxeter_compactification}, thus we get
$H^{2i+1}_c(X(s^2,w))=(0)$ by the K\"unneth formula.

So let $\ell(w)=\ell(\gamma(w))$ and suppose that  $V=j_\lambda(-i)\subset H^{2i+1}_c(X(s^2,w))\neq 0.$ 
Suppose first that $V$ is induced by $H^{2i+1}_c(X(w)).$ 
By Corollary \ref{coro_height1_a} and  by Proposition \ref{prop_misc}  it follows that $i=\ell(w)-2$.

\noindent {\it 1. Case:} $w=s_1s_2 s_3\cdots s_{n-1}s_1$ (or $w=s_{n-1}s_1s_2\cdots s_{n-2}s_{n-1}$ etc.) i.e., $s=s_1$ or $s=s_{n-1}.$

By Corollary \ref{coro_height1_a}  we conclude that $j_\lambda=j_{(n-2,2)}.$ We consider the square
$$Q=\{w,sv'_1s,sv_2's,su's\}\subset F^+$$ with
\begin{eqnarray*}
v_1'& = & s_2s_3\cdots s_{n-3}s_{n-2}, \\
v_2' & = & s_2s_3\cdots s_{n-3} s_{n-1},\\
u' & = & s_2s_3\cdots s_{n-3}.
\end{eqnarray*}
Now $H^{2i}_c(X(su's))=H^{2\ell(su's)}_c(X(su's))=i^G_ {P(su's)}(-i)=i^G_{P_d}(-i)$ with $d=(n-2,1,1)\in \D.$ 
For proving our claim in this special situation, it is enough to see that the 
$j_{(n-2,2)}$-isotypic part in
$H^{2i+1}_c(X(Q))$ vanishes.
For this we consider the boundary map $H^{2i}_c(X(sv_2's) \cup X(su's)) \to H^{2i+1}_c(X(sv_1's)\cup X(w))$
and moreover the extended boundary map
$$H^{2i}_c(X(sv_2's) \cup X(su's) \cup X(sv_2') \cup X(su')) \to H^{2i+1}_c(X(w)\cup X(sv_1's) \cup X(sv_1') \cup X(sw'))$$
which identifies by Proposition \ref{hypersquares_surjective} with
$$H^{2i-2}_c(X(v_2's)\cup X(u's) \cup X(v_2') \cup X(u'))(-1) \to H^{2i-1}_c(X(w's)\cup X(v_1's) \cup X(v_1') \cup X(w'))(-1).$$
By weight reasons we deduce that
$$H^{2i}_c(X(sv_2's) \cup X(su's))\subset H^{2i}_c(X(sv_2's) \cup X(su's) \cup X(sv_2') \cup X(su')).$$
On the other hand, we have $V\not\subset H^{2i}_c(X(sv_1') \cup X(sw'))$ as $H^{2i}_c(X(sv_1'))=i^G_{P_{(n-1,1)}}(-i)$ and 
$H^{2i}_c(X(sw'))=j_{(n-2,,1,1)}(-i)$ by
Proposition  \ref{cohomology_Coxeter}. Hence it suffices to see that $V$ does not appear in the cokernel of the extended boundary map.
Now $$V(-1) \subset H^{2i-2}_c(X(u's))=i^G_{P(u's)}(-i+1)=i^G_{P_{(n-2,1,1)}}(-i+1)$$ resp.
$$V(-1)\subset  H^{2i-2}_c(X(v_2'))=i^G_{P(v_2')}(-i+1))=i^G_{P_{(1,n-3,2)}}(-i+1).$$
On the other hand,
$$V(-1)\subset  H^{2i-1}_c(X(w'))=i^G_{P_{(1,n-2,1)}}/i^G_{P_{(1,n-1)}}(-i+1),$$
$$V(-1)\subset  H^{2i-1}_c(X(v_1's))=i^G_{P_{(n-2,1,1)}}/i^G_{P_{(n-1,1)}}(-i+1)$$
and
$$V(-1)\subset  H^{2i-2}_c(X(v_1'))=i^G_{P_{(1,n-2,1)}}(-i+1).$$
The result follows now easily by intertwining arguments as the contribution  $V(-1) \subset H^{2i-2}_c(X(v_1'))$
maps diagonally to $H^{2i-1}_c(X(v_1's))\bigoplus H^{2i-1}_c(X(w'))$ and $H^{2i-2}_c(X(v_2'))$ maps surjectively onto 
$H^{2i-1}_c(X(w'))$ by Proposition \ref{boundary_Coxeter}.

\noindent {\it 2. Case:} $w=s_is_1s_2\cdots \widehat{s_i}\cdots s_{n-1}s_i$ with $2\leq i \leq n-2.$

By Corollary \ref{coh_kleiner_Cox} we see that $H^{2i+1}_c(X(w))$ is a direct sum of quotients of induced representations $i^G_P(-i)$ where
$P$ is not a proper maximal subgroup. But $H^{2i}_c(X(sw'))$ does not kill $H^{2i+1}_c(X(w))$ by
Prop. \ref{height_one}. Further  $H^{2i}_c(X(sv'))=H^{2\ell(w)-4}_c(X(sv'))=H^{2\ell(sv')}_c(X(sv'))=i^G_{P(sv')}(-i)$ for all $v'\prec w'$
with $\ell(v')=\ell(w')-1,$ where $P(sv')\subset G$ is a proper maximal subgroup. Hence the representations of the
first kind have to be  killed by weight reasons by  some
$H^{2i}_c(X(su's))$ with $\ell(u')=\ell(w')-2.$
The claim follows.

If $V$ is not necessarily induced by $H^{2i+1}_c(X(w)),$ then we argue as follows.
In the first case above we consider the subsquares  $X(s_1^2,s_1s_3s_4\ldots s_{n-1}s_1)$, $X(s_1s_2s_1,s_1s_2s_4s_5\ldots s_{n-1}s_1)$,
$X(s_1s_2s_3s_1,s_1s_2s_3s_5s_6\ldots s_{n-1}s_1)$ within $X(s_1^2,w)$. These are homeomorphic to
$X(s_1^2)\times \overline{X(s_3s_4\ldots s_{n-1})}$, $X(s_1s_2s_1)\times \overline{X(s_4s_5\ldots s_{n-1})}$, 
$X(s_1s_2s_3s_1)\times \overline{X(s_5s_6\ldots s_{n-1})}$ etc. 
As for the first two subvarieties $U$ we have $H^{2i+1}_c(U)\langle - i \rangle=0.$ On the other hand,
for the other varieties $U$, it follows by induction that we have $H^{2i+1}_c(U)\langle - i \rangle=0.$
Hence we get a contradiction since there has to be one out of these varieties which induce $V.$

In the second case one proceeds similarly.
\end{proof}

\vspace{0.5cm}

\section{Some further results on the cohomology of DL-varieties}

We shall proof some further statements given in the introduction.

Recall we may write all elements in $B^+$ modulo cyclic shift  in the shape $w=sw's$ for some $s\in S$ and $w'\in B^+.$ In particular,
for $w\in {F^+}$, there is always an element $sw's \in {F^+}$ with  $\ell(w)=\ell(sw's)$ and with $H^\ast_c(X(w))=H^\ast_c(X(sw's)).$
We shall use this property to prove the next statement.

\begin{lemma}\label{vanish}
 Let $w\in {F^+}$. Then $H^i(\overline{X}(w))=0$ for $i$ odd.
\end{lemma}

\begin{proof}
Since $\overline{X}(w)$ is smooth and projective it suffices to show that all eigenvalues of the Frobenius on the cohomology groups
$H^\ast(\overline{X}(w))$ are integral powers of $q.$ By considering the spectral sequence to the  stratification
$\overline{X}(w)=\bigcup_{v\preceq w}X(v)$ it suffices to show that this property is valid for the cohomology groups $H^\ast_c(X(v)).$
By what we have said above,  we may suppose that $v=sv's$ for some  $s\in S$ and $v'\in {F^+}$. By induction on the length we know that the assertion
is true for $H_c^\ast(X(v's))$ and $H^\ast_c(X(v')),$ hence for $H^\ast_c(X(v's)\cup X(v'))$. But $X(v) \cup X(sv')$ is an
$\mA^1$-bundle over $X(v's) \cup X(v')$ by Remark \ref{same_is_true_for_F}. Thus the assertion is true for $H^\ast_c(X(v)\cup X(sv'))$. Finally,
by considering again the corresponding long exact cohomology sequence  the claim follows.
\end{proof}

Let $X(w)$ be a DL-variety
attached to an element $w\in W$ and let $w=s_{i_1}\cdots s_{i_r}$ be a reduced decomposition of $w.$
In order to  compute the cohomology of $X(w)\cong X^{F^+}(w)$ we consider the stratification
$\overline{X}(w)=\bigcup_{v\preceq w}X^{{F^+}}(v)$ in which $X(w)$ appears
as an open stratum. Write
$$\overline{X}(w)=X^{{F^+}}(w) \stackrel{\cdot}{\cup} Y$$
where $Y=\bigcup_{\genfrac{}{}{0pt}{2}{v \prec w}{\ell(v)=\ell(w)-1}} \overline{X}(v).$
We consider the induced
spectral sequence $$ E_1^{p,q} \Longrightarrow H^{p+q}_c(X^{{F^+}}(w))$$
with
$$E_1^{p,q}=\bigoplus_{\genfrac{}{}{0pt}{2}{\{v_1,\ldots, v_p\}}{v_i \prec w, \ell(v_i)=\ell(w)-1}} H^q_c\Big(\bigcap_{i=1}^p \overline{X}(v_i)\Big)$$
for $p\geq 1$ and $E_1^{0,q}=H^q(\overline{X}(w))$ .
Note that the intersection $\bigcap_1^p \overline{X}(v_i) $ is nothing else but
$\overline{X}(v)$ where $v\in {F^+}$ is the unique element of length $\ell(w)-p$ with $v \preceq v_i, \, i=1,\ldots,p$.

\begin{rk}
The element $v$ could be considered as the greatest common divisor or the meet of the elements $v_1,\ldots,v_p.$ Note that the set
$Q(1,w)$ is a bounded distributive lattice.
\end{rk}

Hence the $i$th row of $E_1$ is given by the complex
\begin{equation}\label{E_1_term}
 0 \to H^i(\overline{X}(w)) \to  \bigoplus_{\genfrac{}{}{0pt}{2}{v\prec w}{\ell(v)=\ell(w)-1}} H^{i}(\overline{X}(v)) \to
\bigoplus_{\genfrac{}{}{0pt}{2}{v\prec w}{\ell(v)=\ell(w)-2}} H^{i}(\overline{X}(v))
\to \cdots
\end{equation}
We shall analyse this spectral sequence. As all varieties $\overline{X}(v)$ are smooth and projective their cohomology is pure. We conclude
that $E_2=E_\infty$ and hence by weight reasons that
$$H^i_c(X(w))= \bigoplus_{p+q=i} E_2^{p,q}.$$

\begin{prop}\label{Frobss}
 The representations  $H^\ast(\overline{X}(w))$ and $H_c^\ast(X(w))$ are Frobenius semisimple  for all $w\in F^+.$
\end{prop}

\begin{proof}
Again the proof is by induction on $\ell(w).$ The start of induction is given by Proposition \ref{cohomology_Coxeter_compactification}.
As the weights of $H^{i-1}(Y)$ are different from $H^i(\overline{X}(w)),$ it is enough to prove that both of these objects are Frobenius semisimple.
But by considering the $E_2$-term of the obvious  spectral sequence converging to the cohomology of $Y$ and by induction hypothesis  it suffices to show that
$H^i(\overline{X}(w))$ is Frobenius semisimple.

Since  $\overline{X}(w)$ is smooth and projective we get by Poincar\'e duality the identity $H^i(\overline{X}(w))=H^{2\ell(w)-i}(\overline{X}(w))(-\ell(w)+i)$ for $i \leq \ell(w).$
So it suffices to consider the case $i\leq \ell(w).$ Further we know that $H^i_c(X(w))=(0)$ for all $i<\ell(w)$, cf. Remark \ref{vanishing_coh}.
Hence we deduce that
$H^i(\overline{X}(w))\subset H^i(Y)$ for all $i< \ell(w).$ In the latter case the claim follows by induction considering again the spectral
sequence to $Y$.

If $i=\ell(w)$ (is even and positive), then we consider the long exact sequence
\begin{eqnarray*}
0 \to & H^{\ell(w)-1}(Y) & \to H^{\ell(w)}_c(X(w)) \to H^{\ell(w)}(\overline{X}(w)) \\ \to & H^{\ell(w)}(Y) & \to H^{\ell(w)+1}_c(X(w)) \to 0.
\end{eqnarray*}
We claim that if there is some irreducible subrepresentation $V=j_\lambda(-j) \subset H^{\ell(w)}_c(X(w))$, then $j<\frac{i}{2}.$
Here we may suppose that $w=sw's.$ If $V\subset H^{\ell(w)-1}(Y)$, then  the claim follows
by weight reasons. If $V\subset H^{\ell(w)}(\overline{X}(w))$, i.e. $j=\frac{i}{2}$, then it is in the kernel of the map
$$H^{\ell(w)}(\overline{X}(w)) \to \bigoplus_{v \prec w \atop \ell(v)=\ell(w)-1} H^{\ell(w)}(\overline{X}(v)).$$
Since $w's$  appears as index in this direct sum, the kernel  is by Remark \ref{P1_closures} the same as the kernel of the map
$$H^{\ell(w)-2}(\overline{X}(w's))(-1) \to H^{\ell(w)}(\overline{X}(sw'))\bigoplus \bigoplus_{v' \prec w's \atop \ell(v')=\ell(w')-1} H^{\ell(w)-2}(\overline{X}(v's))(-1).$$
In particular, it is contained in the kernel of the map
$$H^{\ell(w)-2}(\overline{X}(w's))(-1) \to  \bigoplus_{v' \prec w's \atop \ell(v')=\ell(w')-1} H^{\ell(w)-2}(\overline{X}(v's))(-1).$$
Since the contribution of $w'$ is missing on the RHS, we deduce that
$$V(1)=j_\lambda(-j+1)\subset H_c^{\ell(w)-2}(X(w's) \cup X(w')).$$ But  $H_c^{\ell(w)-2}(X(w's))=(0)$, as
$\ell(w')=\ell(w)-2< \ell(w's).$ Hence $V(1)\subset H_c^{\ell(w')}(X(w')).$ Again by induction we know that $j-1<\frac{\ell(w')}{2}.$
But $\frac{\ell(w')}{2}=\frac{\ell(w)-2}{2}=\frac{\ell(w)}{2}-1$. Hence we get a contradiction.
\end{proof}

\begin{coro}(of the proof)
 Let $w\in F^+\setminus\{e\}$ and let $V=j_\lambda(-i)\subset H^{\ell(w)}_c(X(w))$ for some $\lambda \in \P.$ Then $2i < \ell(w).$
\end{coro}

\begin{rks}
i) The latter result is proved in \cite[Prop. 3.3.31 (iv)]{DMR} for arbitrary reductive groups.

ii) It was pointed out to me by O. Dudas, that the semi-simplicity of $H^\ast(\overline{X}(w))$, Lemma \ref{vanish} and the upcoming
Proposition \ref{cycle_isom} can be  deduced from \cite{L4}.
\end{rks}

We further have the following vanishing result.

\begin{prop}\label{vanishtop}
 Let $w\in F^+$ with $\h(w)\geq 1.$ Then $H^{2\ell(w)-1}_c(X(w))=0.$
\end{prop}

\begin{proof}
By Corollary \ref{coro_height1_a} we may suppose that $\h(sw')\geq 1.$
 We consider the long exact cohomology sequence
\begin{eqnarray*}
 \cdots \to  & H^{2\ell(w)-2}_c(X(w) \cup X(sw')) & \stackrel{r}{\to} H^{2\ell(w)-2}_c(X(sw')) \to H^{2\ell(w)-1}_c(X(w)) \\
\to &  H^{2\ell(w)-1}_c(X(w) \cup X(sw')) & \to \cdots.
\end{eqnarray*}
The map $r=r^{2\ell(w)-2}_{w,sw'}$ has to be surjective since $H^{2\ell(w)-2}_c(X(sw'))=H^{2\ell(sw')}_c(X(sw'))=i^G_G(-\ell(sw'))$ is the top cohomology group.
On the other hand, by Proposition \ref{cohomology_vb} we know  that $H^{2\ell(w)-1}_c(X(w) \cup X(sw'))=H^{2\ell(w)-3}_c(X(w's) \cup X(w'))(-1).$
But $$H^{2\ell(w)-3}_c(X(w's))=H^{2\ell(w's)-1}_c(X(sw'))=0$$ by induction. Further $H^{2\ell(w)-3}_c(X(w'))=H^{2\ell(w')+1}_c(X(w'))=0.$
Hence we conclude the claim.
\end{proof}

This vanishing result has the following consequences.

\begin{coro}\label{coro_l(w)-1}
Let $w\in F^+$ with $\h(w)\geq 1.$ Then
\begin{eqnarray*}
 H^{2\ell(w)-2}(\overline{X}(w)) & = & H^{2\ell(w)-2}_c(X(w))\bigoplus \bigoplus_{v \preceq w \atop \ell(v)= \ell(w)-1} H^{2\ell(w)-2}_c(X(v)) \\
& = & H^{2\ell(w)-2}_c(X(w))\bigoplus \bigoplus_{v \preceq w \atop \ell(v)= \ell(w)-1} i^G_{P(v)}(-(\ell(w)-1)).\qed
\end{eqnarray*}
\end{coro}

\begin{proof}
 We consider the long exact cohomology sequence
$$ \cdots \to   H^{i-1}(Y)  \to H^{i}_c(X(w)) \to H^{i}(\overline{X}(w)) \to H^{i}(Y)  \to \cdots$$
 which coincides  in degree $i=\ell(w)-2$ with
 \begin{eqnarray*}
 \cdots \to  & \bigoplus\limits_{v\prec w \atop \ell(v)=\ell(w)-1}H^{2\ell(w)-3}_c(X(v)) & \stackrel{\delta}{\to} H^{2\ell(w)-2}_c(X(w)) \to H^{2\ell(w)-2}(\overline{X}(w)) \\
\to & \bigoplus\limits_{v\prec w \atop \ell(v)=\ell(w)-1} H^{2\ell(w)-2}_c(X(v)) & \to H^{2\ell(w)-1}_c(X(w)) \to \cdots
\end{eqnarray*}
Now $H^{2\ell(w)-1}_c(X(w))=0$. If $\h(v)\geq 1$ then we conclude by induction that $H^{2\ell(w)-3}_c(X(v))=H^{2\ell(v)-1}_c(X(v))=0.$
Thus if this latter condition is satisfied, we are done. Otherwise $w$ is of height one and there is some $v\prec w$
with $\h(v)=0.$. Then we deduce from Corollary \ref{coro_height1_a} and weight reasons that the map $\delta$ vanishes.
\end{proof}

\begin{coro}
 Let $w=sw's\in F^+$ with $\h(w')\geq 1$ and $\supp(w)=S.$
 Then $$H^{2\ell(w)-2}_c(X(w))=H^{2\ell(w's)-2}_c(X(w's))(-1) \bigoplus (i^G_{P(w')}-i^G_G)(-\ell(w)+1).$$
\end{coro}

\begin{proof}
 Since $\h(sw')\geq 1$ we have an exact sequence
 $$0 \to H^{2\ell(w)-2}_c(X(w)) \to H^{2\ell(w)-2}_c(X(w)\cup X(sw')) \to H^{2\ell(w)-2}_c(X(sw'))\to 0.$$
 But the assumption $\supp(w)=S$ also implies  that $\supp(sw')=S.$ Hence we get
 $$H^{2\ell(w)-2}_c(X(sw'))=H^{2\ell(sw')}_c(X(sw'))=i^G_G(-\ell(w)+1).$$
 Further the identity $H^{2\ell(w)-2}_c(X(w) \cup X(sw'))=H^{2\ell(w)-4}_c(X(w's) \cup X(w'))$ is satisfied
 by Corollary \ref{cohomology_vb} But since $\h(w')\geq 1$ we deduce
 that $$H^{2\ell(w)-5}_c(X(w'))=H^{2\ell(w')-1}_c(X(w'))=0.$$ 
 Now the result follows easily.
\end{proof}

We reconsider  the spectral sequence
$$ E_1^{p,q} \Rightarrow H^{p+q}(H^\ast(\overline{X}(w)))$$
which is induced by the stratification
$\overline{X}(w) =  \;\stackrel{\cdot}{\bigcup}_{v \preceq w} X(v).$
The $q$th line in the  $E_1$-term is  the complex
$$\bigoplus_{v'\preceq w \atop \ell(v')=q} H_c^{2q}(X(v')) \to \cdots \to \bigoplus_{v'\preceq w \atop \ell(v')=q+j} H_c^{2q+j}(X(v'))
\to \bigoplus_{v\preceq w \atop \ell(v)=q+j+1} H_c^{2q+j+1}(X(v)) \to \cdots$$
where the homomorphisms are induced by  the boundary maps $\delta_{v',v}.$

\vspace{0.5cm}
{\tt Picture}: $$\begin{array}{ccccccccccccccc}
 \noindent         & & & &     & &  &  & & & & &   \vdots & & \vdots  \\
     & & & & H^4 & & H^5 & & H^6 & & \cdots & & H^{i+2} & & H^{i+3} \cdots \\ \\
    & & H^2 & & H^3 & & H^4 & & H^5 & & \cdots & & H^{i+1} & & H^{i+2} \cdots \\ \\
H^0 & & H^1 & & H^2 & & H^3 & & H^4 & & \cdots & & H^i & & H^{i+1} \cdots\\
\hline \\
e & & \ell=1 & & \ell=2 & & \ell=3 & & \ell=4 & & \cdots & & \ell=i & & \ell=i+1
\end{array}$$

\bigskip

Of course this spectral sequence degenerates and we may write by weight reasons and by Proposition \ref{Frobss} for all $0\leq i\leq \ell(w)$,
$$H^{2i}(\overline{X}(w))=\bigoplus_{j=i}^{\ell(w)}H^{2i}_c(X(w)(j))'$$
where
$$X(w)(j):=\bigcup_{v\preceq w\atop \ell(v)=j}X(v)$$ and where
$H^{2i}_c(X(w)(j))'\subset H^{2i}_c(X(w)(j))=\bigoplus_{v\preceq w\atop \ell(v)=j} H^{2i}_c(X(v)).$

For $v \preceq w$ with $\ell(v)=i$, we have $H^{2i}_c(X(v))=H^{2i}(\overline{X}(v))=i^G_{P(v)}(-i)$. Here $P(v)\subset G$ is the std psgp attached to
$v$, cf. (\ref{psgp}.)
By Remark \ref{vanishing_square}  the trivial representation does appear in the top cohomology degree of a DL-variety. Hence the subrepresentation
$i^G_G(-i)\subset i^G_{P(v)}(-i)$ survives the spectral sequence. Thus there is grading
$$H^{2i}(\overline{X}(w))=\bigoplus_{z\preceq w\atop \ell(v)=i} H(w)_z$$
with $H(w)_z\supset i^G_{G}(-i)$ for certain representations $H(w)_z.$

In the sequel we shall see that we may suppose that the objects $H(w)_z$ are induced representations from parabolic subgroups.
Recall that the  following three operations $C,K,R$  on elements in $F^+$ allows us to transform an
arbitrary element $w\in F^+$ into the shape $w=sw's$ with $s\in S$ and $w' \in F^+.$

\noindent (I) (Cyclic shift) If $w=sw'$ with $s\in S$, then we set $C(w)=w's.$ \smallskip

\noindent (II) (Commuting relation). If $w=w_1 s t w_2$ with $s,t\in S$ and $st=ts.$ Then we set $K(w)=w_1tsw_2.$\smallskip

\noindent (III) (Replace $sts$ by $tst$) If $w=w_1stsw_2$ with $s,t\in S$  and $sts=tst$.
Then we  set $R(w)=w_1tst w_2.$\smallskip

We shall analyse the induced behaviour on the cohomology of Demazure varieties.
In \cite{DMR} the following generalization of Proposition \ref{Prop_homeomorphic} is proved.

\begin{prop}\label{cyclic_shift}
 Let $w=s_{i_1}\cdots s_{i_r} \in \hat{F}^+$. Then for all $i \geq 0,$ there are isomorphisms of $H$-modules 
 $$H^i_c(X(w)) \to H^i_c(X(C(w)))$$ and 
  $$H^i(\overline{X}(w)) \to H^i(\overline{X}(C(w)))$$ 
\end{prop}

\begin{proof}
 These are special cases of \cite[Proposition 3.1.6]{DMR}.
\end{proof}

\begin{coro}
Let $sw's\in F^+.$ Then $H^i(\overline{X}(sw'))=H^i(\overline{X}(w's))$ for all $i\geq 0.$ \qed
\end{coro}

\begin{rk}
Consider  the equivalence relation $\sim$ on $F^+$ leading to cyclic shift classes in the sense of \cite{GP}, i.e., $v,w\in F^+$ are equivalent if there is some integer
$i\geq 0$ with $C^i(w)=v.$ Thus we can associate to any element in $C^+=F^+ \slash \sim \;$ its cohomology. Moreover, the height function $\h$
on $F^+$ depends  only on the cycle shift class. Sometimes it is useful to interpret in what follows
the image of an element  $w=s_{i_1}s_{i_2}\cdots s_{i_{r-1}}s_{i_r}$ in $C^+$ as a circle, i.e.,

$$\begin{array}{ccc}
  ×& s_{i_1} & \\
  s_{i_r} & & s_2 \\
  \vdots & & \vdots \\
  s_{i_{j+1}} & & s_{i_{j-1}}\\
  & s_{i_j} &
 \end{array}$$

\end{rk}

\bigskip

As we have observed above the Cyclic shift operator does not affect the cohomology of a Demazure variety. The
same holds true for operation (II).

\begin{prop}\label{coh_iso_comm}
Let $w=w_1 s t w_2 \in \hat{F}^+$ with $s,t\in S$ (or $s\in S, t\in \hat{S}$) such that $st=ts.$ Set $K(w)=w_1tsw_2.$ Then
$H^i(\overline{X}(w))=H^i(\overline{X}(K(w)))$ for all $i \geq 0.$
\end{prop}

\begin{proof}
This property is implicitly contained in the definition of a generalized Deligne-Lusztig variety attached to elements of the completed
Braid monoid \cite{DMR}.
The reason is that the stratifications of the  varieties $\overline{X}(w)$ and $\overline{X}(K(w))$ are essentially the same in the obvious sense.
\end{proof}

Let $w=w_1stsw_2\in F^+.$ In the sequel we also write $w=w_1s_lts_rw_2$ ($l$ for left, $r$ for right) in order
to distinguish the reflection $s$ in its appearance in $w$. Further we write $w_1s^2w_2$ for the subword
$w_1s_ls_rw_2.$

\begin{prop}\label{coh_sts}
Let $s,t\in S$, $w_1,w_2 \in \hat{F}^+$ such that $st\neq ts$ in $W$. Set $w=w_1stsw_2$ and $v=R(w)=w_1tstw_2.$ Then for all $i\geq 0$,
there is an identity
$$H^i(\overline{X}(v))= H^i(\overline{X}(w)) - H^{i}_c(X(s_l,w_1s^2w_2)) + H^{i}_c(X(t_r,w_1t^2w_2))$$
$$ = H^i(\overline{X}(w)) - H^{i-2}(\overline{X}(w_1sw_2))(-1) + H^{i-2}(\overline{X}(w_1tw_2))(-1).$$
(Here we mean as before when writing $-H^{i-2}(\overline{X}(w_1sw_2))(-1)$ that $H^{i-2}(\overline{X}(w_1sw_2))(-1)$ appears canonically as a submodule in $H^i(\overline{X}(w))$.)
\end{prop}

\begin{proof}
A priori the varieties $X(w)$ and $X(v)$  differ by the locally closed subsets $X(s,w_1s^2w_2)$ $\subset \overline{X}(w)$ resp.
$X(t,w_1t^2w_2)\subset \overline{X}(v)$, i.e., the constructible subsets
$\overline{X}(w)\setminus X(s,w_1s^2w_2)$ and  $\overline{X}(v)\setminus X(t,w_1t^2w_2)$ have homeomorphic stratifications.
Hence we see that their Euler-Poincar\'e characteristics in the Grothendieck group of $G$-modules are the same. More precisely, we have
$${\rm EP}(\overline{X}(w))-{\rm EP}(X(s,w_1s^2w_2))={\rm EP}(\overline{X}(v))-{\rm EP}(X(t,w_1t^2w_2)),$$ where we set for a variety
$X$, ${\rm EP}(X)=\sum_i (-1)^i H^i_c(X)$. Moreover, the variety $X(s_l,w_1s^2w_2)$ (resp. $X(t_r,w_1t^2w_2))$
is a $\mA^1$-bundle over $\overline{X}(w_1sw_2)$ (resp. $\overline{X}(w_1tw_2)$). Hence the individual cohomology groups of
the varieties  vanish in odd degree. It follows immediately that for all $j\geq 0$, the above identity
$H^j(\overline{X}(w)) - H^j_c(X(s_l,w_1s^2w_2))=H^j(\overline{X}(v)) -H^j_c(X(t_r,w_1t^2w_2))$ holds true in the Grothendieck group of $H$-representations.

Finally we have to show that $H^{i-2}(\overline{X}(w_1sw_2))(-1)$ is a submodule of $H^i(\overline{X}(w))$. For this we apply
the following lemma which includes moreover what we have shown before.
\end{proof}

\begin{lemma}\label{lemma_DMR}
Let Let $w=w_1stsw_2 \in \hat{F}^+$ with $s,t\in S$ and  $st\neq ts$ in $W$. Then there is a canonical decomposition 
$H^{2i}(\overline{X}(w))=H^{2i}(X(w_1\widehat{sts}w_2)) \oplus 
H^{2i-2}(X(w_1sw_2))(-1).$ 
\end{lemma}

\begin{proof}
This is a special situation of \cite[Prop. 3.2.9]{DMR}. The argument is that the proper morphism
$\pi:\overline{X}(w)\to \overline{X}(w_1\widehat{sts}w_2)$ by forgetting appropriate entries is an isomorphism over the open subset 
$\overline{X}(w_1\widehat{sts}w_2)\setminus \overline{X}(w_1sw_2)$
and induces over $\overline{X}(w_1sw_2)$ a $\mP^1$-bundle. As both varieties are smooth and projective
the claim follows from considering long exact cohomology sequences.
\end{proof}

\begin{prop}\label{cycle_isom}
 Let $w\in F^+$. Then for all $i \geq 0$, the cycle map $A^i(\overline{X}(w))_{\overline{\mQ}_\ell}\to H^{2i}(\overline{X}(w))$ is an isomorphism
$\bigl($where $A^i(\overline{X}(w))$ is the Chow group of $\overline{X}(w) \mbox{ in degree $i$} \bigr)$.
\end{prop}

\begin{proof}
We may assume that $\supp(w)=S.$
If $w$ is a Coxeter element, then the claim follows from Remark \ref{rk_cox_chow}.
 If $w=sw's$ then $A^i(\overline{X}(w))=A^i(\overline{X}(w's)) \oplus A^{i-1}(\overline{X}(w's))$ and the claim follows by induction
 on $\ell(w).$
 
In general we have to consider the induced behaviour of the operations $K,C,R$ on the Chow groups. First we observe that
$A^i(\overline{X}(w))=A^i(\overline{X}(C(w)))$ since the generating cycles are rational. Further 
$A^i(\overline{X}(w))=A^i(\overline{X}(K(w)))$.
Let $w=w_1stsw_2$, $R(w)=w_1tstw_2$, $y=w_1\widehat{sts}w_2$  and suppose that the claim is true for $w.$ Then we consider as in Lemma \ref{lemma_DMR}
the map $\pi: \overline{X}(w) \to \overline{X}(y)$ which induces short  exact sequences
$$0 \to H^{2i-4}(\overline{X}(w_1sw_2))(-2) \to H^{2i}(\overline{X}(y)) \oplus H^{2i-2}(\overline{X}(w_1s^2w_2))(-1) \to H^{2i}(\overline{X}(w))\to 0$$
and 
$$0 \to A^{i-2}(\overline{X}(w_1sw_2)) \to A^{i}(\overline{X}(y)) \oplus A^{i-1}(\overline{X}(w_1s^2w_2)) \to A^{i}(\overline{X}(w))\to 0.$$
It follows by induction on the length that the cycle map for $\overline{X}(y)$ is an isomorphism, as well. Considering both exact 
sequences for $R(w)$ and using again
induction the claim is true for $R(w).$
\end{proof}

\begin{rk}
  For $m\geq 0$, the Tate twist $-i$ contribution of
 $H^{2i+m}_c(X(w))$ is Bloch's higher Chow group $CH^i(X(w),m)_{\overline{\mQ}_\ell}$, cf. \cite{Bl}. Indeed, this follows from the 
 spectral sequence (\ref{E_1_term}). 
\end{rk}

For $w\in F^+$ and $z\preceq w$, we denote by $i_{z,w}:A_{\ell(z)}(\overline{X}(z)) \to A_{\ell(z)}(\overline{X}(w))$ the  map 
induced by the  inclusion $\overline{X}(z) \subset \overline{X}(w).$ Moreover, we let 
$[\overline{X}(z)] \in A_{\ell(z)}(\overline{X}(z))$ be the sum of all the irreducible components appearing in 
$\overline{X}(z).$

\begin{defn}
 Let $w\in \hat{F}^+.$ A grading $A_{i}(\overline{X}(w))_{\overline{\mQ}_\ell}=\bigoplus_{z\preceq w \atop \ell(z)=i} A(w)_z$ is called geometrical
 if there is an order $z_1,\ldots, z_r$ on the set $ \{ z\preceq w \mid \ell(z)=i\}$  such that 
 $$\bigoplus_{j=1,\ldots,k}A(w)_{z_j} \supset 
 \sum_{j=1,\ldots,k} (i_{z_j,w})_\ast([\overline{X}(z_j)])$$ for all $k \leq r.$
\end{defn}

By using the cycle map we may speak of geometrical gradings on $H^{2i}(\overline{X}(w)).$

Let $w=sw's\in F^+$ as before. Consider the commutative diagram
\begin{equation}\begin{array}{ccccccc}
 & \vdots & & \vdots & & \vdots &  \\
 & \downarrow & & \downarrow & & \downarrow &  \\
\cdots \to &  H_c^{2i}(X(s^2,w)) & \to &  H_c^{2i}(X(s,w)) &  \to & H^{2i}_c(X(s,w's)) &  \to \cdots \\
 & \downarrow & & \downarrow & & \downarrow   &  \\
\cdots \to & H_c^{2i}(X(s,w)) & \to &  H^{2i}(\overline{X}(w)) & \to & H^{2i}(\overline{X}(w's)) & \to \cdots \\
 & \;\;\;\downarrow  g & & \downarrow  & & \downarrow  &   \\
\cdots \to &  H_c^{2i}(X(s,sw')) & \to &  H^{2i}(\overline{X}(sw')) &  \stackrel{f}{\to} & H^{2i}(\overline{X}(w')) &  \to \cdots \\
 & \downarrow & & \downarrow & & \downarrow &  \\
 & \vdots & & \vdots & & \vdots &  \\
\end{array}
\end{equation}
where the maps in this diagram are the natural ones. We slightly generalize this setup  as follows.
Let $v=v_1rv_2$ with $r\in S$ and $u=v_1v_2.$ Set $\breve{v}=v_1r^2v_2=v_1r_lr_rv_2$ and $\bar{v}=v_1r_rv_2.$ Then $v=v_1r_lv_2$ and we may form
the following diagram
\begin{equation}\label{commutative_dia}\begin{array}{ccccccc}
 & \vdots & & \vdots & & \vdots &  \\
 & \downarrow & & \downarrow & & \downarrow &  \\
\cdots \to &  H_c^{2i}(X(r^2,\breve{v})) & \to &  H_c^{2i}(X(r_r,\breve{v})) &  \to & H^{2i}_c(X(r_r,\bar{v})) &  \to \cdots \\
 & \downarrow & & \downarrow & & \downarrow   &  \\
\cdots \to & H_c^{2i}(X(r_l,\breve{v})) & \to &  H^{2i}(\overline{X}(\breve{v})) & \to & H^{2i}(\overline{X}(\bar{v})) & \to \cdots \\
 & \;\;\;\downarrow  g & & \downarrow  & & \downarrow  &   \\
\cdots \to &  H_c^{2i}(X(r_l,v)) & \to &  H^{2i}(\overline{X}(v)) &  \stackrel{f}{\to} & H^{2i}(\overline{X}(u)) &  \to \cdots \\
 & \downarrow & & \downarrow & & \downarrow &  \\
 & \vdots & & \vdots & & \vdots &  \\
\end{array}
\end{equation}

In the setting of Chow groups and Poincar\'e duality the above diagram reads as 

\begin{equation}\label{commutative_dia_chow}\begin{array}{ccccccc}
 & \vdots & & \vdots & & \vdots &  \\
 & \uparrow & & \uparrow & & \uparrow &  \\
\cdots \to &  A_{i}(X(r_r,\bar{v})) & \to &  A_{i}(X(r_r,\breve{v})) &  \to & A_{i}(X(r^2,\breve{v})) &  \to \cdots \\
 & \uparrow & & \uparrow & & \uparrow   &  \\
\cdots \to & A_{i}(\overline{X}(\bar{v})) & \to &  A_{i}(\overline{X}(\breve{v})) & \to & A_{i}(X(r_l,\breve{v})) & \to \cdots \\
 & \uparrow   & & \uparrow  & & \;\;\;\;\;\uparrow  g' &   \\
\cdots \to &  A_{i}(\overline{X}(u)) & \stackrel{f'}{\to} &  A_{i}(\overline{X}(v)) &  \to & A_{i}(X(r_l,v)) &  \to \cdots \\
 & \uparrow & & \uparrow & & \uparrow &  \\
 & \vdots & & \vdots & & \vdots &  \\
\end{array}
\end{equation}

The next result generalizes  Proposition \ref{qi_restriction} to arbitrary pairs of elements $v,u\in F^+$ with $u\prec v$ and 
$\ell(v)=\ell(u)+1$.

\begin{thm}\label{restriction_graded}

a) The cohomology of $\overline{X}(v)$ in degree $2i$ can be written as
$$H^{2i}(\overline{X}(v))=\bigoplus_{z\preceq v\atop \ell(v)=i} H(v)_z$$
with $H(v)_z=i^G_{P^v_z}(-i)$ for certain standard parabolic subgroups $P^v_z\subset G.$

b) There are geometric gradings
$H^{2i}(\overline{X}(v))=\bigoplus_{z\preceq v \atop \ell(z)=i}H(v)_z$, 
$H^{2i}(\overline{X}(u))=\bigoplus_{z\preceq u \atop \ell(z)=i}H(u)_z$ by induced representations as in part a), such that the map
$f:H^{2i}(\overline{X}(v)) \to  H^{2i}(\overline{X}(u))$ is in diagonal form, i.e. it coincides with the graded one. 
Further the induced homomorphisms $H(v)_z \to H(u)_z$ are injective or surjective for all $z\preceq u.$ 

c) There is a geometric grading  as in part a)  $H_c^{2i}(X(r_l,\breve{v}))=H^{2i-2}(\overline{X}(\bar{v}))(-1)=
\bigoplus_{z'\preceq \bar{v} \atop \ell(z')=i-1}H(\bar{v})_{z'}$ such that $H(\bar{v})_{z'}(-1)=H(v)_{z_1rz_2}$ for all 
elements $z'=z_1z_2 \preceq u.$
\end{thm}

We leave the translation of this theorem for covariant Chow groups to the reader.
We start with the following lemma.

\begin{lemma}\label{independence}
Suppose that the theorem is true. Then  statement b) is true if start with an arbitrary geometric grading 
on $H^{2i}(\overline{X}(u)).$
\end{lemma}

\begin{proof} We consider the viewpoint of Chow groups. Let $A_{i}(\overline{X}(u))_{\overline{\mQ}_\ell}=\bigoplus_{z\preceq u \atop \ell(z)=i}A(u)'_z$ be 
another grading.  Then by Corollary \ref{grading_unique}
for any $z\preceq u$ there is an element $z'=\phi(z)$ with $z'\preceq u$, $\ell(z)=\ell(z')=i$ and  with 
$A(u)_z \cong A(u)'_{z'}$. 
The map
$\pi^u_{z,\phi(z)}: A(u)_z \stackrel{\sim}{\to} A(u)'_{\phi(z)}$ which is induced by the neutral element, 
cf. Remark \ref{connection_symmetric group}, and the choice of $P^u_z$-invariant generating cycles for the induced 
representations $i^G_{P^u_z}$ gives rise to such  an isomorphism. 

Suppose first that the  maps $A(u)_z \to A(v)_z$ are  
surjective for all $z \preceq u$ with $\ell(z)=i.$ 
As $G$ acts semi-simple we may consider
$A(v)_z$ as a subrepresentation of $A(u)_z$. We define a new grading on 
$A_{i}(\overline{X}(v))_{\overline{\mQ}_\ell}=\bigoplus_{z\preceq v \atop \ell(z)=i}A(v)'_z$
by setting 
$$A(v)'_{\phi(z)}:=\pi^u_{z,\phi(z)}(A(v)_z)$$ for $z\preceq u$ and $A(v)'_z:=A(v)_z$ if $z \not\preceq u.$ 
We set $\pi^u:=\oplus_{z \prec u \atop \ell(z)=i} \pi^u_{z,\phi(z)}$ so that we get an endomorphism 
$$\pi^u:A_i(\overline{X}(u))_{\overline{\mQ}_\ell} \to  A_i(\overline{X}(u))_{\overline{\mQ}_\ell}.$$ Analogously we define an endomorphism 
$\pi^v:=\oplus_{z \prec v \atop \ell(z)=i} \pi^v_{z,\phi(z)}$ of $A_i(\overline{X}(v))_{\overline{\mQ}_\ell}$
with $\pi^v_{z,\phi(z)}=\id.$ Here we have extended $\phi$ to a function on the set $\{z \preceq v, \ell(z)=i\}$
by $\phi(z)=z$ for $z \not\preceq u.$ We 
get a commutative diagram
\begin{eqnarray*}
 A_i(\overline{X}(v))_{\overline{\mQ}_\ell} & \stackrel{\pi^v}{\to} &  A_i(\overline{X}(v))_{\overline{\mQ}_\ell} \\
 \uparrow f' & & \uparrow f'\\
 A_i(\overline{X}(u))_{\overline{\mQ}_\ell} & \stackrel{\pi^u}{\to} & A_i(\overline{X}(u))_{\overline{\mQ}_\ell}
\end{eqnarray*}
Now the statement follows easily.

In general we devide the set $\{z \preceq u \mid \ell(z)=i \}=A \stackrel{.}{\cup} B$ into two disjoint subsets
where $A$ consists of those $z\preceq u$ such that the map $A(u)_z \to A(v)_z$ is surjective.
To define a new grading on $A_i(\overline{X}(v))_{\overline{\mQ}_\ell}$ we proceed with the set $A$ as above. 
In particular the set $A$ covers the kernel of the map $f':A_i(\overline{X}(u))_{\overline{\mQ}_\ell} \to A_i(\overline{X}(v))_{\overline{\mQ}_\ell}.$ 
Hence for $z\in B$ the map $A(u)_{\phi(z)} \to A_i(\overline{X}(v))_{\overline{\mQ}_\ell}$ is (strictly) injective.
Then there is some induced representation $i^G_P \subset A_i(\overline{X}(v))_{\overline{\mQ}_\ell}$ 
which contains the image of the latter map \footnote{Indeed let $i^G_P \to i^G_Q \oplus i^G_R$ be an injective map. 
Then we may suppose w.l.o.g. that $i^G_P \to i^G_Q$ is injective, as well. We may extend the first map to an isomorphism
$i^G_Q \to i^G_Q \oplus i^G_R$ and the graph contains $i^G_P.$}. We let $i^G_P$  be a constituent of the disired grading. But the map $\oplus_{z\in B}A(u)_{\phi(z)} \to A_i(\overline{X}(v))_{\overline{\mQ}_\ell}$ is injective,
as well. The claim of the lemma follows by applying the former procedure sucessively.
\end{proof}

\begin{proof}
Part a) and b). The last statement of part b) is a consequence of Remark \ref{connection_symmetric group} ii). The remaining proof  is by induction on $\ell(v).$ By Remark \ref{cohomology_not_full_F} and K\"unneth arguments 
we may suppose that $v$ has full support.
If $v$ is  a Coxeter element then the statement follows from Propositions \ref{cohomology_Coxeter_compactification},  \ref{qi_restriction}. 
So let $\h(v)\geq 1.$

\noindent 1. Case. $v=sv's$ for some $s\in S.$

Hence $H^{2i}(\overline{X}(v))= H^{2i}(\overline{X}(v's)) \bigoplus H^{2i-2}(\overline{X}(v's))(-1)$. As for part a)  we may write  by induction
$$H^\ast(\overline{X}(v's))=\bigoplus_{z\preceq v's} i^G_{P^{v's}_z}(-\ell(z))[-2\ell(z)].$$
Then $$H^\ast(\overline{X}(v))=\bigoplus_{z\preceq v} i^G_{P^v_z}(-\ell(z))[-2\ell(z)]$$
where $P^v_z=P^{v's}_z$ if $z\preceq v's$ and $P^v_z=P^{sv'}_{z'}$ if $z\in Q(sv',v)$. Here $z=sz'.$
Concerning part b) we distinguish the following cases:

Subcase a). $u=v's$ (The case $u=sv'$ is symmetric to this one).

Here the proof is trivial since the map $H^{2i}(\overline{X}(v)) \to H^{2i}(\overline{X}(u))$ identifies
with the projection map. 


Subcase b). $u=su's$.

Then $H^{2i}(\overline{X}(u))= H^{2i}(\overline{X}(u's)) \bigoplus H^{2i-2}(\overline{X}(u's))(-1).$
The statements follow now by induction with respect to the homomorphism $H^j(\overline{X}(v's)) \to H^j(\overline{X}(u's))$
with $j\in\{2i-2,2i\}.$

\noindent 2. Case. $v$ is arbitrary.

Then we apply the operations (I) - (III) to arrange $v$ in the shape as in the first case.
Here we use  inner induction on the necessary operations. So let $w \in F^+$ and suppose that the statements are true for 
$w$ and for all $x\prec w$ with $\ell(x)=\ell(w)-1.$ 

The operations (I) and (II) are easy to handle by Propositions \ref{cyclic_shift} and \ref{coh_iso_comm}:

(I) Let $w=sw'$, with $s\in S$ and $w'\in F^+$ and $w's=C(w)$. Let 
$\tilde{C}:H^{2i}(\overline{X}(w))\to H^{2i}(\overline{X}(C(w)))$ be the 
cyclic shift isomorphism. Define for $z\preceq w,$
$$C(z)=\left\{ \begin{array}{cc}
              z's & \mbox{ if } z=sz' \\
	      z & \mbox{ if } z \preceq w' .
             \end{array}\right. $$
Then the assignment $H(C(w))_{C(z)}:=\tilde{C}(H(w)_{z})$ for $\ell(z)=i$,  defines the desired grading on $H^{2i}(\overline{X}(C(w))).$
In the same way we get a grading on $H^{2i}(\overline{X}(C(x))).$ Moreover, the homomorphism $H^{2i}(\overline{X}(C(w))) \to
H^{2i}(\overline{X}(C(x)))$ is graded.

(II) Let  $w=w_1 s t w_2$ and $K(w)=w_1tsw_2$. Let 
$\tilde{K}:H^{2i}(\overline{X}(w))\to H^{2i}(\overline{X}(K(w)))$ be the 
induced isomorphism. Define for $z\preceq w,$
$$K(z)=\left\{ \begin{array}{cc}
              v_1tsv_2 & \mbox{ if } z=v_1stv_2 \\
	      z & \mbox{ if } st\not| \;z.
             \end{array}\right. $$
Then the assignment $H(K(w))_{K(z)}:=\tilde{K}(H(w)_{z})$ for $\ell(z)=i$,  defines the desired grading on $H^{2i}(\overline{X}(K(w))).$
Part b) is proved in the same way as above.

(III) So let $w=w_1stsw_2$ and $v=R(w)=w_1tstw_2$.  Set $\hat{w}=\hat{v}=w_1\widehat{sts}w_2.$

We start with the observation that  we have a geometrical grading on 
$H^{2i}(\overline{X}(\hat{w}))$. Indeed by induction hypothesis applied to $w$ and $x=w_1s^2w_2$,  we have gradings on 
$H^{2i}(\overline{X}(w))$ and $H^{2i}(\overline{X}(w_1s^2w_2))$   such that the natural map  
$H^{2i}(\overline{X}(w)) \to  H^{2i}(\overline{X}(w_1s^2w_2))$ is in diagonal form. In particular it follows that
$H(w)_z=H(w_1s^2w_2)_z$ for all $z \in Q(s,w_1s^2w_2)$ 
(since $\overline{X}(\hat{w})\setminus \overline{X}(w_1sw_2) = \overline{X}(w) \setminus \overline{X}(w_1s^2w_2)$) and henceforth that we have a  geometrical grading on 
$H^{2i}(\overline{X}(\hat{w}))$ by induced representations. Hence we see that we have such a grading on 
$H^{2i}(\overline{X}(v))=H^{2i}(\overline{X}(\hat{v})) \oplus 
H^{2i-2}(\overline{X}(w_1tw_2)).$ More precisely, we set for $z_1\preceq w_1$
and $z_2 \preceq w_2$,
$$R(z)=\left\{ \begin{array}{cc}
              z_1tstz_2 & \mbox{ if } z=z_1stsz_2 \\
	      z_1t_lsz_2 & \mbox{ if } z=z_1ts_rz_2 \\
	      z_1st_rz_2 & \mbox{ if } z=z_1s_ltz_2 \\
	      z_1z_2 & \mbox{ if } z=z_1z_2.
             \end{array}\right. $$
Then  the assignment $H(R(w))_{R(z)}:=H(w)_{z}$  for $z\not\in Q(s_l,z_1s^2z_2)$
and $H(v)_z=H(v/t_r)_{z/t_r}$ for $z\in Q(t_r,w_1t^2w_2)$  with $\ell(z)=i$
defines the desired grading on $H^{2i}(\overline{X}(v)).$


Concerning part b) we distinguish the following situations.

\noindent Subcase a). $u=R(x)=w_1tsw_2.$ (The case $u=R(x)=w_1stw_2$ behaves symmetrically)

We consider the viewpoint of Chow groups.
By induction hypothesis there are geometric gradings  on $A_{i}(\overline{X}(w_1w_2))_{\overline{\mQ}_\ell}$ and 
$A_{i}(\overline{X}(w_1tw_2))_{\overline{\mQ}_\ell}$ such that the homomorphisms 
$A_{i}(\overline{X}(w_1w_2))_{\overline{\mQ}_\ell} \to A_{i}(\overline{X}(w_1tw_2))_{\overline{\mQ}_\ell}$ induced by the inclusion 
is in diagonal form. Further we may suppose that we have a geometric grading on $A_{i}(\overline{X}(u))_{\overline{\mQ}_\ell}$ such that the  
map $A_{i}(\overline{X}(w_1tw_2))_{\overline{\mQ}_\ell} \to A_{i}(\overline{X}(u))_{\overline{\mQ}_\ell}$  is in diagonal form, as well. 
Let $f_w:A_{i}(\overline{X}(u)) \to A_{i}(\overline{X}(w))$ resp. $f_v:A_{i}(\overline{X}(u)) \to A_{i}(\overline{X}(v))$ be 
the homomorphisms induced by the inclusion. Again by assumption there is a grading $A_i(\overline{X}(w))_{\overline{\mQ}_\ell}
=\oplus_z A(w)_z$ such the map 
$f_w$ is graded. We have a natural commutative diagram
\begin{eqnarray*}
 & & \overline{X}(w) \\ &  \nearrow & \downarrow \pi_w \\   \overline{X}(w_1stw_2) & \to & \overline{X}(\hat{w})  .
\end{eqnarray*}
where $\pi_w: \overline{X}(w) \to \overline{X}(\hat{w})$ is the map of Lemma \ref{lemma_DMR}. 
It follows that the map $A_{i}(\overline{X}(u)) \to A_{i}(\overline{X}(\hat{w}))$ is graded, as well, if we consider on $A_{i}(\overline{X}(\hat{w}))_{\overline{\mQ}_\ell}$ the induced 
grading, i.e., $A(\hat{w})_z=\max\{\pi_w(A(w)_{z_1s_lz_2}), \pi_w(A(w)_{z_1s_rz_2}) \}$ for $z=z_1sz_2 \preceq  \hat{w}$ and
$A(\hat{w})_z=\pi_w(A(w)_z)$ for the remaining $z\preceq  \hat{w}$, , cf. Remark \ref{connection_symmetric group} and  \cite[Prop. 6.7]{Fu}.
In order to define the grading on  $A_{i}(\overline{X}(v))_{\overline{\mQ}_\ell}$ we
consider again the  splitting 
\begin{equation}\label{splitting_proof}
 A_{i}(\overline{X}(v))_{\overline{\mQ}_\ell}=A_{i}(\overline{X}(\hat{w}))_{\overline{\mQ}_\ell} \oplus 
A_{i-1}(\overline{X}(w_1t_rw_2))_{\overline{\mQ}_\ell}.
\end{equation}
 Now we apply the induction hypothesis to deduce the existence of a grading on the vector space $A_{i-1}(\overline{X}(w_1t_rw_2))_{\overline{\mQ}_\ell}$ 
 such that under the map
$$A_{i}(\overline{X}(w_1t_lw_2))_{\overline{\mQ}_\ell} \to A_{i-1}(\overline{X}(w_1t_rw_2))_{\overline{\mQ}_\ell}$$ 
we have 
$$A(w_1t_lw_2)_z = A(w_1t_rw_2)_{z/t_l}$$   for all $z$ 
with $t_l\mid z.$ But the latter map factorizes over  the map 
$$r:A_{i}(\overline{X}(w_1tsw_2))_{\overline{\mQ}_\ell} \to A_{i-1}(\overline{X}(w_1t_rw_2))_{\overline{\mQ}_\ell}.$$
Moreover, the contributions $A(w_1t_lw_2)_{z/t_r}$ 
with $t_l\mid z$ do not lie in the image of the  map $r$ since the trivial subrepresentations 
$i^G_G\subset   A(w_1t_rw_2)_{z/t_l}$  are linearly
independent from those of $A_{i}(\overline{X}(w_1tsw_2))_{\overline{\mQ}_\ell}$, as one deduces from \cite[Prop. 2.7]{L4}.
The claim follows.

\noindent Subcase b). $u=R(x)=v_1tstw_2$  with $x=v_1stsw_2$ (or $u=R(x)=w_1tstv_2$).

Here the result follows by writing $f_v$ as the sum of the homomorphisms
$$H^{2i}(\overline{X}(\hat{v})) \to H^{2i}(\overline{X}(\hat{u}))$$
and
$$H^{2i-2}(\overline{X}(w_1tw_2))(-1)\to H^{2i-2}(\overline{X}(v_1tw_2))(-1)$$
where $\hat{u}=v_1\widehat{sts}w_2$.

\noindent Subcase c). $u=w_1t^2w_2.$

Here the homomorphism  $H^{2i}(\overline{X}(v))  \to H^{2i}(\overline{X}(w_1t^2w_2))$  contracts by deleting the summand $H^{2i-2}(\overline{X}(w_1tw_2))(-1)$ on  both sides to the map  $H^{2i}(\overline{X}(\hat{v}))  \to 
H^{2i}(\overline{X}(w_1tw_2)).$ The latter one factorizes over $H^{2i}(\overline{X}(w_1tsw_2))$. By induction
the maps $H^{2i}(\overline{X}(\hat{w}))  \to H^{2i}(\overline{X}(w_1tsw_2))$ and $H^{2i}(\overline{X}(w_1tsw_2))  \to 
H^{2i}(\overline{X}(w_1tw_2))$ are graded. By Lemma \ref{independence} the bases can be chosen in a compatible way. Hence
the composite map is graded.

Alternatively one can avoid this case  by the proceeding lemma.

Part c). The remaining part of the theorem can be proved by induction. Here Lemma \ref{vanishing_odd_degree} serves as the start of the 
induction process. However, part c) is already a  consequence of the diagram (\ref{commutative_dia_chow}). 
To explain this we consider the viewpoint of Chow groups.

As in part b) we choose geometric gradings on $A_{i}(\overline{X}(v))_{\overline{\mQ}_\ell}$ and 
$A_{i}(\overline{X}(u))_{\overline{\mQ}_\ell}$
auch that the natural homomorphim $A_{i}(\overline{X}(u))_{\overline{\mQ}_\ell} \to A_{i}(\overline{X}(v))_{\overline{\mQ}_\ell}$ 
is graded. We need to show that the intersection of $A_{i}(\overline{X}(v))_{\overline{\mQ}_\ell}$ and 
$A_{i}(\overline{X}(\bar{v}))_{\overline{\mQ}_\ell}$
within $A_{i}(\overline{X}(\breve{v}))_{\overline{\mQ}_\ell}$ is as small as possible, i.e., induced by the image of 
$A_{i}(\overline{X}(u))_{\overline{\mQ}_\ell}$. For this, suppose that there is a constituent $A(v)_z=i^G_P$  with $r \mid z$
such that an irreducible subrepresentation $j_\lambda$  of $i^G_P$  lies in $A_{i}(\overline{X}(\bar{v}))_{\overline{\mQ}_\ell}$.
Let $x \in j_\lambda$  and  write $x=y_1 +y_2$ with uniquely determined 
elements $y_1 \in A_{i}(\overline{X}(\bar{v}))_{\overline{\mQ}_\ell}$ and $y_2 \in A_{i}(X(r,\breve{v}))_{\overline{\mQ}_\ell}$. Here we have fixed the natural splittings.
By what we saw in Section 4, it follows that $y_1=\sigma_\ast(x)$ for the cyclic shift map $\sigma: \overline{X}(v)
\to \overline{X}(\bar{v})$. But the map $\sigma$ is an isomorphism, more precisely we have
by Proposition \ref{Prop_homeomorphic} that $\sigma\circ \tau$ is an scalar mulitple of
the identity. It follows that $\sigma$ is up to a scalar a multiple of the morphism induced
by the identity element $e \in W.$ The upshot is that if $j_\lambda$ lies in $A_{i}(\overline{X}(\bar{v}))_{\overline{\mQ}_\ell}$
the same is true for the induced representation $i^G_{P_\lambda} \subset i^G_P$. But all the trivial subrepresentations 
$i^G_G =[\overline{X}(y)]_{\overline{\mQ}_\ell} \subset A(\breve{v})_y$ induced by the cycles $\overline{X}(y)$ with  
$y\preceq \breve{v}$ and $\ell(y)=i$ are linear independent 
in $A_i(\overline{X}(\breve{v}))_{\overline{\mQ}_\ell}$ by the result of Lusztig \cite[Prop. 2.7]{L4}. Hence we get a contradiction.
\end{proof}


\bigskip

 Let $u,v\in F^+$ with $u\prec v$ with $\ell(u)=\ell(v)-1.$ There is the obvious notion of  simultaneous transformation applied to the pair $(v,u)$ with 
 respect to the operations (I) - (III), as long as the corresponding  subword $s,st,sts$ is part of $u,$ as well. Apart from the critical situation where $v=v_1stsv_2$ and $u=v_1ssv_2$,
 we extend the simultaneous transformation to the tuple $(v,u)$ by letting act $C,K,R$ trivially on $u.$

\begin{lemma}\label{reduce_sw's}
 Let $u,v\in F^+$ with $u\prec v$, $\ell(u)=\ell(v)-1$ and $\h(v)\geq 1.$ Then there exists a sequence $(v_i,u_i),i=0,\ldots,n$, of such tuples such that
$(v,u)=(v_0,u_0)$, $v_n$ is of the shape $sv's$
and $(v_{i+1},u_{i+1})$ is induced simultaneously from $(v_i,u_i)$ via one of the operations (I) - (III). Here the situation that
$(v_i,u_i)=(v_1stsv_2,v_1s^2v_2)$ and $(v_{i+1},u_{i+1})=(v_1tstv_2,v_1t^2v_2)$
does not occur.
\end{lemma}

\begin{proof} 
By Theorem \ref{GKP} resp. Lemma \ref{doppelts} we may transform $v$ into the desired shape. We only have to analyse 
the critical case which we want to avoid.
After a series of cyclic shifts we may then suppose that  $v=tsv's$ and $u=sv's.$ If $\gamma(u)\not\in W$, then 
we may transform by Lemma \ref{doppelts} $u$ without cyclic shifts into the shape
$u_1r^2u_2$ for some $r\in S.$ But then $v$ can be transformed into the word $tu_1r^2u_2$ and the claim follows  using cyclic shifts again.
If $\gamma(u)\in W$ but $\gamma(v) \not\in W,$ then we may write $\gamma(u)=tw'$ for some $w'\in W$ and we are done 
again.

So it remains to consider the case where  $\gamma(v)\in W.$ We denote for any $z\in F^+$ and any 
$t=s_k\in S$ by $m_t(z)=m_k(z)\in \mZ_{\geq 0}$ its multiplicity in $z.$


Let $s=s_{m+1}$ and $t=s_m.$  Choose under all successive transformations (without using the forbidden replacement $sts \rightsquigarrow tst$) and all 
cyclic shifts (if we get an element which is not reduced, we are done by the first case)
starting with $v$ an element $z$ such that $m(z):=(m_{n-1}(z),m_{n-2}(z),\ldots,m_{2}(z),m_{1}(z))$ is maximal
for the lexicographical order. Then let $1 \leq i\leq n-1$  be the unique index   with $m_i(z) >m_k(z)$ for all $k>i$ and 
$m_i(z)\geq m_k(z)$ for all $k\leq i.$ Moreover, let $j\leq i$ such that $m_i(z)=m_{i-1}(z)=\cdots = m_j(z).$ 
We claim that $z$ has up to cyclic shift the desired shape $s_{j}z's_j.$
Suppose that the claim is wrong, i.e. $z=z_1s_jz_2s_jz_3$ with $z_1\neq e$ or 
$z_3\neq e$ and $s_j\not\in \supp(z_1) \cup \supp(z_3).$

\noindent {\it 1st Case:} $j>1.$ Consider two consecutive simple reflections $s_j$ together with the corresponding  subword $s_ju's_j$  of
$z$. If there is no simple reflection $s_{j-1}$ appearing in $u'$  there must be (as $\gamma(z)\in W$)  some simple reflection  $s_{j+1}$ in between.
Hence  we can finally replace an expression of the shape $s_ks_{k+1}s_k$ in $s_jus_j$ with $k\geq j$ by $s_{k+1}s_ks_{k+1}$. This gives a 
contradiction to the maximality of $m(z).$ On the other hand since $m_{j-1}< m_j$ there cannot be modulo cyclic shift such a reflection
between any such consecutive pair.

\noindent {\it 2nd Case:} $j={1}.$ In this case we easily see that it is possible to increase  the lexicographical order by a replacing
an expression $s_1s_2s_1$ by $s_2s_1s_2.$ Hence we get a contradiction, too.

The case $s=s_{m}$ and $t=s_{m+1}$ behaves symmetrically using the lexicographical order on 
$(m_{1}(z),m_{2}(z),\ldots,m_{n-2}(z),m_{n-1}(z))$.

\begin{eg}
 a) Let $(1,5)=(1,2)(2,3)(3,4)(4,5)(3,4)(2,3)(1,2)=sw's.$ Then $sw'=(1,2)(2,3)(3,4)(4,5)(3,4)(2,3)$. Shifting the simple
 reflection $(2,3)$ from the RHS to the LHS
 we get from $(v,u)=(sw',w')$ the tuple
$$((2,3)(1,2)(2,3)(3,4)(4,5)(3,4),(2,3)(2,3)(3,4)(4,5)(3,4))$$ which we want to avoid. Instead we consider the sequence
\begin{eqnarray*}
(v_1,u_1) & = & ((1,2)(2,3)(4,5)(3,4)(4,5)(2,3), (2,3)(4,5)(3,4)(4,5)(2,3))\\
& \vdots & \\
(v_4,u_4) & = &   ((4,5)(1,2)(2,3)(3,4)(2,3)(4,5), (4,5)(2,3)(3,4)(2,3)(4,5)).
\end{eqnarray*}

b) Let $v=sw'=(3,4)(2,3)(1,2)(3,4)(4,5)(2,3)$ and $u=(2,3)(1,2)(3,4)(4,5)(2,3).$ Again shifting the reflection $(2,3)$ from the RHS to the LHS, would yield
the (non-desirable) tuple.
Instead we consider the sequence
\begin{eqnarray*}
(v_1,u_1) & = & ((3,4)(2,3)(3,4)(1,2)(4,5)(2,3),(2,3)(3,4)(1,2)(4,5)(2,3)),\\
(v_2,u_2) & = & ((2,3)(3,4)(2,3)(1,2)(4,5)(2,3),(2,3)(3,4)(1,2)(4,5)(2,3)).
\end{eqnarray*}

\end{eg}

\end{proof}

We believe that the above theorem can be stated with respect to words in $\hat{F}^+.$
In particular the following conjecture should hold true.

\begin{conj}
Let $w\in \hat{F}^+$. The cohomology of $\overline{X}(w)$ in degree $2i$ can be written as
$$H^{2i}(\overline{X}(w))=\bigoplus_{z\preceq w\atop \ell(v)=i} H(w)_z$$
with $H(w)_z=i^G_{P^w_z}(-i)$ for certain standard parabolic subgroups $P^w_z\subset G.$ \qed
\end{conj}

\begin{rks}
i) The grading $$H^{2i}(\overline{X}(w))=\bigoplus_{z\preceq w\atop \ell(v)=i} H(w)_z$$ produced above  is  licentious
since it depends among other things on the chosen gradings with respect to the (relative) Coxeter elements in Levi subgroups.

ii) We can use Theorem \ref{restriction_graded}  in order to reprove the statement in
Remark \ref{vanishing_coh} concerning the appearance of  the Steinberg representation in the cohomology of a
DL-variety $X(w).$  Indeed, it is easily seen that
the induced representation $i^G_B$ occurs in the spectral sequence only in the contribution $H^0(\overline{X}(e)).$
Hence the $G$-representation $v^G_B$ occurs by Prop. \ref{qi_cox}  exactly in degree $\ell(w).$

iii) In \cite[Cor. 3.3.8]{DMR} the authors determine the character of $H^{2i}(\overline{X}(w))$ as $H$-representation. But for the author
it is not clear that their result leads to part a) of Theorem \ref{restriction_graded}.
\end{rks}

\vspace{0.5cm}

\section{The spectral sequence revisited}

In this section we reconsider the spectral sequence
$$E_1^{p,q}=\bigoplus_{v \preceq w \atop \ell(v)=\ell(w)-p} H^q(\overline{X}(v))
\Longrightarrow H^{p+q}_c(X(w))$$ of the previous paragraph
and treat the final aspect of the introduction.

\begin{conj}\label{complex_coh_w}
Let $w\in F^+$ and fix an integer $i\geq 0.$ For $v\preceq w$, there are geometric gradings
$H^{2i}(\overline{X}(v))=\bigoplus\limits_{z\preceq v \atop \ell(z)=i}i^G_{P^v_z}(-\ell(z))$ such that
the complex
\begin{equation*}
E_1^{\bullet,2i}: H^{2i}(\overline{X}(w)) \to  \bigoplus_{\genfrac{}{}{0pt}{2}{v\prec w}{\ell(v)=\ell(w)-1}} H^{2i}(\overline{X}(v)) \to \bigoplus_{\genfrac{}{}{0pt}{2}{v\prec w}{\ell(v)=\ell(w)-2}} H^{2i}(\overline{X}(v)) \to
\cdots \to  H^{2i}(\overline{X}(e))
\end{equation*}
is quasi-isomorphic to a direct sum $\bigoplus_{z\preceq w \atop \ell(z)=i} H(\,\cdot\,)_z$ of complexes of the shape
\begin{equation*}
H(\,\cdot\,)_z: i^G_{P^w_{z}} \rightarrow \bigoplus_{\genfrac{}{}{0pt}{2}{z \preceq v\preceq w}{\ell(v)=\ell(w)-1}}i^G_{P^v_{z}} \rightarrow
\bigoplus_{\genfrac{}{}{0pt}{2}{z\preceq v\preceq w}{\ell(v)=\ell(w)-2}}i^G_{P^v_{z}}
\rightarrow \dots \rightarrow i^G_{P^e_z}
\end{equation*}
as in section 1, cf. (\ref{complex_F}).
(Here the maps $i^G_{P^v_z} \to i^G_{P^{v'}_{z}}$ in the complex are induced - up to sign - by the double cosets of
$1$ in $W_{P^v_z}\setminus  W/W_{P^{v'}_{z}}$  via Frobenius reciprocity. Further $i^G_{P^v_{z}}=(0)$ if $z\npreceq v.$).
\end{conj}

By Proposition \ref{Coxeter_complex} the conjecture is true for Coxeter elements.
Thus we deal in what follows  with elements of positive height.

\begin{prop}\label{reduced_complex_sw's}
Let $w\in F^+$ with $\h(w)\geq 1.$ Then for proving the conjecture we may assume that $w$ is of the form $w=sw's.$
\end{prop}

\begin{proof}
We  apply again the operations (I) - (III) to the complex.
Suppose that the statement is true for $w$. We need to show that the assertion is true for the transformed element in $F^+$.

\smallskip

\noindent (I) Let $w=sw'$.
Then  we set  $H(C(v))_{C(z)}=H(v)_{z}$ for $v \preceq w$ and $z\preceq v.$  The corresponding complexes are clearly quasi-isomorphic.

\noindent (II) Let $w=w_1stw_2$ with $st=ts.$
We set  $H(K(v))_{K(z)}=H(v)_{z}$ for $v \preceq w$ and $z\preceq v.$  The corresponding resulting complexes are clearly quasi-isomorphic.

\noindent (III) Let $w=w_1stsw_2$ and $R(w)=w_1tstw_2$.
By Proposition \ref{coh_sts} we know that for $z_i \preceq w_i$, $i=1,2$,
the cohomology
$H^{2i}_c(X(s_l,z_1stsz_2))$ is a direct factor of $H^{2i}(\overline{X}(z_1stsz_2))$ and that
\begin{eqnarray*}
 & & H^{2i}(\overline{X}(z_1stsz_2))-H^{2i}_c(X(s_l,z_1s^2z_2)) =H^{2i}(\overline{X}(z_1\widehat{sts}z_2)) \\ 
 & = & H^{2i}(\overline{X}(z_1tstz_2))-H^{2i}_c(X(t_r,z_1t^2z_2)).
\end{eqnarray*}
Further $H^{2i}(\overline{X}(z_1s^2z_2))=H^{2i}(\overline{X}(z_1sz_2))\bigoplus H^{2i-2}(\overline{X}(z_1sz_2))(-1).$
Hence the complex $E_1^{\bullet,2i}$ for $w$ is quasi-isomorphic to

\smallskip
\noindent $\begin{array}{cccc}
& & \bigoplus\limits_{v_1 \prec w_1 \atop \ell(v_1)=\ell(w_1)-1} H^{2i}(\overline{X}(v_1\widehat{sts}w_2)) \to \cdots &
\\ & \nearrow & & \\
\\
H^{2i}(\overline{X}(\hat{w})) & \to &  H^{2i}(\overline{X}(w_1stw_2)) \oplus H^{2i}(\overline{X}(w_1tsw_2)) &  \to \cdots \\
\\ & \searrow & & \\
& & \bigoplus\limits_{v_2 \prec w_2 \atop \ell(v_2)=\ell(w_2)-1} H^{2i}(\overline{X}(w_1\widehat{sts}v_2)) \to \cdots &
\end{array}$

\medskip
\noindent which coincides with

\smallskip
\noindent
$\begin{array}{cccc}
& & \bigoplus\limits_{v_1 \prec w_1 \atop \ell(v_1)=\ell(w_1)-1} H^{2i}(\overline{X}(v_1\widehat{tst}w_2)) \to \cdots  &
\\ & \nearrow & & \\
\\
H^{2i}(\overline{X}(\hat{R(w)}))  & \to & H^{2i}(\overline{X}(w_1stw_2)) \oplus H^{2i}(\overline{X}(w_1tsw_2)) &  \to \cdots \\
\\ & \searrow & & \\
& & \bigoplus\limits_{v_2 \prec w_2 \atop \ell(v_2)=\ell(w_2)-1} H^{2i}(\overline{X}(w_1\widehat{tst}v_2)) \to \cdots  &
\end{array}$

By reversing the above argument, i.e. by adding the contributions $H_c^{2i}(X(t_r,z_1t^2z_2))$ and  $H^{2i-2}(X(z_1tz_2))(-1)$ to the complex
(with gradings chosen by induction) we see that the $E_1$-term of the transformed element $R(w)$  in degree $2i$ is quasi-isomorphic
to the complex $H(R(\cdot))_{\cdot}$.
\end{proof}

\begin{rk}
Let $w=sw's.$ For every $v'\preceq w'$, we have the identity
$H^{2i}(\overline{X}(sv's))=H^{2i}(\overline{X}(v's)) \bigoplus H^{2i-2}(\overline{X}(v's))(-1).$ Hence the complex (\ref{complex_coh_w}) is 
quasi-isomorphic to the complex

\begin{equation}\label{complex_reduced}\begin{array}{cccc}
& & H^{2i}(\overline{X}(sw')) \to \bigoplus\limits_{v' \prec w' \atop \ell(v')=\ell(w')-1} H^{2i}(\overline{X}(sv')) \bigoplus H^{2i}(X(w'))  &
\\ & \nearrow & \!\!\!\!\!\!\!\!\!\!\!\! \nearrow & \!\!\!\! \!\!\!\! \!\!\!\!  \to \cdots \\
\\
H^{2i-2}(\overline{X}(w's))(-1) & \to  \!\!\!\!\!\!\!\!\!\!\!\!\!\!\!\!\!\!\!\! \!\!\!\! \!\!\!\! \!\!\!\! \!\!\!\!  &  \bigoplus\limits_{v' \prec w'\atop \ell(v')=\ell(w')-1} H^{2i-2}(\overline{X}(v's))(-1) .& \\
\end{array}
\end{equation}

\noindent Moreover, by considering the middle column in the diagram (\ref{commutative_dia}), we see that  the lower line
which we may identify with the complex
$$H^{2i}_c(X(s,sw's)) \to \bigoplus\limits_{v' \prec w' \atop \ell(v')=\ell(w')-1} H^{2i}_c(X(s,sv's))  \to \cdots \to H^{2i}_c(X(s,ss))$$
is the direct sum of the complexes
$$H^{2i}_c(X(s^2,sw's))\langle -i \rangle \to \bigoplus\limits_{v' \prec w' \atop \ell(v')=\ell(w')-1} H^{2i}_c(X(s^2,sv's))\langle -i \rangle  \to \cdots \to H^{2i}_c(X(s^2))\langle -i \rangle$$
and
$$H^{2i}_c(X(s,sw')) \to \bigoplus\limits_{v' \prec w' \atop \ell(v')=\ell(w')-1} H^{2i}_c(X(s,sv'))  \to \cdots \to H^{2i}_c(X(s)).$$
\end{rk}

\vspace{0.5cm}

\begin{thm}\label{thm_conj}
 Let $w\in F^+$. Then Conjecture \ref{complex_coh_w} is true for $i=0,1,\ell(w)-1,\ell(w).$
\end{thm}

\begin{proof}
We may suppose that $w$ has full support. If $\h(w)=0$, then the claim follows from Proposition \ref{Coxeter_complex}. So we assume in the sequel that $\h(w)\geq 1.$

If $i=0$ then the complex coincides with the complex (\ref{complex_0_F}) of section 1 which yields the Steinberg representation $v^G_B$.

If $i=\ell(w)$, the claim is trivial.

If $i=\ell(w)-1$ the assertion follows from Corollary \ref{coro_l(w)-1}.

So let $i=1.$ The proof is by induction on the length.  By Proposition \ref{reduced_complex_sw's} we may assume that $w=sw's.$
For the start of induction, let $w=sw's$ be as in Corollary \ref{coro_height1_a} with $s=s_i$ and
$$w'=s_1\cdots s_{i-1} s_{i+1}\cdots s_{n-1}.$$ 
Here the Tate twist -1 contribution of the cohomology of $X(w)$ is given by 
$$H^\ast_c(X(w))\langle -1 \rangle=H^n_c(X(w))\langle -1 \rangle =v^G_{P(s)}(-1) - v^G_{P_{(2,1,1,\ldots,1)}}(-1).$$
The lower line in the complex (\ref{complex_reduced})
is nothing else but the complex $(\ref{complex_0_F_s})$ which is a resolution of $v^G_{P(s)}$. In view of Theorem \ref{restriction_graded} we define a grading on the upper line
as follows. Set for all $j\neq i,$
\begin{eqnarray*}
H(ss_j)_s & = & H(s_js)_e(-1)=i^G_{P(s_js)}(-1) 
\end{eqnarray*}
and
$$H(t)_t  = H(t)_e(-1)= i^G_{P(t)}(-1)\,\,\,\, \forall t\in S.$$
In particular,  we have fixed $H(ss_j)_{s_j}$ in this way for all $j  \neq i.$ Further we set for $j<i-1,$
$$H(s_js_{j+1})_{s_j}=i^G_{P(s_j)}(-1)$$
and
$$H(s_js_{j+1})_{s_{j+1}}=i^G_{P(s_{j+1})}(-1)$$
for $j>i.$ If $s_js_k=s_ks_j$, then there is a canonical grading  on $H^2(\overline{X}(s_js_k))$. Thus we have defined gradings for all subwords of $sw'$ of length $\leq 2.$
Now we extend the above gradings to the complex
$$0 \to H^{2}(\overline{X}(w')) \to  \bigoplus_{\genfrac{}{}{0pt}{2}{v'\prec w'}{\ell(v')=\ell(w')-1}} H^{2}(\overline{X}(v')) \to
\cdots \to \bigoplus_{\genfrac{}{}{0pt}{2}{v'\prec w'}{\ell(v')=1}}  H^{2}(\overline{X}(v'))$$
 which is induced by the K\"unneth formula and (compatible) gradings with respect to the the relative  Coxeter elements $s_1\cdots s_{i-1}$ and $s_{i+1}\cdots s_{n-1}$. More  precisely,
for $s_{i+1}\cdots s_{n-1}$ we consider the grading described by Proposition \ref{cohomology_Coxeter_compactification} 
whereas for $s_1\cdots s_{i-1}$ we consider the dual grading, i.e.,  induced
by blowing up hyperplanes (For $v'=w'$ the resulting grading coincides with the one in Proposition \ref{qi_restriction}).
Finally we  apply Theorem \ref{restriction_graded} b) once again in order to get the remaining gradings on the upper line in 
the complex
(\ref{complex_reduced}).  Here we have to make a choice for the grading on $H^2(\overline{X}(ss_{i-1}s_{i+1}))$, say 
$H(ss_{i-1}s_{i+1})_{s_{i-1}}\subset H(ss_{i-1}s_{i+1})_{s_{i+1}}$. The resulting graded complex satisfies the claim. 
Indeed, for  any $t \mid s_{i+1}\cdots s_{n-1}$ the complex $H(\cdot)_t$ is acyclic as 
$H(sv')_t=H(sv'/s_{i+1})_t$ for all $v'\preceq w'$ with $s_{i+1}\mid v'$.
Similarly, for $t\mid s_1\cdots s_{i-1}$ the complex $H(\cdot)_t$ is acyclic as $H(sv')_t=H(sv'/s_{i-1})_t$ 
for all  $v'\preceq w'$ with $s_{i-1}\mid v'$. Moreover $H(sv')_{s_{i+1}}=H(v')_{s_{i+1}}$ for all $v'\preceq w'$ which
shows that the complex $H(\cdot)_{s_{i+1}}$ is acyclic.
Finally, one checks that the complex $H(\cdot)_{s_{i-1}}$ is a resolution of the representation $v^G_{P(s)}(-1) - v^G_{P_{(2,1,1,\ldots,1)}}(-1).$
Moreover, we see that the differential 
\begin{equation}\label{last_differential}
 \bigoplus_{v\preceq w \atop \ell(v)=2} H^{2}(\overline{X}(v)) \to  \bigoplus_{t\prec w \atop \ell(t)=1} H^{2}(\overline{X}(t))
\end{equation}
is surjective.

Let's proceed with the induction step. So let $w=sw's\in F^+$ with $\h(sw')\geq 1.$

\noindent {\it Claim:} The map (\ref{last_differential}) is surjective, as well.

Here  we may consider the complex (\ref{complex_reduced}) again. By induction hypothesis we deduce that the map
$\bigoplus_{v\preceq sw' \atop \ell(v)=2} H^{2}(\overline{X}(v)) \to  \bigoplus_{t\prec sw' \atop \ell(t)=1} H^{2}(\overline{X}(t))$
is surjective. On the other hand, we have a surjection $H^{2}(\overline{X}(s^2))\to H^{2}(\overline{X}(s_r)).$ The claim follows.

We distinguish finally the following cases.

Case a). $s\in \supp(w')$. Then the lower line in the complex (\ref{complex_reduced}) coincides with the complex (\ref{complex_0_F_s}).
It is contractible by Proposition \ref{qi_cox_s}. By induction hypothesis
the statement is true for the upper line. We extend the grading to the complex  with respect to   $w$ in the obvious way.


Case b). $s\not\in \supp(w')$. Then the lower line in the complex (\ref{complex_reduced}) coincides with the complex (\ref{complex_0_F_s})
and gives a resolution of the generalized Steinberg representation $v^G_{P(s)}.$  By induction hypothesis
the statement is true for the upper line. As the map
$\bigoplus_{v\preceq sw' \atop \ell(v)=2} H^{2}(\overline{X}(v)) \to  \bigoplus_{t\prec sw' \atop \ell(t)=1} H^{2}(\overline{X}(t))$
is surjective the representation $v^G_{P(s)}(-1)$ occurs in the cohomology of $X(w)$.
Here, we extend the grading to the complex with respect to  $w$ in the obvious way, as well.
\end{proof}

By the proof of the preceding theorem we get an inductively formula for the Tate twist $-1$ contribution of the cohomology of DL-varieties.

\begin{coro}
 Let $w=sw's\in F^+$ with $\h(sw')\geq 1.$ Then

$$H^\ast_c(X(w))\langle -1 \rangle = \left\{\begin{array}{cc}
                                                     H^\ast_c(X(sw'))\langle -1 \rangle[-1] & \mbox{ if } s\in \supp(w') \\ \\
  H^\ast_c(X(sw'))\langle -1 \rangle[-1] \bigoplus v^G_{P(s)}(-1)[-\ell(w)] & \mbox{ if } s \notin \supp(w')
                                                    \end{array} \right. .$$
\end{coro}

\medskip
In view of the lower line in (\ref{complex_reduced}) we generalize Conjecture \ref{conjecture_coxeter_u}.

\begin{conj}
Let $w\in F^+$ and fix $u\prec w$. For $u\preceq v\preceq w$, there are geometric gradings $H^{2i}(\overline{X}(v))=\bigoplus_{z\preceq v \atop \ell(z)=i}H(v)_z$ such that
the complex
\begin{equation}\label{complex_coh_wu}
 0 \to H^{2i}(\overline{X}(w)) \to  \bigoplus_{\genfrac{}{}{0pt}{2}{u\preceq v\prec w}{\ell(v)=\ell(w)-1}} H^{2i}(\overline{X}(v)) \to
\cdots \to  H^{2i}(\overline{X}(u))
\end{equation}
is quasi-isomorphic to the graded direct sum of complexes of the shape

\begin{equation*}
0 \rightarrow  i^G_{P^w_{z}} \rightarrow \bigoplus_{\genfrac{}{}{0pt}{2}{z \preceq v\preceq w}{\ell(v)=\ell(w)-1}}i^G_{P^v_{z}} \rightarrow
\bigoplus_{\genfrac{}{}{0pt}{2}{z\preceq v\preceq w}{\ell(v)=\ell(w)-2}}i^G_{P^v_{z}}
\rightarrow \dots \rightarrow i^G_{P^e_z} \rightarrow 0
\end{equation*}
cf. (\ref{complex_F}).
\end{conj}

\begin{rk}
As in Proposition \ref{reduced_complex_sw's} one can reduce the conjecture to the case where $w$ is of the shape $w=sw's.$
Again we only have seriously to consider the operation (III) in this process of transformations.
So let $w=w_1stsw_2$ and $R(w)=\overline{w}=w_1tstw_2.$
If $$u=v_1stsv_2, v_1s^2v_2, v_1stv_2, v_1tsv_2, v_1v_2,$$ respectively, we
set $$\overline{u}=v_1tstv_2, v_1t^2v_2, v_1tsv_2, v_1stv_2, v_1v_2.$$  Then the complex (\ref{complex_coh_wu}) is quasi-isomorphic to
\begin{equation}
 0 \to H^{2i}(\overline{X}(\overline{w})) \to  \bigoplus_{\genfrac{}{}{0pt}{2}{u\preceq v\prec w}{\ell(v)=\ell(w)-1}} H^{2i}(\overline{X}(\overline{v})) \to
\cdots \to  H^{2i}(\overline{X}(\overline{u})).
\end{equation}
Similarly, we transform the lower complex in the conjecture above.
\end{rk}

\begin{rk}
If the conjecture is true then for  $w=sw's$ the gradings considered on $H^{2i}(\overline{X}(sw'))$ in the
upper line in (\ref{reduced_complex_sw's}) are not necessarily in the way that the induced complex is quasi-isomorphic to the complex
\begin{equation*}
 0 \to H^{2i}(\overline{X}(sw')) \to  \bigoplus_{\genfrac{}{}{0pt}{2}{v\prec sw'}{\ell(v)=\ell(w)-1}} H^{2i}(\overline{X}(v)) \to
\cdots \to  H^{2i}(\overline{X}(e)).
\end{equation*}
\end{rk}

\begin{rk}
Suppose that Conjecture \ref{complex_coh_w} is true. Then we can reprove the statement in Remark \ref{vanishing_coh} concerning the appearing of the trivial
representation in the cohomology of a DL-variety $X(w).$  Indeed, the multiplicity of the trivial representation $i^G_G$ in an induced
representation $i^G_P$ is always $1$. Hence for any $z\preceq w$, the  $v$-contribution
\begin{equation*}
0 \rightarrow  i^G_{P^w_{z}} \rightarrow \bigoplus_{\genfrac{}{}{0pt}{2}{z \preceq v\preceq w}{\ell(v)=\ell(w)-1}}i^G_{P^v_{z}} \rightarrow
\bigoplus_{\genfrac{}{}{0pt}{2}{z\preceq v\preceq w}{\ell(v)=\ell(w)-2}}i^G_{P^v_{z}}
\rightarrow \dots \rightarrow i^G_{P^e_z} \rightarrow 0
\end{equation*}
restricted to $i^G_G$ is acyclic  as the resulting index set   is contractible (a lattice).
It follows that the $i^G_G$ occurs only in the top cohomology group of $X(w).$
\end{rk}

\vspace{0.5cm}

\section{Examples}

Here we present some examples concerning Theorem \ref{thm_conj}.
In the following we omit the Tate twists for reasons of clarity. For $w\in F^+$ and $z\preceq w$, we write $i^G_P(z)$
instead of $H(w)_z=i^G_P.$

\bigskip

\noindent a) Let $G=\GL_3$ and let $w=(1,2)(2,3)(1,2)\in F^+$.

$\alpha)$ Let $i=2$. Here the complex (\ref{complex_reduced}) is

\medskip

\noindent
$\begin{array}{cccc}
& & H^{4}(\overline{X}((1,2)(2,3)))=i^G_G
\\ & \nearrow & & \\
\\
H^{2}(\overline{X}((2,3)(1,2)) & \to  &  H^{2}(\overline{X}((1,2)_r))=i^G_{P_{(2,1)}} .& \\
\end{array}$

\medskip

\noindent We consider the  grading $H^{2}(\overline{X}(2,3)(1,2))=i^G_G(2,3) \bigoplus i^G_{P_{(2,1)}}(1,2).$
Thus we see that the above complex is contractible, as it should be by Example \ref{Example3}.

\medskip
$\beta)$ Let $i=1$. Here the complex (\ref{complex_reduced}) is

\noindent
\begin{equation*}\begin{array}{ccc}
& \!\!\!\!\!\!\!\!\!\!\!\!\!\!\!\!\!\!\!\!\!\!\!\! \!\!\!\!\!\!\!\!\!\!\!\! \!\!\!\!\!\!\!\!\!\!\!\! \!\!\!\!\!\!\!\!\!\!\!\! \!\!\!\!\!\!\!\!\!\!\!\!\!\!\!\!\!\!\!\!\!\!\! \!\!\!\!\!\!\!\!\!\!\!\! \!\!\!\!\!\!\!\!\!\!\!\!  H^{2}(\overline{X}((1,2)(2,3))) \to & \!\!\!\!\!\!\!\!\!\!\!\! \!\!\!\!\!\!\!\!\!\!\!\! \!\!\!\!\!\!\!\!\!\!\!\! \!\!\!\!\!\!\!\!\!\!\!\!  H^{2}(\overline{X}((1,2))) \bigoplus H^{2}(X((2,3))) = i^G_{P_{(2,1)}} \bigoplus i^G_{P_{(2,1)}}
\medskip
\\ \nearrow & \!\!\!\!\!\!\!\!\!\!\!\! \nearrow &   \\
\\
H^{0}(\overline{X}((2,3)(1,2)))=i^G_G  \to &   H^{2}(\overline{X}((1,2)_r))= i^G_{P_{(2,1)}} .&  \\
\end{array}
\end{equation*}

\noindent We consider the grading $H^{2}(\overline{X}((2,3)(1,2)))=i^G_G(1,2) \bigoplus i^G_{P_{(2,1)}}(2,3).$
Thus we see that the above complex is contractible, as it should be by Example \ref{Example3}.

\medskip
\noindent b)  Let $G=\GL_4$ and let $w=(3,4)(1,2)(2,3)(3,4)\in F^+$.

$\alpha)$ Let $i=2.$ We have $H^4(\overline{X}((1,2)(2,3)))=i^G_{P_{(3,1)}}.$ We consider the grading
$$H^{4}(\overline{X}((3,4)(1,2)(2,3))=i^G_G((3,4)(2,3)) \bigoplus i^G_{P_{(3,1)}}((1,2)(2,3)) \bigoplus i^G_{P_{(2,2)}}((1,2)(3,4)).$$
The reduced complex (\ref{complex_reduced}) is given by

\medskip

\noindent $\begin{array}{cccc}
& &  H^{4}(\overline{X}((1,2)(2,3)))=i^G_{P_{(3,1)}}  &
\\ & \nearrow & & \\
H^{4}(\overline{X}((3,4)(1,2)(2,3))))  & \to &  H^{4}(\overline{X}((3,4)(1,2)))=i^G_{P_{(2,2)}}  &
\\ & \searrow & & \\
& &  H^{4}(\overline{X}((3,4)(2,3)))=i^G_{P_{(1,3)}} &
\end{array}$

\bigskip

\hspace{1cm} $\begin{array}{cccccc}
& & & \uparrow & & \\
& & & H^{2}(\overline{X}((1,2)(3,4)))  & &
\\ & \uparrow & \nearrow & & \searrow &   \\
\\ & H^2(\overline{X}(\cox)) & &  \uparrow & & H^2(\overline{X}((3,4)))=i^G_{P_{(1,1,2)}} .
\\ & & \searrow & & \nearrow & \\
& & & H^{2}(\overline{X}((2,3)(3,4)))   & &
\end{array}$

\bigskip

\noindent We consider the gradings
\begin{eqnarray*}
H^{2}(\overline{X}((2,3)(3,4))) & = & i^G_{P_{(1,3)}}(2,3) \bigoplus i^G_{P_{(1,1,2)}}(3,4),\\
H^{2}(\overline{X}((1,2)(3,4))) & = & i^G_{P_{(2,2)}}(1,2) \bigoplus i^G_{P_{(2,2)}}(3,4)
\end{eqnarray*}
and
$$H^{2}(\overline{X}(\cox))=i^G_G(2,3) \bigoplus i^G_{P_{(3,1)}}(3,4) \bigoplus i^G_{P_{(2,2)}}(1,2).$$

\noindent We get $H_c^\ast(X(w))\langle -2 \rangle = j_{(2,2)}[-5],$ as it should be by Example \ref{Example_(1,3,4)}.

\medskip

$\beta)$ Let $i=1.$ We consider for $w'=(1,2)(2,3)$ the graded complex

\medskip

\noindent $\begin{array}{cccc}
& &  H^{2}(\overline{X}((1,2)))=i^G_{P_{(2,1,1)}}  &
\\ & \nearrow & & \\
H^{2}(\overline{X}((1,2)(2,3)))  &  &   &
\\ & \searrow & & \\
& &  H^{2}(\overline{X}((2,3)))=i^G_{P_{(1,2,1)}}  &
\end{array}$

\bigskip

\noindent with $H^{2}(\overline{X}((1,2)(2,3)))=i^G_{P_{(3,1)}}(2,3) \bigoplus i^G_{P_{(2,1,1)}}(1,2).$ 

\noindent We consider further the gradings
\begin{eqnarray*}
 H^{2}(\overline{X}((3,4)(2,3))) & = & i^G_{P_{(3,1)}}(3,4) \bigoplus i^G_{P_{(1,2,1)}}(2,3),\\
H^{2}(\overline{X}((3,4)(1,2)) & = & i^G_{P_{(2,2)}}(1,2) \bigoplus i^G_{P_{(2,2)}}(3,4)
\end{eqnarray*}
\noindent and
$$H^{2}(\overline{X}((3,4)(1,2)(2,3)))=i^G_G(3,4) \bigoplus i^G_{P_{(3,1)}}(2,3) \bigoplus i^G_{P_{(2,2)}}(1,2).$$

\noindent The reduced complex (\ref{complex_reduced}) is given by

\medskip

\noindent $\begin{array}{cccccc}
& &  H^{2}(\overline{X}((1,2)(2,3))) & \to &  H^{2}(X((1,2)) = i^G_{P_{(2,1,1)}} &
\\ & \nearrow & & \nearrow \!\!\!\!\!\! \searrow \\
H^{2}(\overline{X}((3,4)(1,2)(2,3)))  & \to &  H^{2}(\overline{X}((3,4)(1,2))) &  \to &  H^{2}(\overline{X}((2,3)))=i^G_{P_{(1,2,1)}} &
\\ & \searrow & & \nearrow \!\!\!\!\!\! \searrow \\
& &  H^{2}(\overline{X}((3,4)(2,3))) & \to &  H^{2}(\overline{X}((3,4)))=i^G_{P_{(1,1,2)}} &
\end{array}$

\bigskip

\hspace{0.5cm} $\begin{array}{ccccc}
& & \uparrow & & \\
& & H^{0}(\overline{X}((1,2)(3,4)))=i^G_{P_{(2,2)}}  & &
\\ \uparrow & \nearrow & & \searrow &  \uparrow \\
\\H^0(\overline{X}(\cox))=i^G_G & &  \uparrow & & H^0(\overline{X}((3,4)))=i^G_{P_{(1,1,2)}}.
\\ & \searrow & & \nearrow & \\
& & H^{0}(\overline{X}((2,3)(3,4)))=i^G_{P_{(1,3)}}   & &
\end{array}$

\medskip

It follows that the complex is contractible, as it should be by Example \ref{Example_(1,3,4)}.

\medskip
\noindent  c)  Let $G=\GL_4$ and let $w=(2,3)(1,2)(3,4)(2,3)\in F^+$.

$\alpha)$ Let $i=2.$ We have $H^4(\overline{X}((1,2)(3,4)))=i^G_{P_{(2,2)}}.$ We consider the grading
$$H^{4}(\overline{X}((2,3)(1,2)(3,4))=i^G_G((2,3)(1,2)) \bigoplus i^G_{P_{(3,1)}}((2,3)(3,4)) \bigoplus i^G_{P_{(2,2)}}((1,2)(3,4)).$$
The reduced complex (\ref{complex_reduced}) is given by

\medskip

\noindent $\begin{array}{cccc}
& &  H^{4}(\overline{X}((1,2)(3,4)))=i^G_{P_{(2,2)}}  &
\\ & \nearrow & & \\
H^{4}(\overline{X}((2,3)(1,2)(3,4))))  & \to &  H^{4}(\overline{X}((2,3)(1,2)))=i^G_{P_{(3,1)}}  &
\\ & \searrow & & \\
& &  H^{4}(\overline{X}((2,3)(3,4)))=i^G_{P_{(1,3)}} &
\end{array}$

\bigskip

\hspace{-0.5cm} $\begin{array}{cccccc}
& & & \uparrow & & \\
& & & H^{2}(\overline{X}((1,2)(2,3)))  & &
\\ & \uparrow & \nearrow & & \searrow &   \\
\\ & H^2(\overline{X}((1,2)(3,4)(2,3))) & &  \uparrow & & H^2(\overline{X}((2,3)))=i^G_{P_{(1,2,1)}}.
\\ & & \searrow & & \nearrow & \\
& & & H^{2}(\overline{X}((3,4)(2,3)))   & &
\end{array}$

\bigskip

\noindent We consider the gradings
\begin{eqnarray*}
H^{2}(\overline{X}((3,4)(2,3))) & = & i^G_{P_{(1,3)}}(3,4) \bigoplus i^G_{P_{(1,2,1)}}(2,3),\\
H^{2}(\overline{X}((1,2)(2,3))) & = & i^G_{P_{(3,1)}}(1,2) \bigoplus i^G_{P_{(1,2,1)}}(2,3)
\end{eqnarray*}
and
$$H^{2}(\overline{X}((1,2)(3,4)(2,3)))=i^G_G(1,2) \bigoplus i^G_{P_{(3,1)}}(3,4) \bigoplus i^G_{P_{(2,2)}}(2,3).$$

\noindent We get $H_c^\ast(X(w))\langle -2 \rangle = i^G_{P_{(2,1,1)}}/i^G_{P_{(2,2)}}[-5],$ as it should be by Example \ref{Example_(1,3)(2,4)}.

\medskip

$\beta)$ Let $i=1.$ We consider for $w'=(1,2)(3,4)$ the graded complex

\medskip

\noindent $\begin{array}{cccc}
& &  H^{2}(\overline{X}((1,2)))=i^G_{P_{(2,1,1)}}  &
\\ & \nearrow & & \\
H^{2}(\overline{X}((1,2)(3,4)))  &  &   &
\\ & \searrow & & \\
& &  H^{2}(\overline{X}((3,4)))=i^G_{P_{(1,1,2)}}  &
\end{array}$

\bigskip

\noindent with $H^{2}(\overline{X}((1,2)(3,4)))=i^G_{P_{(2,2)}}(1,2) \bigoplus i^G_{P_{(2,2)}}(3,4).$
We consider further the gradings
\begin{eqnarray*}
 H^{2}(\overline{X}((2,3)(3,4))) & = & i^G_{P_{(1,3)}}(2,3) \bigoplus i^G_{P_{(1,1,2)}}(3,4),\\
H^{2}(\overline{X}((2,3)(1,2))) & = & i^G_{P_{(3,1)}}(2,3) \bigoplus i^G_{P_{(2,1,1)}}(1,2)
\end{eqnarray*}
\noindent and
$$H^{2}(\overline{X}((2,3)(1,2)(3,4)))=i^G_G(2,3) \bigoplus i^G_{P_{(3,1)}}(3,4) \bigoplus i^G_{P_{(2,2)}}(1,2).$$

\noindent The reduced complex (\ref{complex_reduced}) is given by

\medskip

\noindent $\begin{array}{cccccc}
& &  H^{2}(\overline{X}((2,3)(1,2))) & \to &  H^{2}(X((1,2))) = i^G_{P_{(2,1,1)}} &
\\ & \nearrow &  & \nearrow \!\!\!\!\!\! \searrow \\
H^{2}(\overline{X}((2,3)(1,2)(3,4)))  & \to &  H^{2}(\overline{X}(((1,2)(3,4))) &  \to &  H^{2}(\overline{X}((2,3)))=i^G_{P_{(1,2,1)}} &
\\ & \searrow & & \nearrow \!\!\!\!\!\! \searrow  \\
& &  H^{2}(\overline{X}((2,3)(3,4))) & \to &  H^{2}(\overline{X}((3,4)))=i^G_{P_{(1,1,2)}} &
\end{array}$

\bigskip

\hspace{-0.5cm} $\begin{array}{ccccc}
& & \uparrow & & \\
& & H^{0}(\overline{X}((1,2))(2,3))=i^G_{P_{(3,1)}}  & &
\\ \uparrow & \nearrow & & \searrow &  \uparrow \\
\\H^0(\overline{X}((1,2)(3,4)(2,3)))=i^G_G & &  \uparrow & & H^0(\overline{X}((2,3)))=i^G_{P_{(1,2,1)}}.
\\ & \searrow & & \nearrow & \\
& & H^{0}(\overline{X}((3,4)(2,3)))=i^G_{P_{(1,3)}}   & &
\end{array}$

\medskip

\noindent We get $H_c^\ast(X(w))\langle -1 \rangle = j_{(2,2)}[-4],$ as it should be by Example \ref{Example_(1,3)(2,4)}.

\vspace{0.5cm}

\section{Appendix A}\label{long_version}

In this section we present a different but more vague method for determining the cohomology of DL-varieties attached to elements in the
Weyl group. The analogous situation for elements of the monoid $F^+$ can be treated in the same way. Here the approach is the other way round compared
to the foregoing version. Some results presented here might be also treated by using Demazure resolutions and the results of the previous sections.

Let $w=sw's \in W$ with $\ell(w)=\ell(w')+2$ and $Z=X(w) \cup X(sw')\subset X$  as before. Reconsider for $i\geq 0,$
the natural map $r^i=r^i_{w,sw'}:H^{i}_c(Z) \to H^{i}_c(X(sw')).$ Now we write $H^{i}_c(Z)=H^{i-2}_c(Z')(-1)=A \bigoplus B$ where
$$A \cong {\rm coker} \big(H^{i-3}_c(X(w')) \to H^{i-2}_c(X(w's))\big)(-1)$$ and $$B=\ker \big(H^{i-2}_c(X(w')) \to H^{i-1}_c(X(w's))\big)(-1).$$
By Remark \ref{contribute} we know that $r^i_{\mid A}=0.$

Motivated by the Examples in the  $\GL_4$-case 
we pose the following conjecture.


\begin{conj}\label{si-surjective}
For $i\geq 0$, the map $r^i_{\mid B}: B  \to H^{i}_c(X(sw'))$ has si-full rang.
\end{conj}

\begin{rk}\label{rk_sisurjective}
For our purpose it is even enough to have the validity of a weaker form of the above conjecture. More precisely, it suffices to know that
for all irreducible $H$-representations $V$ with $i=-2t(V)$ the map $r^{i}_{\mid B^V}: B^V \to H^{2i}_c(X(sw'))^V$ has full rang.

Indeed if the assumption $i=-2t(V)$ is not satisfied, then we proceed as follows to determine the $V$-isotypic part of the map $r^i_{w,sw'}.$
We fix a reduced decomposition of $w'$. Since $i>-2t(V)$ there is by purity of $\overline{X}(w)$  some
hypersquare $Q\subset {F^+}$ of dimension $d$ (which is assumed to be minimal)
 with head $w$ and with $\{w,sw'\}\subset Q$ and such that a given irreducible subrepresentation $W\subset B^V \subset H^i_c(X(w) \cup X(sw'))$ 
is induced by an isomorphic subrepresentation $W' \subset 
H^{i-1}_c(X(Q)\setminus X(w) \cup X(sw'))$ via the corresponding boundary map $\delta^{i-1}.$
A minimal hypersquare exists since the extreme case where $\tail(Q)=e$  yields one. Then $\tail(Q)=sv'$ or $\tail(Q)=v'$ for some $v'\in {F^+}$
with $v'\preceq w'$. 

Suppose that $d=2.$

\noindent 1. Case.  $\tail(Q)=sv'.$ Thus $W'\subset H^{i-1}_c(X(sv's) \cup X(sv'))$  maps onto
$W\subset H^i_c(X(w)\cup X(sw'))$ via the boundary map $\delta^{i-1}$.

Subcase a) $W'$ is induced by  $H^{i-1}_c(X(sv'))$.

Subsubcase i) $\delta_{sw'sv'}^{i-1}(W') \neq 0$ where  $\delta_{sw'sv'}^{i-1}$ is the boundary map
$H^{i-1}_c(X(sv')) \to H^i_c(X(sw'))$ (which is known by induction, cf. the following pages).  
In this case $r_{w,sw'}^i$ maps $W\subset H^i_c(X(w)\cup X(sw'))$ onto $\delta_{sw'sv'}^{i-1}(W')  \subset H^i_c(X(sw'))$
by considering the commutative diagram:
\medskip

$\begin{array}{ccccccc}
 & \vdots & & \vdots & & \vdots &  \\
 & \uparrow & & \uparrow & & \uparrow &  \\
\cdots \to & H^i_c(X(w)) & \to &  H_c^{i}(X(w)\cup X(sw')) & \to & H^i_c(X(sw')) &  \to \cdots \\
 & \uparrow & &  \;\;\;\;\;\;\;\; \uparrow  \delta^{i-1}& & \;\;\;\;\;\;\;\;\;\;\; \uparrow \delta^{i-1}_{sw',sv'} &  \\
\cdots \to &  H_c^{i-1}(X(v)) & \to &  H_c^{i-1}(X(v) \cup X(sv')) &  \to & H^{i-1}_c(X(sv')) &  \to \cdots \\
 & \uparrow  & & \uparrow  & & \uparrow &  \\
 & \vdots & & \vdots & & \vdots &  \\
\end{array}$

\medskip
Subsubcase ii) $\delta_{sw'sv'}^{i-1}(W') = 0$. Then $r^i(W)=0.$

Subcase b) $W'$ is induced by  $H^{i-1}_c(X(v))$. In this case $r^i(W)=0$.

\noindent 2. Case. $\tail(Q)=w'.$  This case yields  a contradiction since the boundary map $H^{i-1}_c(X(w's) \cup X(w')) \to H^i_c(X(w) \cup X(sw'))$
vanishes by Corollary \ref{cohomology_pb}.

\medskip

Suppose that $d=3.$

\noindent 1. Case $\tail(Q)=sv'$ for some $v'\prec w'$ with $\ell(w')=\ell(v')+1,$ i.e. $Q$ has the shape

$Q:\;\,\; \begin{array}{ccccc}
 & & w & & \\ & \nearrow & \uparrow & \nwarrow & \\ sw' & & sv_1's & & sv_2's  \\ \uparrow & \nearrow \nwarrow & & \nearrow \nwarrow & \uparrow
\\ sv_1'& & sv_2' & & sv's  \\ & \nwarrow & \uparrow & \nearrow & \\  & & sv' & &
\end{array}
$

\noindent for some $v_1',v_2' \in F^+.$ We set $A=\{sv_1',sv_2',sv'\}$ and $=\{sv_1's,sv_2's,sv's\}.$ Then $X(A)$ is closed in $X(Q)\setminus (X(w)\cup X(sw'))$ whereas $X(B)$
is open in the latter space. Now we may imitate the procedure of the case $d=2.$ The variety $X(A)$ corresponds to $X(sv')$ whereas
$X(B)$ plays the role of $X(v).$

\noindent 2. Case $\tail(Q)=v'$ for some $v'\prec w'$ with $\ell(w')=\ell(v')+1,$ i.e., $Q$ has the shape

$Q:\;\,\; \begin{array}{ccccc}
 & & w & & \\ & \nearrow & \uparrow & \nwarrow & \\ sw' & & sv's & & w's  \\ \uparrow & \nearrow \nwarrow & & \nearrow \nwarrow & \uparrow
\\ sv'& & w' & & v's  \\ & \nwarrow & \uparrow & \nearrow & \\  & & v' & &
\end{array}
$

We claim that $r_{w,sw'}^i$ is trivial. Indeed, as $X(sv's) \cup X(sv')$ and $X(w's) \cup X(w')$ are both open in
$X(Q)\setminus (X(w) \cup X(sw'))$ whereas $X(v's) \cup X(v')$ is closed we see that

\begin{itemize}
 \item $W'\subset H^{i-1}_c(X(sv's) \cup X(sv'))$  gives a contradiction to the minimality with respect to $d$.
\item  $W'\subset H^{i-1}_c(X(w's) \cup X(w'))$ gives a contradiction as the boundary map  is trivial.
\end{itemize}
So $W'$ is induced by $H^{i-1}_c(X(v's) \cup X(v'))$ and it is mapped to $W\subset H^i_c(X(Q(w',w)))$ via the boundary map.
But the latter one is given by the diagram
\medskip

$\begin{array}{ccccc}
\vdots & & \vdots &  & \vdots   \\
\uparrow & & \uparrow &  & \uparrow   \\
 H^i_c(X(Q(w',w)) & = &  H_c^{i}(X(w)\cup X(sw')) & \oplus & H^i_c(X(w's)\cup X(w'))  \\
 \uparrow & & \uparrow & & \uparrow   \\
 H_c^{i-1}(X(Q(v',v)) & = &  H_c^{i-1}(X(v) \cup X(sv')) & \oplus & H^{i-1}_c(X(v's) \cup X(v')) \\
 \uparrow  & & \uparrow & & \uparrow  \\
\vdots & & \vdots &  & \vdots   \\
\end{array}$

\medskip
\noindent i.e. it is for trivial reasons the direct sum of the summands. Hence we get a contradiction.

The higher dimensional cases $d\geq 4$ behave as above whether $\tail(Q)=sv'$ or $\tail(Q)=v'$. The latter case gives a contradiction. \qed
\end{rk}

Thus under the validity of the above conjecture for determining the cohomology of $X(w)$, it remains to compute the cohomology of the edge $X(w') \cup X(w's)$ which we explain now.

Let $w,v \in W$ with $\ell(w)=\ell(v)+1.$ We want to determine the cohomology of the locally closed subvariety $X(w) \cup X(v)\subset X.$
Suppose that we may write  $w=sw's$ as
in the previous sections. If $v\in \{sw',w's\}$ then $H^i_c(X(w) \cup X(v))= H^{i-2}_c(X(w') \cup X(w's))(-1)$ and we may suppose
by induction on the length of $w$ that these groups are known.

So let $v=sv's$ with $v' <w'.$ Then the cohomology of $X(w)\cup X(v)$ sits in a long exact cohomology sequence
$$\cdots \to H^i_c(X(w) \cup X(v)) \to H^i_c(X(Q)) \to H^i_c(X(sw') \cup X(sv')) \to \cdots $$
where $Q$ is the  square $Q=\{w,v,sw',sv'\} \subset W$.  The cohomology of
$X(sw') \cup X(sv')$ is known by induction on the length of $w's.$ On the other hand, the square
\smallskip

$
Q:\,\,\, \begin{array}{ccccc}
 & & w & & \\ & \nearrow & & \nwarrow & \\ v & & & & sw' . \\ & \nwarrow & & \nearrow & \\ & & sv' & &
\end{array}
$

\medskip
\noindent is induced by the square

\smallskip
$
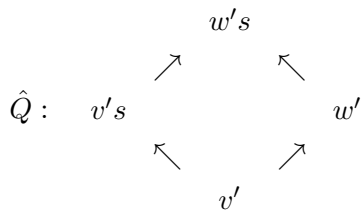
\hat{Q}:\;\,\;\begin{array}{ccccc}
 & & w's & & \\ & \nearrow & & \nwarrow & \\ v's & & & & w'  \\ & \nwarrow & & \nearrow & \\ & & v' & &
\end{array}
$

\medskip
\noindent via multiplication with $s\in S$ from the left. The union $Q\cup \hat{Q}$ gives rise to a cube or a 3-dimensional hypersquare
\smallskip

$Q\cup \hat{Q}:\;\,\; \begin{array}{ccccc}
 & & w & & \\ & \nearrow & \uparrow & \nwarrow & \\ sw' & & v & & sw'  \\ \uparrow & \nearrow \nwarrow & & \nearrow \nwarrow & \uparrow
\\ sv'& & w' & & v's  \\ & \nwarrow & \uparrow & \nearrow & \\  & & v' & &
\end{array}
$

\noindent Then by Proposition \ref{hypersquares_surjective} the variety $X(Q)$ is an $\mA^1$-bundle over $X(\hat{Q})$ and $X(Q \cup \hat{Q})$
is a $\mP^1$-bundle over $X(\hat{Q})$. Further the restriction map in cohomology
$$r^i_{Q,\{sw',sv'\}}:H^i_c(X(Q)) \to H^i_c(X(sw') \cup X(sv'))$$
can be computed in the same way as in Conjecture \ref{si-surjective} resp. Remark \ref{rk_sisurjective}. Thus if we are able to determine the cohomology
group $H^\ast_c(X(\hat{Q}))$ we have  knowledge of the
cohomology of $X(w) \cup X(v).$ Hence we have  reduced the question of determining the cohomology of the edge $X(w) \cup X(v)$ by the prize of
enlarging the square but where the head has smaller length.

\begin{eg}
Let $G=\GL_4.$ Let $w=(1,4)(2,3)\in W.$ Here we write  $w=sw's$ $=(3,4)(1,3)(2,4)(3,4)$ so that
$sw'=(2,3)(3,4)(2,3)(1,2)(2,3)$, $w'=(2,3)(1,2)(3,4)(2,3)$  and  $w's=(2,3)(1,2)(2,3)(3,4)(2,3)$.
Hence for computing the cohomology of the edge $X(w's) \cup X(w')$ we consider the 
square 

$
Q:\,\,\, \begin{array}{ccccc}
 & & w's & & \\ & \nearrow & & \nwarrow & \\ w'   & & & & (2,3)(1,2)(2,3)(3,4) \\ & \nwarrow & & \nearrow & \\ & & (2,3)(1,2)(3,4) & &
\end{array}
$ 

\medskip
and henceforth the square
\medskip

$
\hat{Q}:\,\,\, \begin{array}{ccccc}
 & & (1,2)(2,3)(3,4)(2,3) & & \\ & \nearrow & &  \nwarrow & \\ (1,2)(3,4)(2,3) & & & & (1,2)(2,3)(3,4) . \\ & \nwarrow & & \nearrow & \\ & & (1,2)(3,4) & &
\end{array}
$ 

The cohomology of $X(\hat{Q})$ is by using the results in section 5 easily computed as
\begin{eqnarray*}
 H^\ast_c(X(\hat{Q})) & = & v^G_{P_{(1,2,1)}}[-2] \oplus j_{(3,1)}(-1)[-3]^2 \oplus i^G_G(-2)[-4] \\
& &  \oplus j_{(3,1)}(-2)[-5]\oplus i^G_G(-3)[-6]^2 \oplus  i^G_G(-4)[-8].
\end{eqnarray*}
Hence
\begin{eqnarray*}
 H^\ast_c(X(Q)) & = & v^G_{P_{(1,2,1)}}(-1)[-4] \oplus j_{(3,1)}(-2)[-5]^2 \oplus i^G_G(-3)[-6] \\
& &  \oplus j_{(3,1)}(-3)[-7]\oplus i^G_G(-4)[-8]^2 \oplus  i^G_G(-5)[-10].
\end{eqnarray*}
On the other hand the cohomology of $Y=X((2,3)(1,2)(2,3)(3,4)) \cup X((2,3)(1,2)(3,4))$ is given 
(using loc.cit.) by
$$ H^\ast_c(Y)=v^G_{P_{(2,1,1)}}(-1)[-4] \oplus j_{(2,2)}(-2)[-5] \oplus j_{(3,1)}(-2)[-5]\oplus i^G_G(-3)[-6] \oplus  i^G_G(-4)[-8].$$
By a careful study of the restriction map we compute that the cohomology of $X(w's)\cup X(w')$ is hence given by
\begin{eqnarray*}H^\ast_c(X(w's)\cup X(w')) & = & j_{(2,2)}(-1)[-4] \oplus j_{(3,1)}(-2)[-5] \oplus j_{(2,2)}(-2)[-6] \\
& &  \oplus j_{(3,1)}(-3)[-7] \oplus  i^G_G(-4)[-8] \oplus  i^G_G(-5)[-10].
\end{eqnarray*}
Exemplarily, the contribution $j_{(3,1)}(-2)[-5]$ in $H^\ast_c(Y)$ is induced by the Coxeter element $(2,3)(1,2)(3,4).$  
Further one checks that this representation lies in the  image of the restriction map 
$H^\ast_c(X((1,3,4))\cup X((2,3)(1,2)(3,4))) \to H^\ast_c(X((2,3)(1,2)(3,4))).$ Finally one verifies
that the representation on the LHS is induced by $H^5(X(Q))$ which implies that $j_{(3,1)}(-2) \subset H^5_c(Y)$
is killed  by $H^5_c(X(Q)).$ Alternatively, on can apply the method presented in Remark \ref{rk_sisurjective}.
With the same methods one deduces that
\begin{eqnarray*}
 H^\ast_c(X((1,4)(2,3))) & = & v^G_B[-6] \oplus j_{(2,1,1)}(-2)[-7] \oplus j_{(2,2)}(-3)[-8]^2 \\
& &  \oplus j_{(3,1)}(-4)[-9] \oplus  i^G_G(-6)[-12].
\end{eqnarray*}
Again we consider exemplarily the contribution $j_{(2,2)}(-2) \subset H^6_c(X(sw')).$ It is induced by
$H^5_c(X((1,4,3))).$ The set $Q'=\{w,sw',(1,4),(1,4,3)\}$ is a square and the restriction map
$H^5_c(X((1,4))\cup X((1,4,3)))  \to H^5_c(X((1,4,3)))$ is surjective. By computing $H^4_c(X(Q'))=H^2_c(X(\hat{Q'}))(-1)$
with $Q'=\{w's,w',(1,3,4),(1,3)\}$ one verifies that $j_{(2,2)}(-2) \subset H^6_c(X(sw'))$ is killed by the restriction
map.
\end{eg}

\begin{rk}
In general we have to apply to $w\in W$ the operations (I) - (III) (in the sense of Weyl groups) of the previous section in order to write it in the shape $w=sw's,$
cf. Lemma \ref{reduce_sw's}. In what follows, we hope that this Lemma (or rather a variant)  generalizes to arbitrary hypersquares.
\end{rk}

The reader might expect how the strategy  works in higher dimensions. Here we suppose that a similar conjecture as above  holds true.
Hence we have to determine the cohomology of $X(Q)$ for squares
$Q\subset W.$ So let $Q$ be such a square with head $w=sw's.$
\medskip

\noindent {\it Case 1:} $Q$ is of the shape

$
Q:\,\,\, \begin{array}{ccccc}
 & & w & & \\ & \nearrow & & \nwarrow & \\ sw' & & & & w's  \\ & \nwarrow & & \nearrow & \\ & & w' & &
\end{array}
$.

\noindent In this case the we have by Proposition \ref{cohomology_pb} a splitting $H^i_c(X(Q))=H^i_c(X(w's) \cup X(w')) \oplus H^{i-2}_c(X(w's) \cup X(w'))(-1).$ By induction  the cohomology
of $X(w's)\cup X(w')$ and hence of $X(Q)$ is known.

\medskip

\noindent {\it Case 2:} $Q$ is of the shape

$
Q:\,\,\, \begin{array}{ccccc}
 & & w & & \\ & \nearrow & & \nwarrow & \\ v & & & & sw'  \\ & \nwarrow & & \nearrow & \\ & & sv' & &
\end{array}
$

\noindent with $v=sv's$. In this case the we have $H^i_c(X(Q))=H^{i-2}_c(X(\hat{Q}))(-1)$ where $\hat{Q}=\{w's,v's,w',v'\}.$ By induction the cohomology
of $X(\hat{Q})$ and hence of $X(Q)$ is known.

\medskip

\noindent {\it Case 3:} $Q$ is of the shape

$
Q:\,\,\, \begin{array}{ccccc}
 & & w & & \\ & \nearrow & & \nwarrow & \\ v_1 & & & & v_2  \\ & \nwarrow & & \nearrow & \\ & & v_3 & &
\end{array}
$

\noindent with $v_i=sv_i's$  for $i=1,2,3$ and where $Q':=\{w',v_1',v_2',v_3'\}\subset W$ is a square. In this case  we consider
as in the case of edges above the square $sQ'=\{sw',sv_1',sv_2',sv_0' \}.$
Then the cohomology of $X(Q)$ sits in a long exact cohomology sequence
$$\cdots \to H^{i-1}_c(X(sQ') \to H^i_c(X(Q)) \to H^i_c(X(Q)\cup X(sQ')) \to H^i_c(X(sQ')) \to \cdots .$$
Again by induction on the length of the head of a cube the cohomology of $X(sQ')$, $X(Q's)$ and $X(Q')$ are known. Further we
have $H^i_c(X(Q) \cup X(sQ'))=H^{i-2}_c(X(Q's)\cup X(Q'))(-1),$
where $Q's:=\{w's,v_1's,v_2's,v_3's\}.$
Thus we may compute $r^i_{Q\cup sQ',sQ'}$ as in the lower dimensional cases, i.e. as in Conjecture \ref{si-surjective} resp.
Remark \ref{rk_sisurjective}.

But there are yet two other  kind of squares.

\medskip

\noindent {\it Case 4:} $Q$ is of the shape

$
Q:\,\,\, \begin{array}{ccccc}
 & & w & & \\ & \nearrow & & \nwarrow & \\ sw' & & & & w's  \\ & \nwarrow & & \nearrow & \\ & & su' & &
\end{array}
$

\noindent with $su'=u's$.

\vspace{0.5cm}

\noindent and

\vspace{0.5cm}

\noindent {\it Case 5:} $Q$ is of the shape

$
Q:\,\,\, \begin{array}{ccccc}
 & & w & & \\ & \nearrow & & \nwarrow & \\ sw' & & & & sv's  \\ & \nwarrow & & \nearrow & \\ & & su's & &
\end{array}
$

\noindent with $u' < v' <w'$ and $\ell(w')=\ell(u') +2.$

\bigskip

Unfortunately, as we see, the higher the dimension of a square is the more complicated the situation behaves. This is of course due to
the relations which exist in the Weyl group $W.$
For this reason, we also consider the monoid $F^+$ in the sequel where this phenomenon does not appear.

Before we proceed we recall the following well-known fact.

\begin{lemma}\label{lemma_DL}
Let $w \in W$ and $s,t\in S$. Suppose that $\ell(sw)=\ell(w)+1$, $\ell(wt)=\ell(w)+1$ and $\ell(swt)=\ell(w).$ Then
$w=swt.$
\end{lemma}

\begin{proof}
 This is \cite[Lemma 1.6.4]{DL}
\end{proof}

Thus if $\ell(sw')=\ell(w's)=\ell(w')+1$ and $\ell(sw's)=\ell(w')$ we have $sw'=w's.$

\vspace{0.5cm}

Let $Q\subset W$ be a hypersquare  with head  $w=sw's$. There are a priori for the tail of $Q$ the following possibilities (mod symmetry) where $u'\leq w'$.

\noindent {\bf Case A:} $\tail(Q)=u'$ with $\ell(su's)=\ell(u')+2.$

\noindent {\bf Case B:} $\tail(Q)=u'$ with $\ell(su')=\ell(u's)=\ell(u')+1$ and $\ell(su's)=\ell(u').$

\noindent {\bf Case C:} $\tail(Q)=su'$ with $\ell(su')=\ell(u')+1$ and $\ell(su's)=\ell(u').$

\noindent {\bf Case D:} $\tail(Q)=su'$ with $\ell(su')=\ell(u')+1$ and $\ell(su's)=\ell(u')+2.$

\noindent {\bf Case E:} $\tail(Q)=su's$ with $\ell(su's)=\ell(u')+2.$

\medskip
We shall examine in all cases the structure of $Q$ and the cohomology of $X(Q).$

\medskip

\noindent {\bf  Case A:} $\tail(Q)=u'$ with $u'<w'$ and $\ell(su's)=\ell(u')+2.$

Let $\dim(Q)=d$, so that $\# Q=2^d.$ We consider the subintervals $I(u,w),I(su',sw')$, $I(u's,w's),I(u',w').$ Since $Q$ is a square each of them
is a square as well and has consequently $2^{d-2}$ elements. We shall see that the union of them is $Q.$
Indeed, let $v\in Q.$

\noindent Case 1) If $v \leq w'$ then $v\in I(u',w').$

\noindent  Case 2) Let $v \leq sw'$ and $v \not\leq w'.$ Thus we may write $v=sv'$ with $v' \leq w'.$ As $\ell(su')=\ell(u')+1$ we see that
by considering reduced decompositions that we must have $v\geq su'.$ Thus $w\in I(su',sw').$

\noindent Case 3) Let $v\leq w's$ and $v \not\leq w'.$ This case is symmetric to Case 2, hence $w\in I(u's,w's).$

\noindent Case 4) Let $v\leq sw's$ and $v\not \leq sw'$ and $v\not\leq w's.$  Then as in Case 2 we argue that $v=sv's$ with $v' \geq u'$.
Hence $v\in I(su's,sw's).$

As $4\cdot 2^{d-2}=2^d.$ we see that the pairwise intersection of the above 4 subsquares is empty and that
$I(su',sw')$ (resp. $I(u,w),I(u's,w's))$ is induced by $I(u',w')$ by multiplying with $s$ from the left (resp. conjugating with $s$, multiplying with $s$
from the right). Hence $Q$ is a union of special squares (cf. Definition \ref{special_square}) and the
cohomology is consequently given by
$$H^i_c(X(Q))=H^i_c(X(Q(u',w's))) \oplus H^{i-2}_c(X(Q(u',w's)))(-1)$$
which is known by induction on the length.

\medskip

\noindent{\bf Case B}. $\tail(Q)=u'$ with $u'< w'$ and $\ell(su's)=\ell(u').$ It follows that $su'=u's$ by Lemma \ref{lemma_DL}. We claim that this case does not appear.
The case where $\dim (Q)=2$ does not occur. Let $\dim (Q)=3.$ Then $Q$ must have the shape

$I(w,u'): \;\,\; \begin{array}{ccccc}
 & & w & & \\ sw' & & sv's & & w's  \\ sv'& & w' & & v's  \\  & & u' & &
\end{array}$

\noindent for some $v'\geq u'$ since $\ell(su')=\ell(u')+1$ (consider a reduced decomposition of $v'$). Hence $v'=u'.$ But then $sv's=su's$,
a contradiction as $\ell(su's)=\ell(u').$

If $\dim Q >3$ then we argue by induction. Indeed in $Q$ there must be a subsquare of dimension $\dim(Q)-1$ with ${\rm head}(Q)=sv's$ and
${\rm tail}(Q)=u'.$
By induction this is not possible.

\medskip

\noindent{\bf Case C:} $\tail(Q)=su'$ with $\ell(us') > \ell(u')< \ell(su').$

We shall see that this case behaves very rigid. More precisely, we shall see that $Q$ is paved by squares of type ${\rm Case}$ 4.
Here we make usage of the following statement.
\begin{lemma}\label{not_possible}
 Let $w'\in W$, $s\in S$ with $\ell(sw's)=\ell(w')+2.$ Then there is no $v'\leq w'$ with $\ell(v')=\ell(w')-1$ and such that $v'=su'=u's$
for some $u'\leq v'.$
\end{lemma}

\begin{proof}
Let $w'=s_1\cdots s_r$  be a reduced decomposition. Suppose that there exists such a $v'=su'=u's$ as above. Then there is some index
$1\leq i \leq r$ with $su'=s_1\cdots \hat{s_i} \cdots s_r.$  On the other hand, since $\ell(sv')<\ell(v')$ there exits by the {\it Exchange Lemma}
some integer $1\leq m \leq r$ with $ss_1\cdots s_{m-1} =s_1\cdots s_m$ (with $s_i$ omitted depending on whether $i<m$ or $i>m$).
If $m<i$, then $w'=s_1\cdots s_r = ss_1\cdots \hat{s_m}\cdots s_r$ a contradiction to the assumption that $\ell(sw')=\ell(w')+1.$
If $m>i$, then $ss_1\cdots \hat{s_i}\cdots s_{m-1}= s_1\cdots \hat{s_i}\cdots s_{m}.$
But $su'=s\cdot s_1\cdots \hat{s_i}\cdots \hat{s_m}\cdots s_r=s_1\cdots \hat{s_i}\cdots \hat{s_m}\cdots s_r\cdot s$ as $su'=u's.$
Hence we deduce that $s_{m+1}\cdots s_r \cdot s=s_m\cdot s_{m+1}\cdots s_r.$ Again by plugging this expression into the reduced decomposition for $w'$,
 we obtain a contradiction to the assumption that $\ell(w's)=\ell(w')+1.$
\end{proof}

We start with the case of a square. Here it is as in Case 4 before. Consider now a hypersquare $Q$ of dimension 3. Thus it
must have the shape

$Q:\;\,\; \begin{array}{ccccc}
 & & w & & \\ sw' & & sv's & & w's  \\ ? & & ? & & ?  \\  & & su' & &
\end{array}$

\noindent with $su'=u's$, $v'<w'$ and $u'< w'$ for certain elements $?\in W.$ As $\ell(su')=\ell(u')+1$ we deduce that $ v' \geq u'$. Hence we can make the structure
of $Q$ more precise, i.e.

$Q:\;\,\; \begin{array}{ccccc}
 & & w & & \\ sw' & & sv's & & w's  \\ sv' & & ? & & v's  \\  & & su' & &
\end{array}$

\noindent As one verifies there are a priori for $?\in W$ only 2 possibilities : $?\in \{w',sz'\}$ with $sz'=z's$ and $z'\in I(u',w').$
But by Lemma \ref{not_possible} we must have $?=sz'.$ Hence $Q$ is the union of two squares of the shape as in Case 4.

Let $\dim(Q)=d>3$. The square $Q$ begins with the following elements

$Q:\;\,\; \begin{array}{ccccc}
 & & w & & \\ sw' & & sv_1's, \cdots ,sv_{d-2}'s  & & w's  \\ \vdots & & \vdots & & \vdots  \\  & & su' & &
\end{array}$

\noindent By induction on the size of $Q$ we know that all subsquares of $Q$ of the shape $Q(su',sv_i's)$ are union of squares of the
the desired shape. Now one verifies that square $Q$ ends as follows

$Q:\;\,\; \begin{array}{ccccc}
 & & w & &  \\ \vdots & & \vdots & & \vdots \\ sx' & & sy_1', \cdots ,sy_{d-2}'  & & x's  \\  & & su' & &
\end{array}$

\noindent with $sy_i=y_is, \; i=1,\ldots,d.$ There must be some $y_i$ with $sy'\leq w$ and again by induction the statement follows.

Next we turn to the cohomology with respect to  these hypersquares. We start again with the 2-dimensional case.
We fix reduced decompositions of $w'$ and $u'$. We consider the hypersquare $Q^{{F^+}}(u',w)\subset F^+$

$Q^{F^+}(u',w):\;\,\; \begin{array}{ccccc}
 & & w & & \\ & \nearrow & \uparrow & \nwarrow & \\ sw' & & su's & & sw'  \\ \uparrow & \nearrow \nwarrow & & \nearrow \nwarrow & \uparrow
\\ su'& & w' & & u's  \\ & \nwarrow & \uparrow & \nearrow & \\  & & u' & &
\end{array}$

\noindent and the interval $I(u',w)$ in $W$. A case by case study together  with the fact (which follows by Lemma \ref{lemma_DL})  that apart from $u'$ there is no
$z'\prec w'$ with $\ell(z')=\ell(w')-1$ and $\gamma(z')=\gamma(u')$ it is seen  that the preimage of $I(u',w)$
under the proper map $\pi:X^{\ell(w)+1} \to X$ is just $X(Q^{F^+}(u',w)).$ Hence we get a
proper surjective  map
$$\pi: X(Q^{F^+}(u',w)) \to X(I(u',w)). $$
The hypersquare $Q$ is an open subset of the interval $I(u',w)$ so that $U:=X(Q)$ is an open subvariety of
$X(I(u',w)).$ The closed complement is given by $Y:=X(I(u',w))\setminus X(Q).$
We consider their preimages $U':=\pi^{-1}(X(Q))$ and $Y':=\pi^{-1}(Y)$ in $X^{\ell(w)+1}.$
We get a commutative diagram of long exact cohomology sequences
\begin{equation}\label{lonexact_diagram}
 \begin{array}{ccccccc}
\cdots \to & H^i_c(U') & \to &  H_c^{i}(X(Q^{F^+}(u',w))) & \to & H^i_c(Y') &  \to \cdots \\
 & \uparrow & & \uparrow & & \uparrow &  \\
\cdots \to &  H_c^{i}(X(Q)) & \to &  H_c^{i}(X(I(u',w)) &  \to & H^{i}_c(Y) &  \to \cdots \\
\end{array}
\end{equation}

We claim that we can recover the cohomology of $X(Q)$ by this diagram. In fact, the cohomology of $X(Q^{F^+}(u',w))$ is known by its
particular structure since it is the union of special squares, i.e.
$$H^i_c(X(Q^{F^+}(u',w)))=H^i_c(X(Q^{F^+}(u',w's))) \oplus  H^{i-2}_c(X(Q^{F^+}(u',w's)))(-1).$$
Further $Y=X(w') \cup X(u')$ whereas $Y'=X_1^{F^+}(su's) \cup Y^{F^+}$ with the obvious meaning for $Y^{F^+}$ and $X_1^{F^+}(su's)$ is the closed
subset of $X^{F^+}(su's)$ in Remark \ref{same_is_true_for_F}.  Then $$H^i_c(Y')=H^i_c(Y) \oplus H^{i-2}_c(X(u'))(-1)$$
and $$H^i_c(U')=H^i_c(U) \oplus H^{i-2}_c(X(u's))(-1).$$
Hence the restriction  map
$H^i_c(X(Q^{F^+}(u',w))) \to H^i_c(Y')$ is induced by the sum of the maps
$$H^i_c(X(Q(u',w's))) \to H^i_c(Y) \mbox{ and } H^{i-2}_c(X(Q(u',w's)))(-1) \to H^{i-2}_c(X(u'))(-1).$$
The latter one factories over the representation $H^{i-2}_c(X(w')\cup X(v'))(-1)$. Hence both maps are known by induction. Thus we deduce the cohomology of $U'$.
By factoring out the second summand in $H^i_c(U')=H^i_c(U) \oplus H^{i-2}_c(X(u's))(-1)$ we get the cohomology of $X(Q).$
Furthermore, the boundary map  $H^{i-2}_c(X(u')) \to H^{i-1}_c(X(u's))$ which appears in the boundary map $H_c^{i}(Y') \to H^{i+1}_c(U')$
 is known, as well.

If $\dim(Q)>2$ then one verifies that the above description generalizes to the higher dimensional setting. In particular, we get
 $$H^i_c(Y')=H^i_c(Y) \oplus H^{i-2}_c(\bigcup_{sv'\in Q \atop sv'=v's}X(v'))(-1)$$
and
$$H^i_c(U')=H^i_c(U) \oplus H^{i-2}_c(\bigcup_{sv'\in Q \atop sv'=v's}X(sv'))(-1).$$
The second summand is known by induction as the set $\{sv'\in Q \mid sv'=v's\}$ forms a subsquare in $W$.

\medskip

For the remaining two cases (D and E),  we introduce the following partial order on the sets of squares of type Case 1 - 5 via
the following pre-order  diagram.

$$\begin{array}{ccccc}
 & & {\rm Case\; 3} & & \\
 &  & \downarrow & & \\
 & & {\rm Case\; 5} & &  \\
& \swarrow  & &  \searrow & \\
{\rm Case\; 4} & & & &  {\rm Case\; 2}\\
& & & & \downarrow \\
 & & & &  {\rm Case\; 1}
 \end{array}$$

\noindent Here the arrow  ${\rm Case}$ i $\to$ {\rm Case} j means that Case j $<$ Case i.

\medskip

\noindent{\bf Case D:} $\tail(Q)=su'$ with $u'<w'$ and $\ell(su's)=\ell(u')+2.$

We start with the case of a square. Here it has the shape as in Case 2 before. Consider now a hypersquare $Q$ of dimension 3. Thus it
must have the shape

$Q:\;\,\; \begin{array}{ccccc}
 & & w & & \\ sw' & & sv's & & ?  \\ sv' & & ? & & su's  \\  & & su' & &
\end{array}$

\noindent for certain elements $?\in W$ where $v'\geq u'$ as $\ell(u's)=\ell(u')+1$. There are two possibilities for completing $Q$.
If $w'\geq su'$, then we get

$Q:\;\,\; \begin{array}{ccccc}
 & & w & & \\ sw' & & sv's & & w's  \\ sv' & & w' & & su's  \\  & & su' & &
\end{array}.$

On the other hand, if $w'\not\geq su'$, then we get

$Q:\;\,\; \begin{array}{ccccc}
 & & w & & \\ sw' & & sv's & & sz's  \\ sv' & & sz' & & su's  \\  & & su' & &
\end{array}.$

Hence if we write $Q=Q(su',sv's) \stackrel{\cdot}{\cup} Q'$ then $Q'$ is a specialization of $Q(su',sv's)$, i.e. $Q' \leq Q(su',sv's).$
For  $\dim(Q)=d>3,$ we claim that $Q$ is paved by $3$-dimensional hypersquares of this kind. More precisely, if $Q'\subset Q$ is a
3-dimensional subsquare which we write as $Q=Q_1 \stackrel{\cdot}{\cup} Q_2$  where  $\head(Q_i)=sv_is$ with $v_2 \prec v_1,$ then $Q_1 \leq Q_2.$
Indeed, the square $Q$ begins with the following elements

$Q:\;\,\; \begin{array}{ccccc}
 & & w & & \\ sw' & & sv_1's, \cdots ,sv_{d-2}'s  & & ?  \\ \vdots & & \vdots & & \vdots  \\  & & su' & &
\end{array}$

\noindent with $?\in \{w's,sv'_{d-1}s\}.$
By induction on the size of $Q$ we know that all subsquares of $Q$ of the shape $Q(su',sv_i's)$ are union of squares of the
the desired shape. But the union over all these squares exhaust all apart from
$\{w,sw',?,?\}.$ Indeed the number of elements in this union is $2^{d-1}+2^{d-2}+\cdots + 2^2 =2^d-4.$
Again if $w'\geq su'$, then  $\{w,sw',?,?\}=Q_w$ and the claim follows. On the other hand if $w' \not\geq su'$, then we can even says that $Q$
is paved by squares of type Case 2, since there cannot be an element $v'\leq w'$ with $v'\in Q$ and  $v'\geq su'.$

Let's determine the cohomology of $X(Q)$. If $\dim (Q)=2$, then we get $$H^i_c(Q(su',w))=H^{i-2}_c(Q(u',w's))(-1).$$
If $\dim(Q)=3$ and $w'\not\geq su'$ then again - by the observation above -  we get $H^i_c(Q(su',w))=H^{i-2}_c(Q(u',w's))(-1).$
Consider the other possibility of a cube. Here we consider as in the previous case the hypersquare
$\hat{Q}:=Q^{F^+}(u',w)$ in ${F^+}$ and the interval $I(u',w)$ in $W$.

$\hat{Q}:\;\,\; \begin{array}{ccccccccccc}
 & & & & & w & & & & &  \\
& & sw' & & sv's & & w's & & sz's & &  \\
sv' & & w' & & su's & &  v's & & sz' &  & z's \\
& & su' & & v' & & u's & & z' & &  \\
& & &  & & u' & & & & &
\end{array}$

\noindent with $z'=su'.$ The map $\pi:X^{\ell(w)+1} \to X$ induces surjective map
$$\pi: X(Q^{F^+}(u',w)) \to X(I(u',w))$$
which is even proper although $X(Q^{F^+}(u',w))$ might be strictly contained in the subset  $\pi^{-1}(X(I(u',w)))$.
(The reason is that for any closed subset $A\subset \pi^{-1}(X(I(u',w)))$ the identity $\pi(A)=\pi(A\cap X(Q^{F^+}(u',w)))$ holds).
We set  $U:=X(Q) \subset X(I(u',w))$ and $Y:=X(I(u',w))\setminus X(Q).$
We consider their preimages $U':=\pi^{-1}(X(Q))$ and $Y':=\pi^{-1}(Y)$ in $X(Q^{F^+}(u',w)).$
Again we claim that we can recover the cohomology of $X(Q)$ by the diagram \ref{lonexact_diagram}. The reasoning is similar to Case C.
In fact, we have
$$H^i_c(X(Q^{F^+}(u',w)))=H^i_c(X(Q^{F^+}(u',w's))) \oplus  H^{i-2}_c(X(Q^{F^+}(u',w's)))(-1)$$
since $Q^{F^+}(u',w)$ is paved by special squares.
Further
$$Y=X(u') \cup X(v') \cup X(u's) \cup X(v's)$$ and
$$Y'=Y^{F^+} \cup X_1^{F^+}(sz') \cup X^{F^+}(sz's)$$ whereas
$$U'=U^{F^+} \cup X^{F^+}(z')\cup  X^{F^+}(z's) \cup X_2^{F^+}(sz') \cup X_2^{F^+}(sz's).$$
Here $X_2^{F^+}(su's)$ is the open subset of $X^{F^+}(su's)$ in Remark \ref{same_is_true_for_F}.
Then $$H^i_c(Y')=H^i(Y) \oplus H^{i-2}_c(X(u')\cup X(u's))(-1)$$
and $$H^i_c(U')=H^i_c(U) \oplus H^{i-2}_c(X(su's)\cup X(su'))(-1).$$
The restriction map
$H^i_c(X(Q^{F^+}(u',w))) \to H^i_c(Y')$ is given by the sum of the maps
$$H^i_c(X(Q(u',w's))) \to H^i_c(Y) \mbox{ and } H^{i-2}_c(X(Q(u',w's)))(-1) \to H^{i-2}_c(X(u's)\cup X(u'))(-1).$$
Again both maps are known by induction. Thus we deduce the cohomology of $U'$.
By factoring out the second summand in $H^i_c(U')=H^i_c(U) \oplus H^{i-2}_c(X(su's)\cup X(su'))(-1)$ we get the cohomology of $X(Q).$

The higher dimensional case is treated similar as in Case C.

\medskip

\noindent{\bf Case E:} $\tail(Q)=su's$ with $u'<w'$ and $\ell(su's)=\ell(u')+2.$

We start with the case of a square. Here it has the shape as in Case 3 or Case 5 before.
Consider now a hypersquare $Q$ of dimension 3. Thus if its lower subsquare is as in Case 5
must have the shape

$Q:\;\,\; \begin{array}{ccccc}
 & & w & & \\ ? & & sv's & & ?  \\ sv' & & ? & & sz's  \\  & & su's & &
\end{array}$

\noindent for certain elements $?\in W.$ There are three possibilities for completing $Q$.
If $w'\geq su'$, then we get

$Q:\;\,\; \begin{array}{ccccc}
 & & w & & \\ sw' & & sv's & & w's  \\ sv' & & sy' & & sz's  \\  & & su's & &
\end{array}$

\noindent with $sy'=y's.$ On the other hand, if $w'\not\geq su'$, then we get

$Q:\;\,\; \begin{array}{ccccc}
 & & w & & \\ sw' & & sv's & & sy's  \\ sv' & & ? & & sz's  \\  & & su's & &
\end{array}$

\noindent with $?=sy'$ or $?=sx's$ for some $x'<v'.$
Hence if we write $Q=Q(su's,sv's) \stackrel{\cdot}{\cup} Q'$ then $Q'$ is a specialization of $Q(su's,sv's)$, i.e. $Q' \leq Q(su',sv's).$
For  $\dim(Q)=d>3,$ one proves as in Case D that is paved by $3$-dimensional hypersquares of this kind.

Consider now a hypersquare $Q$ of dimension 3 such that its lower subsquare is as in Case 3,

$Q:\;\,\; \begin{array}{ccccc}
 & & w & & \\ ? & & sv's & & ?  \\ sx's & & ? & & sy's  \\  & & su's & &
\end{array}$

\noindent for certain elements $?\in W.$ This is the most generic case in the sense that the upper square in $Q$ can be arbitrary, i.e.
if we write $Q=Q(su's,sv's) \stackrel{\cdot}{\cup} Q'$ then $Q'$ is a any specialization of $Q(su',sv's).$

Now we consider the cohomology and discuss only the case of  a square $Q$. The higher dimensional cases are treated as before,
see also Remark \ref{rk_pi} for a general approach. So let $Q$ be a square as in Case 5 (Case 3 has been already explained)
\smallskip

$
Q:\,\,\, \begin{array}{ccccc}
 & & w & & \\ & \nearrow & & \nwarrow & \\ sw' & & & & sv's  \\ & \nwarrow & & \nearrow & \\ & & su's & &
\end{array}
$
\medskip

\noindent where $
\,\,\, \begin{array}{ccccc}
 & & w' & & \\ & \nearrow & & \nwarrow & \\ y' & & & & v'  \\ & \nwarrow & & \nearrow & \\ & & u' & &
\end{array}
$

\noindent is a square with $y'=u's.$ We consider the extended interval $I(su',w) \subset W$ resp. the cube $\hat{Q}:=Q^{F^+}(su',w)\subset F^+$.
\smallskip

$\hat{Q}:\;\,\; \begin{array}{ccccc}
 & & w & & \\ & \nearrow & \uparrow & \nwarrow & \\ sw' & & sv's & & sy's  \\ \uparrow & \nearrow \nwarrow & & \nearrow \nwarrow & \uparrow
\\ sv'& & su's & & sy'  \\ & \nwarrow & \uparrow & \nearrow & \\  & & su' & &
\end{array}
.$
\smallskip

\noindent Here $$Y=X(su') \cup X(sv').$$

$$Y'=Y^{F^+}  \cup X_1^{F^+}(sy's).$$

$$U= X(w) \cup X(sw') \cup X(sv's) \cup X(su's).$$

$$U'= U^{F^+} \cup X^{F^+}(sy') \cup X^{F^+}_2(sy's).$$

\noindent Hence we get
$$H^i_c(Y')=H^i_c(Y) \oplus H^{i-2}_c(X(u's))(-1)$$ and
$$H^i_c(U')=H^i_c(U) \oplus H^{i-2}_c(X(su's))(-1).$$
The restriction map $H^i_c(X(Q^{F^+}(su',w))) \to H^i_c(Y')$ is given as follows.
First note that $H^i_c(X(\hat{Q}))= H^{i-2}_c(X(s\backslash Q))(-1),$
where $s\backslash Q$ is the cube such that $s \cdot (s\backslash Q) =Q.$
With respect to the summand $H^i_c(Y)$ we know the map factorizes over $H^i_c(X(Q^{F^+}(su',sw'))) \to H^i_c(X(sv') \cup X(su'))$.
As for the summand  $H^{i-2}_c(X(su'))(-1)$ the necessary  information follows from the identity  $H^i_c(X(\hat{Q}))= H^{i-2}_c(X(s\backslash Q))(-1).$
All maps are known. On the other hand we know by induction the boundary map $H^{i-2}_c(X(su')) \to H^{i-1}_c(X(su's)).$
Again we deduce the cohomology of $U'$ and by factoring out the summand  $H^{i-2}_c(X(su's))(-1)$ we get the cohomology of
$U$.

Thus we have examined all cases. In remains to say that the start of induction is the situation where $\head(Q)$ is minimal in its conjugacy class.
This case can be handled explicitly using successively Proposition \ref{boundary_Coxeter}.

\begin{rk}\label{rk_pi}
Let $I=I(u,w)\subset W$ be any interval. The map $\pi:X^{\ell(w)+1} \to X$ induces a
proper map
$$\pi: \pi^{-1}(X(I)) \to X(I). $$
The following lines gives a description of the preimage $Z=\pi^{-1}(X(I)) \subset X^{\ell(w)+1}$. Let $v\in Q^{{F^+}}(1,w).$

\noindent 1. Case. $\ell(\gamma(v))=\ell(v).$

Subcase a) $\gamma(v) \not\geq u.$ In this case $X^{{F^+}}(v)\cap Z =\emptyset$.

Subcase b) $\gamma(v) \geq u.$ In this case $X^{{F^+}}(v) \subset Z$ and the restriction of $\pi$ to $X^{{F^+}}(v)$
induces an isomorphism $X^{{F^+}}(v) \stackrel{\sim}{\to} X(\gamma(v)).$

\noindent 2. Case. $\ell(\gamma(v))<\ell(v).$ By Lemma \ref{doppelts} we may suppose that $v=v_1\cdot t\cdot t\cdot v_2.$
Thus we may write $X^{{F^+}}(v)=X^{{F^+}}_1(v) \bigcup X^{{F^+}}_2(v)$ where $X^{{F^+}}_1(v)$ is closed and $X^{{F^+}}_2(v)$ is open.
We have $\mA^1$-bundles
$X_1^{F^+}(v) \to X^{F^+}(v_1v_2)$ and $X^{F^+}_2(v) \cup X^{F^+}(v_1tv_2) \to X^{F^+}(v_1tv_2).$ The
map $\pi|_{X^{F^+}} :X^{F^+}(v) \to X$ factorizes through
$X^{F^+}(v_1tv_2) \cup X^{F^+}(v_1v_2).$ Hence we have reduced the question to elements of lower length.

Suppose additionally that $w=sw's.$ Then $Q^{F^+}(1,w)$ is paved by special squares. Then it is possible
to say what the image of such a special square $Q_v=\{v,sv',v's,v'\} $ under the map $\pi$ is. But we do not carry out this since
there are to many cases.

\end{rk}

\vspace{0.5cm}

\section{Appendix B}

Here we give summarizing tables of the cohomology of DL-varieties with respect to Weyl group elements  of full support 
in  $\GL_3$ and $\GL_4$. We  list only representatives of cyclic shift classes.

\begin{center}
\begin{tabular}{c|c}
$\GL_3$ & $H^\ast_c(X(w))$\\ \hline \\
$(1,2,3)$ & $j_{(1,1,1)}[-2] \oplus j_{(2,1)}(-1)[-3] \oplus j_{(3)}(-2)[-4]$ \\ \\ \hline \\
$(1,3)$ & $j_{(1,1,1)}[-3] \oplus j_{(3)}(-3)[-6]$ \\ &
\end{tabular}
\end{center}
\begin{center}
\begin{tabular}{c|c}
$\GL_4$ & $H^\ast_c(X(w))$\\ \hline \\
$(1,2,3,4)$ & $j_{(1,1,1,1)}[-3] \oplus j_{(2,1,1)}(-1)[-4] \oplus j_{(3,1)}(-2)[-5] \oplus j_{(4)}(-3)[-6] $  \\ \\ \hline \\
$(1,2,4)$ & $j_{(1,1,1,1)}[-4] \oplus j_{(2,2)}(-2)[-5] \oplus j_{(4)}(-4)[-8]$ \\ \\ \hline \\
$(1,3)(2,4)$ & $j_{(1,1,1,1)}[-4] \oplus j_{(2,2)}(-1)[-4] \oplus j_{(2,1,1)}(-2)[-5] \oplus j_{(3,1)}(-2)[-5] \oplus $ \\ \\ &  $j_{(2,2)}(-3)[-6] \oplus j_{(4)}(-4)[-8]$ \\ \\ \hline \\
$(1,3,2,4)$ & $j_{
(1,1,1,1)}[-5] \oplus j_{(2,2)}(-2)[-6] \oplus j_{(2,1,1)}(-2)[-6] \oplus j_{(2,2)}(-3)[-7] \oplus $ \\ \\ &  $j_{(3,1)}(-3)[-7] \oplus j_{(4)}(-5)[-10]$ \\ \\ \hline \\
$(1,4)$ & $j_{(1,1,1,1)}[-5] \oplus j_{(2,1,1)}(-1)[-5] \oplus j_{(2,2)}(-2)[-6] \oplus j_{(2,2)}(-3)[-7] \oplus $ \\  \\&  $j_{(3,1)}(-4)[-8] \oplus j_{(4)}(-5)[-10]$ \\ \\ \hline \\
$(1,4)(2,3)$ & $j_{(1,1,1,1)}[-6] \oplus j_{(2,1,1)}(-2)[-7] \oplus j_{(2,2)}(-3)[-8]^2  \oplus j_{(3,1)}(-4)[-9] \oplus j_{(4)}(-6)[-12].$ 
\end{tabular}
\end{center}


\vspace{1cm}



\begin{thebibliography}{ABC2'}
\bibitem[BGG]{BGG} I.N. Bernstein, I.M. Gelʹfand, S.I. Gelʹfand, {\it  Differential operators on the base affine space and a study of g-modules}.
 Lie groups and their representations (Proc. Summer School, Bolyai J\'anos Math. Soc., Budapest, 1971), pp. 21 -– 64, Halsted, New York, 1975.
\bibitem[BL]{BL} S. Billey, V. Lakshmibai, {\it Singular loci of Schubert varieties. Progress in Mathematics}, {\bf 182}. Birkh\"auser Boston,
Inc., Boston, MA, 2000.
\bibitem[Bl]{Bl} S. Bloch, {\it Algebraic cycles and higher K-theory}. Adv. in Math. {\bf 61} (1986), no. 3, 267--304.
\bibitem[BM]{BM} M. Brou\'e, J. Michel, {\it Sur certains \'el\'ements r\'eguliers des groupes de Weyl et les vari\'et\'es de Deligne-Lusztig
associ\'ees}, Finite reductive groups (Luminy, 1994), 73--139, Progr. Math., {\bf 141}, Birkh\"auser Boston, Boston, MA, 1997.
\bibitem[De]{De} P. Deligne, {\it  Action du groupe des tresses sur une catégorie}, Invent. Math. {\bf 128} (1997), no. 1,  159--175.
\bibitem[DL]{DL} P. Deligne, G. Lusztig, {\it Representations of reductive groups over finite fields},  Ann. of Math. (2)  {\bf 103}  (1976), no. 1, 103--161.
\bibitem[DM]{DM}  F. Digne, J. Michel, {\it Endomorphisms of Deligne-Lusztig varieties}, Nagoya Math. Journal {\bf 183} (2006), 35--103.
\bibitem[DMR]{DMR} F. Digne, J. Michel, R. Rouquier, {\it Cohomologie des vari\'et\'es de Deligne-Lusztig}, Adv. Math. {\bf 209}  (2007), No. 2, 749-822.
\bibitem[DOR]{DOR} J.-F. Dat, S. Orlik, M. Rapoport, {\it Period domains over finite and p-adic fields},  Cambridge Tracts in Mathematics (No. {\bf 183}), Cambridge
University Press, 2010.
\bibitem[Dr]{Dr} V.G. Drinfelʹd, {\it Coverings of p-adic symmetric domains}, Funkcional. Anal. i Priložen. {\bf 10} (1976), no. 2, 29--40.
\bibitem[Du]{Du} O. Dudas, {\it Cohomology of Deligne-Lusztig varieties for short-length regular elements in exceptional groups}, J. Algebra, {\bf 392}, (2013), 276--298.
\bibitem[FH]{FH} W. Fulton, J. Harris, {\it Representation theory. A first course}. Graduate Texts in Mathematics, {\bf 129} Readings in Mathematics. Springer-Verlag, New York, 1991.
\bibitem[Fu]{Fu} W. Fulton, {\it Intersection theory}. Second edition. Ergebnisse der Mathematik und ihrer Grenzgebiete. 3. Folge. A Series of Modern Surveys in Mathematics [Results in Mathematics and Related Areas. 3rd Series. A Series of Modern Surveys in Mathematics], 2. Springer-Verlag, Berlin, 1998.
\bibitem[Ge]{Ge}  A. Genestier, {\it  Espaces sym\'etriques de Drinfeld}, Ast\'erisque No. {\bf 234} (1996).
\bibitem[GK]{GK}  E. Grosse-Kl\"onne, {\it Integral structures in the p-adic holomorphic discrete series}. Represent. Theory {\bf 9} (2005), 354--384.
\bibitem[GP]{GP} M. Geck, G. Pfeiffer, {\it On the irreducible characters of Hecke algebras},  Adv. Math.  {\bf 102}   (1993),  no. 1, 79--94.
\bibitem[GP']{GP'} M. Geck, G. Pfeiffer, {\it Characters of finite Coxeter groups and Iwahori-Hecke algebras},  London Mathematical Society Monographs. New Series, {\bf 21}. The Clarendon Press, Oxford University Press, New York,  2000.
\bibitem[GKP]{GKP} M. Geck, S. Kim, G. Pfeiffer, {\it  Minimal length elements in twisted conjugacy classes of finite Coxeter groups},  J. Algebra  {\bf 229} (2000),  no. 2, 570--600.
\bibitem[HV]{HV} M. Hazewinkel, T. Vorst, {\it On the Snapper \slash \, Liebler-Vitale \slash \, Lam Theorem on permutation representations of the symmetric group}, Journal of Pure and Applied Algebra {\bf 23} (1982), 29--32.
\bibitem[Ho]{Ho} R. Howe, {\it Harish-Chandra homomorphisms for p-adic groups}. With the collaboration of Allen Moy. CBMS Regional Conference Series in Mathematics, 59. Published for the Conference Board of the Mathematical Sciences, Washington, DC; by the American Mathematical Society, Providence, RI, 1985.
\bibitem[Hu]{Hu} R. Huber, {\it \'Etale cohomology of rigid analytic varieties and adic spaces}. Aspects of Mathematics, E30. Friedr. Vieweg und Sohn, Braunschweig, 1996.
\bibitem[I]{I} T. Ito, {\it Weight-monodromy conjecture for $p$-adically uniformized varieties}, Invent. Math. {\bf 159} (2005), no. 3, 607--656.
\bibitem[Ku]{Ku} S. Kumar, {\it  Kac-Moody groups, their flag varieties and representation theory}. Progress in Mathematics, {\bf 204}. Birkhäuser Boston, Inc., Boston, MA, 2002.
\bibitem[Le]{Le} G. I. Lehrer, {\it The spherical building and regular semisimple elements}, Bull. Austral. Math. Soc. {\bf 27} (1983), no. 3, 361-379.
\bibitem[L]{L} G. Lusztig,  {\it Representations of finite Chevalley groups}, Expository lectures from the CBMS Regional Conference held at Madison, Wis., August 8--12, 1977. CBMS Regional Conference Series in Mathematics, {\bf 39}.
American Mathematical Society, Providence, R.I.,  1978.
\bibitem[L2]{L2} G. Lusztig, {\it Coxeter orbits and eigenspaces of Frobenius}, Invent. Math. {\bf 38} (1976/77), no. 2, 101--159.
\bibitem[L3]{L3} G. Lusztig,  {\it Characters of reductive groups over a finite field}. Annals of Mathematics Studies, 107. Princeton University Press, Princeton, NJ, 1984.
\bibitem[L4]{L4} G. Lusztig, {\it Homology bases arising from reductive groups over a finite field}, Carter, R. W. (ed.) et al., 
Algebraic groups and their representations. Proceedings of the NATO Advanced Study Institute on modular representations 
and subgroup structure of algebraic groups and related finite groups, Cambridge, UK, June 23--July 4, 1997. 
Dordrecht: Kluwer Academic Publishers. NATO ASI Ser., Ser. C, Math. Phys. Sci. {\bf 517}, (1998), 53-72.
\bibitem[LV]{LV} R.A. Liebler, M.R. Vitale, {\it Ordering the partition characters of the symmetric group}, J. Algebra {\bf 25} (1973), 487--489.
\bibitem[M]{M} J. Milne, {\it Lectures on etale cohomology}, http://www.jmilne.org/math/.
\bibitem[O]{O} S. Orlik, {\it  Kohomologie von Periodenbereichen \"uber endlichen K\"orpern}, J. Reine Angew. Math. {\bf 528} (2000), 201--233.
\bibitem[RTW]{RTW} B. R\'emy, A. Thuillier, A. Werner, {\it Automorphisms of Drinfeld half-spaces over a finite field}, Compositio Math. {\bf 149} (2013), 1211-1224.
\bibitem[SGA5]{SGA5} S\'eminaire de G\'eométrie Alg\'ebrique du Bois Marie - 1965-66 - Cohomologie l-adique et Fonctions L - (SGA 5). Lecture notes in mathematics. 589. Berlin; New York: Springer-Verlag.
\bibitem[SS]{SS} P. Schneider, U. Stuhler, {\it The cohomology of $p$-adic symmetric spaces},  Invent. Math. {\bf 105} (1991), 47--122.
\end{thebibliography}
\end{document}